\documentclass[11pt,twoside]{amsart}

\usepackage{amsmath,amsthm,amssymb,amsfonts,amscd}
\usepackage{url}
\usepackage{color}
\usepackage{setspace}
\usepackage{enumerate}
\usepackage{mathabx}

\usepackage{epsfig}
\usepackage{graphicx}
\usepackage{mathrsfs}
\usepackage{pinlabel}
\usepackage[outercaption]{sidecap}

\definecolor{light-gray}{gray}{0.60}

\theoremstyle{plain}
\newtheorem{theorem}{Theorem}
\newtheorem{proposition}[theorem]{Proposition}

\newtheorem{corollary}[theorem]{Corollary}

\newtheorem{lemma}[theorem]{Lemma}

\newtheorem{fact}[theorem]{Fact}

\newtheoremstyle{theoremwithref}{}{}{\itshape}{}{\bfseries}{.}{.5em}{#1 #2 #3}
\theoremstyle{theoremwithref}

\theoremstyle{definition}
\newtheorem{definition}[theorem]{Definition}
\newtheorem{example}[theorem]{Example}
\newtheorem{remark}[theorem]{Remark}

\numberwithin{theorem}{section}
\numberwithin{equation}{section}

\newcommand{\D}{\mathrm{d}}
\newcommand{\Id}{\mathrm{Id}}

\newcommand{\ZZ}{\mathbb{Z}}

\newcommand{\RR}{\mathbb{R}}
\newcommand{\EE}{\mathbb{E}}
\newcommand{\CC}{\mathbb{C}}
\newcommand{\HH}{\mathbb{H}}
\newcommand{\PP}{\mathbb{P}}
\newcommand{\RP}{\mathbb{RP}}

\newcommand{\SL}{\mathsf{SL}}
\newcommand{\GL}{\mathsf{GL}}
\newcommand{\SO}{\mathsf{SO}}
\newcommand{\OO}{\mathsf{O}}
\newcommand{\PO}{\mathsf{PO}}
\newcommand{\PSO}{\mathsf{PSO}}
\newcommand{\PSL}{\mathsf{PSL}}
\newcommand{\PGL}{\mathsf{PGL}}
\newcommand{\Sp}{\mathsf{Sp}}

\newcommand{\Aff}{\mathsf{Aff}}
\newcommand{\Isom}{\mathsf{Isom}}
\newcommand{\axisH}{\mathscr{A}}
\newcommand{\axisR}{\mathcal{A}}
\newcommand{\metric}{\mathsf{g}}

\newcommand{\Ad}{\operatorname{Ad}}

\newcommand{\id}{\mathrm{id}}
\newcommand{\Hom}{\mathrm{Hom}}
\newcommand{\spa}{\mathrm{span}}

\newcommand{\ie}{i.e.\ }

\newcommand{\eps}{\varepsilon}

\newcommand{\diag}{\mathrm{diag}}

\newcommand{\Gr}{\mathbf{Gr}}

\title[Affine actions with Hitchin linear part]{
Affine actions with Hitchin linear part}

\author{Jeffrey Danciger}
\address{Department of Mathematics, The University of Texas at Austin, 1 University Station C1200, Austin, TX 78712, USA}
\email{jdanciger@math.utexas.edu}

\author{Tengren Zhang}
\address{Department of Mathematics, National University of Singapore, 21 Lower Kent Ridge Road, Singapore 119077}
\email{matzt@nus.edu.sg}

\thanks{J.D. was partially supported by an Alfred P. Sloan Foundation fellowship, and by the National Science Foundation under grants DMS~1510254 and DMS~1812216. T.Z. was partially supported by the National Science Foundation grant DMS~1566585, and by the NUS-MOE grant R-146-000-270-133.
The authors also acknowledge support from the GEAR Network, funded by the National Science Foundation under grant numbers DMS 1107452, 1107263, and 1107367 (``RNMS: GEometric structures And Representation varieties").}

\begin{document}

\maketitle

\begin{abstract}
Properly discontinuous actions of a surface group by affine automorphisms of $\RR^d$ were shown to exist by Danciger-Gueritaud-Kassel. We show, however, that if the linear part of an affine surface group action is in the Hitchin component, then the action fails to be properly discontinuous. The key case is that of linear part in $\SO(n,n-1)$, so that the affine action is by isometries of a flat pseudo-Riemannian metric on $\RR^d$ of signature $(n,n-1)$. Here, the translational part determines a deformation of the linear part into $\PSO(n,n)$-Hitchin representations and the crucial step is to show that such representations are not Anosov in $\PSL(2n,\RR)$ with respect to the stabilizer of an $n$-plane. 
We also prove a negative curvature analogue of the main result, that the action of a surface group on the pseudo-Riemannian hyperbolic space of signature $(n,n-1)$ by a $\PSO(n,n)$-Hitchin representation fails to be properly discontinuous.
\end{abstract}

\tableofcontents

\section{Introduction}

This paper is about an application of some rapidly developing tools from higher Teichm\"uller-Thurston theory to the study of properly discontinuous group actions in affine geometry, flat pseudo-Riemannian geometry, and also pseudo-Riemannian hyperbolic geometry.

An \emph{affine manifold} is a manifold $M$ equipped with a flat, torsion-free affine connection $\nabla$. If the geodesic flow of $\nabla$ is complete, then $M$ is called a \emph{complete affine manifold}.
Equivalently, a complete affine manifold is the quotient $M = \Gamma \backslash \RR^d$ of a \emph{proper affine action}, \ie a properly discontinuous action of a group $\Gamma$ by affine automorphisms of~$\RR^d$. Here the group $\Gamma$, which identifies with the fundamental group $\pi_1 M$, is required to be torsion free (otherwise the quotient is an orbifold rather than a manifold).
Complete affine manifolds are generalizations of complete Euclidean manifolds, for which the connection $\nabla$ is the Levi-Cevita connection of a complete flat Riemannian metric or equivalently the action by $\Gamma$ preserves the standard Euclidean metric on $\RR^d$. In this case, by Bieberbach's theorems, $\Gamma$ contains a finite index subgroup $\Gamma_0 \cong \ZZ^k$ for which the corresponding finite cover of $M$ deformation retracts onto a totally geodesic $k$-torus.  

By contrast to the setting of Euclidean geometry, the general picture of what complete affine manifolds $M$ can look like is much more mysterious. The Auslander conjecture~\cite{Aus1964, AMS3} gives a conjectural analogue of Beiberbach's theorem for the case that $M$ is compact. However, in the non-compact case, it is unclear what restrictions the presence of a complete affine structure puts on the topology of $M$. Indeed in 1983, Margulis~\cite{mar83, mar87} found examples of proper affine actions by non-abelian free groups in dimension three, destroying the natural intuition that a complete flat affine structure ought to obstruct word hyperbolicity in the fundamental group. The geometry, topology, and deformation theory of complete affine three-manifolds with free fundamental group, now known as Margulis spacetimes, has been studied thoroughly in recent years, see e.g.~\cite{DG,GM,GLM,CDG, ChoiGoldman2, DGK1, DGKstrips, DGK2, ChDG}.

Recently, Danciger-Gu\'eritaud-Kassel~\cite{DGK3} found examples of proper affine actions for \emph{any} right-angled Coxeter group, and consequently any subgroup of such a group. 
While this class of groups is very large and rich, let us focus on the sub-class of \emph{surface groups}, \ie the fundamental groups $\pi_1 S$ of closed orientable surfaces $S$ of genus $g \geq 2$. In this case, the construction of~\cite{DGK3} gives examples of proper affine actions in dimension as low as $d = 6$. 

Here we take up the problem of classifying proper affine actions by surface groups $\pi_1 S$, or equivalently complete affine manifolds which are homotopy equivalent to a surface $S$.
The advantage in considering surface groups is that tools to study representations of surface groups have developed rapidly over recent years.  Indeed, this paper will make use of some recent results in higher Teichm\"uller-Thurston theory  in order to obstruct properness for affine actions coming from a well-studied component of representations, called the \emph{Hitchin component}.

The group of affine automorphisms $\mathrm{Aff}(\RR^d) = \GL(d, \RR) \ltimes \RR^d$ decomposes as the semi-direct product of the linear automorphisms $\GL(d,\RR)$ with the translation subgroup $\RR^d$. Hence an affine action of the group $\Gamma$ consists of two pieces of data
\begin{align*}
(\rho, u): \Gamma \to \mathrm{Aff}(\RR^d) = \GL(d, \RR) \ltimes \RR^d
\end{align*}
where here $\rho: \Gamma \to \GL(d,\RR)$, a homomorphism, is called the \emph{linear part}, and $u: \Gamma \to \RR^d$, a cocycle twisted by $\rho$, is called the \emph{translational part}. The main theorem is:

\begin{theorem}\label{thm:main-affine}
Suppose that $(\rho, u): \pi_1S \to \mathrm{Aff}(\RR^d) = \GL(d,\RR) \ltimes \RR^d$ is a proper affine action. Then the linear part $\rho$ does not lie in a Hitchin component.
\end{theorem}

Here, the term {Hitchin component} refers to a special connected component (in fact, multiple related components) of representations that was singled out by Hitchin~\cite{Hitchin} for its connection to Teichm\"uller theory.
Goldman \cite{Goldman1988} proved that the space $\Hom(\pi_1 S, \PSL(2,\RR))$ has $4g-3$ components, where $g$ is the genus of $S$. The discrete faithful representations sort into two components, called the \emph{Teichm\"uller components}, corresponding to oriented hyperbolic structures on $S$ of each possible orientation. For $G$ an adjoint real split semi-simple Lie group, such as $G = \PSL(d,\RR)$,
the compositions of representations in the Teichm\"uller components with the principle representation $\tau_G: \PSL(2,\RR) \to G$ are called \emph{Fuchsian representations} and the connected components of $\Hom(\pi_1 S, G)$ containing all deformations of Fuchsian representations are called \emph{$G$-Hitchin components} and their elements called \emph{$G$-Hitchin representations} (we suppress the $G$ when clear from context). See Section~\ref{sec:Hitchin}. In the case $G = \PSL(d,\RR)$, if $d$ is odd, there is one Hitchin component and if $d$ is even, there are two Hitchin components which are nonetheless referred to as ``the" Hitchin component since the two components are related by an automorphism of $\PSL(d,\RR)$. Hitchin showed that, like the Teichm\"uller components for $\PSL(2,\RR)$, a $G$-Hitchin component is (after dividing out by conjugation) homeomorphic to a ball of dimension $\dim(G)\cdot(2g-2)$ inside which the locus of Fuchsian representations make up a $6g-6$ dimensional sub-manifold (also a ball). 

Labourie~\cite{labourie2001} proved that a (lift of a) Fuchsian representation $\rho$ is never the linear part of a proper affine action and Theorem~\ref{thm:main-affine} extends Labourie's result to the entire Hitchin component. We note that in the case $d = 3$, the key case of Theorem~\ref{thm:main-affine} follows from Mess~\cite{Mess2007} and Goldman-Margulis~\cite{GM}.
We also note that, unlike the case $d =2$, for $d \geq 3$ the space $\Hom(\pi_1 S, \PSL(d, \RR))$ has only three (resp. six) connected components if $d$ is odd (resp. if $d$ is even). However, the behavior of the representations in the other two (or four) components is very different and still quite mysterious, making a study of proper affine actions with linear part in those components intractible at this time.

Hitchin representations have many nice properties. In particular, Labourie \cite{labourie2006} showed that every $\PSL(d,\RR)$-Hitchin representation is \emph{Anosov}; indeed he invented the notion of Anosov representation, now central in higher Teichm\"uller-Thurston theory, for the purpose of studying the $\PSL(d,\RR)$-Hitchin component. 
Anosov representations were generalized by Guichard--Wienhard~\cite{GW} to the setting of representations of any word hyperbolic group into a semi-simple Lie group $G$. There is a notion of Anosov for each parabolic subgroup $P$ of $G$. 
For $G$ an adjoint real split semi-simple Lie group, the $G$-Hitchin representations satisfy this notion for the minimal parabolic (the Borel subgroup) $B$, or equivalently for all of the parabolic subgroups of $G$.
Anosov representations, including some recent characterizations  due to Guichard--Gu\'eritaud--Kassel--Wienhard~\cite{GGKW} and Kapovich--Leeb--Porti~\cite{KLPa,KLPb}, will be the essential tool for the proof of Theorem~\ref{thm:main-affine}.

\subsection{Flat pseudo-Riemannian geometry in signature $(n,n-1)$}

The affine transformation $(\rho(\gamma), u(\gamma)) \in \mathrm{Aff}(\RR^d)$ fixes a point if $\rho(\gamma)$ does not have one as an eigenvalue. Hence if $(\rho,u)$ is a free affine action by $\pi_1 S$, then the linear part $\rho(\gamma)$ has one as an eigenvalue, for all $\gamma \in \pi_1 S$, and the same property passes to the Zariski closure of $\rho(\pi_1 S)$. In the context of Theorem~\ref{thm:main-affine}, Guichard's characterization of the possible Zariski closures of Hitchin representations~\cite{Guichard2} allows us to reduce to the case that $d = 2n-1$ is odd (with $n\geq 2$), and that the linear part $\rho(\pi_1 S) \subset \SO(n,n-1)$ is contained in the special orthogonal group of the standard indefinite symmetric bilinear form of signature $(n,n-1)$. The vector space $\RR^d$ together with this form will be denoted by $\RR^{n,n-1}$ and the affine space of this vector space, equipped with the induced flat pseudo-Riemannian metric, will be denoted by $\EE^{n,n-1}$. Hence, in this case the affine action $(\rho,u)$ is by isometries of $\EE^{n,n-1}$. 
Theorem~\ref{thm:main-affine} is a corollary of:
\begin{theorem}\label{thm:main-flat}
Suppose $(\rho, u): \pi_1 S \to \mathrm{Isom^+}(\EE^{n,n-1}) = \SO(n,n-1) \ltimes \RR^{n,n-1}$ is an action by isometries of $\EE^{n,n-1}$ with linear part $\rho$ a $\SO(n,n-1)$-Hitchin representation. Then the action is not proper.
\end{theorem}

In fact, if $n$ is odd, Theorem~\ref{thm:main-affine} follows from an observation of Abels-Margulis-Soifer~\cite{AMS1} (see Remark~\ref{rem:AMS}), so we need only treat the case that $n$ is even. However, for much of the setup we will not distinguish between the case $n$ odd and $n$ even.

The strategy for Theorem~\ref{thm:main-affine} follows the key point of view in the work of Danciger-Gu\'eritaud-Kassel~\cite{DGK1,DGKstrips, DGK4, DGK2} on proper actions by free groups in $\EE^{2,1}$ and their quotients, called Margulis spacetimes. In that context, Margulis spacetimes were studied as limits of their negative curvature counterparts, namely three-dimensional AdS spacetimes which are quotients of anti de Sitter space $\mathrm{AdS}^3 = \HH^{2,1}$. Similarly, here we will study the above isometric actions on $\EE^{n,n-1}$ by thinking of these as infinitesimal versions of isometric actions on the pseudo-Riemannian hyperbolic space $\HH^{n,n-1}$.

\subsection{Deforming into hyperbolic geometry of signature $(n,n-1)$}\label{sec:intro:negcurv}
The pseudo-Riemannian hyperbolic space $\HH^{n,n-1}$ is the model for constant negative curvature in signature $(n,n-1)$. 
The projective model for $\HH^{n,n-1}$ is:
\begin{align*}
\HH^{n,n-1} &= \PP \left\{ x \in \RR^{n,n} \smallsetminus \{0\} : \langle x, x \rangle_{n,n} < 0 \right\} \subset \PP(\RR^{n,n}),
\end{align*}
where here $\RR^{n,n}$ denotes the vector space $\RR^{2n}$ equipped with the standard symmetric bilinear form $\langle \cdot, \cdot \rangle_{n,n}$ of signature $(n,n)$.
The projective special orthogonal group $\PSO(n,n)$ acts transitively on $\HH^{n,n-1}$ as the orientation preserving isometry group of a complete metric of constant negative curvature with signature $(n,n-1)$. With coordinates respecting the orthogonal splitting $\RR^{n,n} = \RR^{n,n-1} \oplus \RR^{0,1}$, the stabilizer in $\PSO(n,n)$ of the basepoint $\mathbf x_0 = [0:\cdots:0:1]$ is precisely the orthogonal group $\OO(n,n-1)$ acting on the $\RR^{n,n-1}$ factor in the standard way and acting trivially on the $\RR^{0,1} = \RR\cdot\mathbf x_0$ factor. 

Now consider a Hitchin representation $\rho: \pi_1 S \to \SO(n,n-1)$. 
Let $\iota_{n,n}: \SO(n,n-1)\hookrightarrow \PSO(n,n)$ be the natural inclusion for the orthogonal splitting $\RR^{n,n} = \RR^{n,n-1} \oplus \RR^{0,1}$ above. Then $\iota_{n,n} \circ \rho$ stabilizes the basepoint $\mathbf x_0 \in \HH^{n,n-1}$ and acts on the tangent space of that point, a copy of $\RR^{n,n-1}$, in the standard way by linear isometries. Consider a deformation path $\varrho_\eps: \pi_1 S \to \PSO(n,n)$ based at $\varrho_0 =\iota_{n,n} \circ \rho$. 
Any such deformation $\varrho_\eps$ (for $\eps$ not necessarily small, or possibly zero) is a $\PSO(n,n)$-Hitchin representation and such representations make up the $\PSO(n,n)$-Hitchin component. The derivative of $\varrho_\eps$ at time $\eps = 0$ is naturally a cocycle $v: \pi_1 S \to \mathfrak{pso}(n,n)$ twisted by the adjoint action of $\varrho_0$, which splits as an invariant orthogonal sum
\begin{align*}
\mathfrak{pso}(n,n) = \mathfrak{so}(n,n-1) \oplus \RR^{n,n-1},
\end{align*} where the action in the first factor is by the adjoint representation and the action in the second factor is by the standard representation. Hence the projection $u$ of the infinitesimal deformation $v$ to the $\RR^{n,n-1}$ factor gives a cocycle of translational parts for an affine action $(\rho, u)$ on $\EE^{n, n-1}$. 
The geometric way to think of this fact is as follows: As $\eps \to 0$, the action of each $\varrho_\eps(\gamma)$ moves the basepoint $\mathbf x_0$ less and less, and by zooming in on the basepoint at just the right rate as $\eps \to 0$ and taking a limit (in an appropriate sense), the action converges to an affine action $(\rho,u)$ on $\EE^{n,n-1}$ whose action on the basepoint is translation by the derivative $\left.\dfrac{\D}{\D \eps}\right|_{\eps = 0}\varrho_\eps(\gamma)\mathbf x_0 \in T_{\mathbf x_0} \HH^{n,n-1} = \RR^{n,n-1}$, where here we identify the tangent space $T_{\mathbf x_0} \HH^{n,n-1}$ with $\RR^{n,n-1}$.

Now, since $\iota_{n,n} \circ \rho$ has trivial centralizer in $\PSO(n,n)$, results of Goldman~\cite{Goldmansymplectic} on representation varieties of surface groups imply that any $\rho$-cocycle $u: \pi_1 S \to \RR^{n,n-1}$ is realized as the $\RR^{n,n-1}$ part of the derivative of a smooth deformation path $\varrho_\eps$ as above (and the $\mathfrak{so}(n,n-1)$ part may be taken to be trivial).
We prove a key lemma (Lemma~\ref{lem:key}) that connects a criterion~\cite{GLM, GT} for properness of the affine action $(\rho, u)$ on $\EE^{n,n-1}$ with the first order behavior of the two middle eigenvalues of elements $\varrho_\eps(\gamma)$, an inverse pair $\lambda_n, \lambda_n^{-1}$ which converges to the two one-eigenvalues of $\iota_{n,n} \circ \rho$ as $\eps \to 0$. From this eigenvalue behavior, we use~\cite{GGKW,KLPa,KLPb} to prove that if $(\rho,u)$ acts properly on $\RR^{n,n-1}$, then the representations $\varrho_\eps$ satisfy an unexpected Anosov condition. Specifically, for $\eps > 0$ small enough, $\iota_{2n} \circ \varrho_\eps$ is Anosov with respect to the stabilizer $P_n$ of an $n$-plane in $\RR^{2n}$, where here $\iota_{2n}: \PSO(n,n) \to \PSL(2n,\RR)$ is the inclusion, see Theorem~\ref{thm:general}. 
Theorem~\ref{thm:main-flat} then follows from the next theorem, which is the main technical result of the paper:

\begin{theorem}\label{thm:main2} 
If $\varrho:\pi_1 S \to\PSO(n,n)$ is a $\PSO(n,n)$-Hitchin representation, then $\iota_{2n} \circ \varrho: \pi_1 S \to \PSL(2n,\RR)$ is not $P_n$-Anosov .
\end{theorem}

As discussed above, $\PSO(n,n)$-Hitchin representations enjoy all possible forms of Anosovness available in $\PSO(n,n)$. However, a $\PSO(n,n)$-Hitchin representation has no reason to be $P_n$-Anosov in the larger group $\PSL(2n, \RR)$, and the representations landing in the subgroup $\iota_{n,n}\SO(n,n-1)$ (of the form above $\iota_{n,n}\circ \rho$) obviously fail this condition. Theorem~\ref{thm:main2} says that $P_n$-Anosovness in $\PSL(2n, \RR)$ never happens, even by accident, to the inclusion of a $\PSO(n,n)$-Hitchin representation.

The proof of Theorem~\ref{thm:main2} uses more than just Anosovness of $\PSO(n,n)$-Hitchin representations, specifically it uses that $\PSO(n,n)$-Hitchin representations satisfy Fock-Goncharov positivity~\cite{FockGoncharov}. However, 
we remark that the proof of Theorem~\ref{thm:main-flat} outlined above only requires Theorem~\ref{thm:main2} for $\PSO(n,n)$-Hitchin representations which are small deformations of $\SO(n,n-1)$-Hitchin representations. We remark that a proof of Theorem~\ref{thm:main2} in that case can be achieved without the full strength of positivity in $\PSO(n,n)$, using in its place some special properties of $\SO(n,n-1)$-Hitchin representations (Labourie's Property (H) from~\cite{labourie2006}) and an argument about persistence of such properties under small deformation into $\PSO(n,n)$, but we do not give this proof here. 

\begin{remark}\label{rem:Sourav}
Sourav Ghosh has announced independent work~\cite{Sourav} that has overlap with some of our results. 
Specifically, Lemma~\ref{lem:key} and Theorem~\ref{thm:general}, showing that a proper action on $\EE^{n,n-1}$ whose linear part is Anosov with respect to the stabilizer of an isotropic $(n-1)$-plane corresponds to a deformation path into $\PSO(n,n)$ for which the inclusions into $\PSL(2n,\RR)$ are $P_n$-Anosov, is also proved by Ghosh. Theorem~\ref{thm:general} is one of the two main inputs in our proof of Theorem~\ref{thm:main-flat}, the other being Theorem~\ref{thm:main2} which Ghosh does not obtain. We also remark that, whereas we focus on surface groups only, Ghosh works in the more general setting of actions by any word hyperbolic group.
\end{remark}

\subsection{Proper actions in $\HH^{n,n-1}$}

We also use Theorem~\ref{thm:main2}, together with a properness criterion, Theorem~\ref{thm:gen-H}, based on techniques from~\cite{GGKW} to show the negative curvature analogue of Theorem~\ref{thm:main-flat}.

\begin{theorem}\label{thm:main-H}
A $\PSO(n,n)$-Hitchin representation $\varrho: \pi_1 S \to \PSO(n,n)$ does not act properly on $\HH^{n,n-1}$.
\end{theorem}

We note that proper actions by surface groups on $\HH^{n,n-1}$ do exist when $n$ is even (Okuda~\cite{okuda13}), but not when $n$ is odd (Benoist~\cite{Benoist}). 
Note that in the case $n=2$, Theorem~\ref{thm:main-H} follows from work of Mess~\cite{Mess2007} or of Gu\'eritaud--Kassel~\cite{GueKas2017}.
In that case $\HH^{n,n-1} = \HH^{2,1}$ is the three-dimensional anti-de Sitters space, whose study is greatly simplified by  the accidental isomorphism between $\PSO(2,2)_0$ and $\PSL(2,\RR) \times \PSL(2,\RR)$. The $n=2$ case of the proof given here of Theorem~\ref{thm:main-H}, through Theorem~\ref{thm:main2}, is fundamentally different. Indeed, the work of Mess and of Gu\'eritaud--Kassel does not naturally generalize to higher $\HH^{n,n-1}$, though it does generalize to the setting of some other homogeneous spaces whose structure group is a product.

\subsection{Overview of proofs and organization}
The paper naturally splits into two main parts, namely Sections~\ref{sec:grass}--\ref{sec:thm:main-H} and Sections~\ref{sec:length-functions}--\ref{sec:geometric limit}. The proof of Theorem~\ref{thm:main2} is given in Section~\ref{sec:main2}, which builds on Sections~\ref{sec:grass}--\ref{sec:Anosov-repns}. Sections~\ref{sec:grass} gives background information about flag manifolds. Section~\ref{sec:3} introduces Hitchin representations and positivity and then proves a new transversality result, Corollary~\ref{cor:transversality}, for triples on the positive curve associated to a $\PSO(n,n)$-Hitchin representation. Section~\ref{sec:Anosov-repns} introduces Anosov representations, reviews some relevant recent results about them, and also proves Proposition~\ref{prop:ratio of norms}, a key dynamical input for Theorem~\ref{thm:main2}.
Section~\ref{sec:thm:main-H} proves a general theorem, Theorem~\ref{thm:gen-H}, connecting properness of actions on $\HH^{n,n-1}$ with certain Anosov conditions and then proves Theorem~\ref{thm:main-H}.

In the second, more geometric part of the paper, Section~\ref{sec:length-functions} reviews the properness criterion (Theorem~\ref{thm:GLM-GT}) for actions on $\EE^{n,n-1}$ with Anosov linear part. This criterion is stated in terms of a signed length function associated to such an affine action. We introduce a new length function in the setting of $\HH^{n,n-1}$, defined analogously. Section~\ref{sec:geometric limit} gives the main geometric transition arguments connecting the behavior of actions on $\EE^{n,n-1}$ with that of actions on $\HH^{n,n-1}$. We explain how a $\EE^{n,n-1}$ action determines a path of $\HH^{n,n-1}$ actions, which are in a certain sense nearby, and prove Lemma~\ref{lem:key}, which relates the length function for the $\EE^{n,n-1}$ action to the first order behavior of the length functions for the associated path of $\HH^{n,n-1}$ actions. We then prove Theorem~\ref{thm:general}, relating proper discontinuity of an $\EE^{n,n-1}$ action with certain Anosov behavior of the nearby $\HH^{n,n-1}$ actions. Finally, we prove Theorems~\ref{thm:main-flat} and then Theorem~\ref{thm:main-affine}.

\subsection{Acknowledgements}
We thank Bill Goldman, Oliver Guichard, Fanny Kassel, and Qiongling Li for valuable discussions.

We also acknowledge the independent work of Sourav Ghosh~\cite{Sourav}, already mentioned in Remark~\ref{rem:Sourav}, which includes similar statements to our Lemma~\ref{lem:key} and Theorem~\ref{thm:general} (see Proposition 0.0.1 and Theorem 0.0.2), and also an alternate proof of Theorem~\ref{thm:gen-H} based on ideas of~\cite{GT} (see Theorem 0.0.3). When we recently found out about this work, we discussed our results with Ghosh. Thanks to those discussions, we realized that the techniques from Sections~\ref{sec:length-functions} and~\ref{sec:geometric limit}, which are similar to Ghosh's techniques, and which we had originally written in the more narrow context of affine actions whose linear part is Anosov with respect to the minimal parabolic, actually apply in the more general context where the linear part is only assumed to be Anosov with respect to the stabilizer of an isotropic $(n-1)$-plane, as in Theorem~\ref{thm:general}.

We also thank the Mathematisches Forschungsinstitut Oberwolfach for hosting both authors in September 2018 for the meeting ``New trends in Teichm\"uller theory and mapping class groups", at which the first author gave a lecture about this project. See~\cite{MFO} for a short description of the contents of the lecture. The first author acknowledges with regret that at the time of that lecture, he was not aware that Theorem~\ref{thm:GLM-GT} had already been shown by Ghosh-Trieb~\cite{GT} in full generality, and he failed to cite that work during the lecture.

\section{Grassmannians and flag manifolds.}\label{sec:grass}
In this paper, three semi-simple Lie groups will appear frequently, namely $G := \PSL(d,\RR)$ with $d\geq 2$, and $G' := \PSO(n,n)$ and $G'' := \SO(n,n-1)$ with $n\geq 2$. After introducing these groups, the goal of this section will be to give a description of the relevant flag manifolds on which these groups act, and to then give some basic facts about them that will be used throughout the paper.

First, recall that $\SL(d,\RR)$ is the space of volume preserving linear automorphims of the vector space $\RR^d$. Fixing a basis of $\RR^d$, $\SL(d,\RR)$ identifies with the group of $d \times d$ real matrices of determinant one. The group $\PSL(d,\RR)$ is the quotient of $\SL(d,\RR)$ by its center, which is trivial if $d$ is odd, or is $\{ \pm \id\}$ if $d$ is even.

Next, consider the symmetric anti-diagonal matrix,
\[J_d=\left(\begin{array}{ccccc}
0&0&\dots&0&1\\
0&0&\dots&1&0\\
\vdots&\vdots&&\vdots&\vdots\\
0&1&\dots&0&0\\
1&0&\dots&0&0\\
\end{array}\right),\] 
In the case $d = 2n$ is even, the signature of $J_d$ is $(n,n)$ (meaning it has $n$ positive eigenvalues and $n$ negative eigenvalues). We will use the notation $\langle \cdot, \cdot \rangle_{n,n}$ to denote the symmetric bilinear form on $\RR^{2n}$ whose matrix is $J_{2n}$:
\[ \langle x, y\rangle_{n,n} = \sum_{i=1}^{2n} x_i y_{2n+1-i} \]
where $(x_1, \ldots, x_{2n})$ and $(y_1, \ldots, y_{2n})$ are the coordinates of $x, y \in \RR^{2n}$ with respect to the standard basis $(e_1, \ldots, e_{2n})$. We prefer to work in this basis in the initial part of the paper. However, in the final sections of the paper it will be more natural to work in a basis $(e_1', \ldots, e_{2n}')$ which diagonalizes the form $\langle \cdot, \cdot \rangle_{n,n}$. We will use the notation $\RR^{n,n}$ to denote the vector space $\RR^{2n}$ together with the symmetric bilinear form $\langle \cdot, \cdot \rangle_{n,n}$.  We equip $\RR^{n,n}$ with the orientation making the standard basis positive and define $\SO(n,n) < \SL(2n,\RR)$ to be the orientation preserving automorphism group of $\RR^{n,n}$, that is the special linear automorphisms of $\RR^{2n}$ which preserve $\langle \cdot, \cdot \rangle_{n,n}$. We define $\PSO(n,n)$ to be the projection of $\SO(n,n)$ to $\PSL(2n,\RR)$. 

In the case $d = 2n-1$ is odd, the signature of $J_{2n-1}$ is $(n,n-1)$ (meaning $J_{2n-1}$ has $n$ positive eigenvalues and $n-1$ negative eigenvalues). We will use the notation $\langle \cdot, \cdot \rangle_{n,n-1}$ to denote the symmetric bilinear form on $\RR^{2n-1}$ whose matrix is $J_{2n-1}$:
\[ \langle x, y\rangle_{n,n-1} = \sum_{i=1}^{n-1} (x_i y_{2n-i} + x_{2n-i} y_i) + x_{n}y_{n} \]
where $(x_1, \ldots, x_{2n-1})$ and $(y_1, \ldots, y_{2n-1})$ are the coordinates of $x, y \in \RR^{2n}$ with respect to the standard basis $(f_1, \ldots, f_{2n-1})$ of $\RR^{2n-1}$.
We will use the notation $\RR^{n,n-1}$ to denote the vector space $\RR^{2n-1}$ together with the symmetric bilinear form $\langle \cdot, \cdot \rangle_{n,n-1}$.  We equip $\RR^{n,n-1}$ with the orientation making the standard basis positive
and define $\SO(n,n-1) = \PSO(n,n-1)$ to be the orientation preserving automorphism group of $\RR^{n,n-1}$, that is the special linear automorphisms of $\RR^{2n-1}$ that preserve the symmetric bilinear form $\langle \cdot, \cdot \rangle_{n,n-1}$. 

We will often embed $\RR^{n,n-1}$ in $\RR^{n,n}$ in the following way. The vector $e_{n} - e_{n+1}$ has negative signature in $\RR^{n,n}$. Hence the orthogonal complement $(e_n - e_{n+1})^\perp$ has signature $(n, n-1)$ and we think of it as a copy of $\RR^{n,n-1}$. More specifically, we embed $\RR^{2n-1}$ as a subspace of $\RR^{2n}$ by the linear map
\begin{align*}
f_i &\mapsto e_i \text{ for } 1 \leq i \leq n-1\\
f_i &\mapsto e_{i+1} \text{ for } n+1 \leq i \leq 2n-1\\
f_n &\mapsto \frac{1}{2}(e_n + e_{n+1}).
\end{align*} 
Then, the restriction of the form $\langle \cdot, \cdot \rangle_{n,n}$ to (the image of) $\RR^{2n-1}$ is precisely 
the (image of the) form $\langle \cdot, \cdot \rangle_{n,n-1}$. Hence we will write $\RR^{n,n} = \RR^{n,n-1} \oplus \RR^{0,1}$, where on the right-hand side $\RR^{n,n-1}$ is understood to be the image of $\RR^{2n-1}$ under the above map and $\RR^{0,1}$ is understood to be the span of $e_n - e_{n+1}$, and each is equipped with the restriction of $\langle \cdot, \cdot \rangle_{n,n}$.

\subsection{Grassmanians and Isotropic Grassmannians}\label{sec:grassmannian}
We introduce some natural compact homogeneous spaces associated to the main Lie groups of interest. 

For $1 \leq k \leq d-1$, let $\Gr_k(\RR^d)$, denote the space of $k$-dimensional vector subspaces in $\RR^d$, known as the \emph{Grassmannian} of $k$-planes in $\RR^d$. 

In the case $d = 2n$, and $1 \leq k \leq n$, let $\Gr_k(\RR^{n,n}) \subset \Gr_k(\RR^{2n})$ denote the space of \emph{null} $k$-planes:
\begin{align*}
\Gr_k(\RR^{n,n}) &:= \left\{ H \in \Gr_k(\RR^{2n}) : \langle x, x \rangle_{n,n} = 0 \text{ for all } x \in H \right\}
\end{align*}
For $n+1 \leq k \leq 2n$, denote
\begin{align*}
\Gr_k(\RR^{n,n}) &:= \left\{ H \in \Gr_k(\RR^{2n}): H^\perp \in \Gr_{2n-k}(\RR^{n,n}) \right\}
\end{align*}
where $H^\perp$ denotes the orthogonal space to $H$ with respect to $\langle \cdot, \cdot \rangle_{n,n}$.
Note that $\perp$ defines a canonical isomorphism $\Gr_k(\RR^{n,n}) \cong \Gr_{2n-k}(\RR^{n,n})$.

Similarly, for $1 \leq k \leq n-1$,
\begin{align*}
\Gr_k(\RR^{n,n-1}) &:= \left\{ H \in \Gr_k(\RR^{2n-1}) : \langle x, x \rangle_{n,n} = 0 \text{ for all } x \in H \right\}
\end{align*}
is the space of null $k$-planes in $\RR^{2n-1}$. For  for $n \leq k \leq 2n-1$, define
\begin{align*}
\Gr_k(\RR^{n,n-1}) &:= \left\{ H \in \Gr_k(\RR^{2n}): H^\perp \in \Gr_{2n-1-k}(\RR^{n,n}) \right\}
\end{align*}
where here $H^\perp$ denotes the orthogonal space to $H$ with respect to $\langle \cdot, \cdot \rangle_{n,n-1}$.
Note that $\perp$ defines a canonical isomorphism $\Gr_k(\RR^{n,n-1}) \cong \Gr_{2n-1-k}(\RR^{n,n-1})$.

\begin{proposition}\label{prop:transitivity}\
\begin{enumerate}
\item For all $1 \leq k \leq d-1$, $G := \PSL(d,\RR)$ acts transitively on $\Gr_k(\RR^d)$. Hence $\Gr_k(\RR^d) = G/P_k$ is a homogeneous space of $G$, where $P_k$ denotes the stabilizer of the $k$-plane $\spa \{e_1, \ldots, e_k\}$.
\item For all $1 \leq k \leq 2n-1$, $k \neq n$, $G' := \PSO(n,n)$  acts transitively on $\Gr_k(\RR^{n,n})$. Hence $\Gr_k(\RR^{n,n}) = G'/P'_k$ is a homogeneous space of $G'$, where $P_k'$ denotes the stabilizer of the null $k$-plane $\spa \{e_1, \ldots, e_k\}$.
\item For all $1 \leq k \leq 2n-2$, $G'' := \SO(n,n-1)$ acts transitively on $\Gr_k(\RR^{n,n-1})$. Hence, $\Gr_k(\RR^{n,n-1}) = G''/P''_k$ is a homogeneous space of $G''$, where $P_k''$ denotes the stabilizer of the null $k$-plane $\spa \{f_1, \ldots, f_k\}$.
\end{enumerate} 
\end{proposition}

\begin{proof}
The proofs are well-known linear algebra exercises. Let us quickly recall a proof of~(2) to highlight what is different about the situation $k = n$, to be discussed after this proof. 

 By the isomorphism $\Gr_k(\RR^{n,n})\simeq \Gr_{2n-k}(\RR^{n,n})$, we may assume $1 \leq k \leq n-1$. Let $H \in \Gr_k(\RR^{n,n})$, a null $k$-plane. Let $v_1, \ldots, v_k$ be a basis of $H$. 
 Since the form $\langle \cdot, \cdot \rangle_{n,n}$ is non-degenerate, there exists vectors $v_1', \ldots, v_k'$ so that $\langle v_i, v_j' \rangle = \delta_{ij}$ for $1 \leq i,j \leq k$. By adjusting $v_1', \ldots, v_k'$ with elements of $H$ we may further arrange that $v_1', \ldots, v_k'$ span a null $k$-plane $H'$, which necessarily intersects $H$ trivially. Then $H \oplus H'$ is a non-degenerate subspace of $\RR^{n,n}$ which therefore has signature $(k,k)$. Its orthogonal complement $(H \oplus H')^\perp$ has signature $(n-k, n-k)$ and a basis $w_1, \ldots, w_{n-k}, w_1', \ldots, w_{n-k}'$ with the property that $w_1, \ldots, w_{n-k}$ and $w_1', \ldots, w_{n-k}'$ each span null $(n-k)$-planes and satisfy $\langle w_i, w_j' \rangle = \delta_{ij}$. Then the following defines an orthogonal transformation of $\RR^{n,n}$:
 \begin{align*}
e_i &\mapsto v_i &\text{ for all } 1 \leq i \leq k\\
e_{i} &\mapsto w_{i - k} &\text{ for all } k+1 \leq i \leq n\\
e_i &\mapsto w_{2n-k-i+1}' &\text{ for all } n+1 \leq i \leq 2n-k\\
e_i &\mapsto v_{2n+1 -i}' &\text { for all }  2n-k +1 \leq i \leq 2n
 \end{align*} 
 This automorphism maps the standard isotropic $k$-plane, spanned by $e_1, \ldots, e_k$ to $H$. However, this automorphism might not preserve orientation. To fix that issue, we may precompose with the orientation reversing automorphism which swaps $e_n$ and $e_{n+1}$ and leaves all other basis vectors fixed. Of course, since $k < n$, this does not change the fact that $\spa \{ e_1, \ldots, e_k \} \mapsto H$.
\end{proof}

By contrast to Proposition~\ref{prop:transitivity}.(2), the $\PSO(n,n)$ action on $\Gr_n(\RR^{n,n})$ has two orbits. Here is an invariant that distinguishes them (which can already be seen in the proof above). For any $H\in\Gr_n(\RR^{n,n})$, choose a basis $(v_1,\dots,v_n)$ of $H$. This extends \emph{uniquely} to a basis $(v_1, \ldots, v_n, v_n', \ldots, v_1')$ of $\RR^{2n}$ so that $H' = \spa \{v_1', \ldots, v_n'\}$ is also a null $n$-plane and $\langle v_i, v_j' \rangle_{n,n} = \delta_{ij}$ for $1 \leq i,j \leq n$. Then define
\begin{equation}\label{eqn:tau}\tau(H):=\frac{v_1\wedge\dots\wedge v_{n} \wedge v_n' \wedge \dots v_1'}{e_1 \wedge \dots \wedge e_{2n}}\in\{\pm1\},\end{equation}

In fact, $\tau(H)$ does not depend on the choice of the basis $(v_1,\dots,v_n)$ for $H$: In the procedure above, a change of the basis $v_1, \ldots, v_n$ of $H$, represented by a $n \times n$ matrix, leads to a change of the basis $v_1', \ldots, v_n'$ of $H'$ by the exact same matrix, hence any change of orientation in the basis for $H$ is canceled out by the same change of orientation in the basis for $H'$. We define:
\begin{align*}
\Gr_n^+(\RR^{n,n}) &:=\{H\in\Gr_n(\RR^{n,n}):\tau(H)=+1\},\\
\Gr_n^-(\RR^{n,n}) &:=\{H\in\Gr_n(\RR^{n,n}):\tau(H)=-1\}
\end{align*}
and refer to the first as the Grassmannian of \emph{positive} isotropic $n$-planes and to the second as the Grassmannian of \emph{negative} isotropic $n$-planes. See Figure~\ref{fig:hyperboloid} for the case $n=2$.

\begin{remark}\label{rem:switch orientation}
There is no intrinsic difference between   the space of positive isotropic $n$-planes and the space of negative isotropic $n$-planes. Indeed, an element of $\PO(n,n) \setminus \PSO(n,n)$ reverses the orientation of $\RR^{n,n}$ and hence takes the positive isotropic $n$-planes to the negative ones and vice versa. Hence any argument about $\Gr_n^+(\RR^{n,n})$ that does not use a particular choice of orientation on $\RR^{n,n}$ also applies to $\Gr_n^-(\RR^{n,n})$.
\end{remark}

\begin{proposition}\label{prop:projections}
The action of $G' = \PSO(n,n)$ on $\Gr_n(\RR^{n,n})$ has two orbits, $\Gr_n^+(\RR^{n,n})$ and $\Gr_n^-(\RR^{n,n})$. 
Each isotropic $(n-1)$-plane $H_0 \in \Gr_{n-1}(\RR^{n,n})$ is contained in a unique positive isotropic $n$-plane $H_+ \in \Gr_n^+(\RR^{n,n})$ and a unique negative isotropic $n$-plane $H_- \in \Gr_n^-(\RR^{n,n})$.
The maps 
\begin{align*}
\varpi_+&: \Gr_{n-1}(\RR^{n,n}) \to \Gr_n^+(\RR^{n,n}),\\
\varpi_-&: \Gr_{n-1}(\RR^{n,n}) \to \Gr_n^-(\RR^{n,n})
\end{align*}
defined respectively by $H_0 \mapsto H_+$ and $H_0 \mapsto H_-$ are $G'$-equivariant fiber bundle projections, with fiber 
a copy of $\RP^{n-1}$.
\end{proposition}

\begin{proof}
Consider $H_0:=\spa \{e_1,\dots,e_{n-1}\} \in \Gr_{n-1}(\RR^{n,n})$. Then $H_+ := \spa \{e_1, \ldots, e_n\}$ and $H_- := \spa \{e_1, \ldots, e_{n-1}, e_{n+1}\}$ are the unique isotropic $n$-planes containing $H_0$. We see this as follows. Let $H_0^\perp = \spa  \{e_1, \ldots, e_{n+1}\} \supset H_0$ denote the orthogonal space to $H_0$. Then the inner product $\langle \cdot, \cdot \rangle_{n,n}$ descends to a well-defined inner product on the quotient $H_0^\perp/H_0$ which has signature $(1,1)$. Hence $H_0^\perp/H_0$ contains exactly two null lines whose inverse images in $H_0^\perp$ are $H_+$ and $H_-$. Note also that $\tau(H_+) = +1$ and $\tau(H_-) = -1$.
By transitivity of the $G'$-action on $\Gr_{n-1}(\RR^{n,n})$ (Proposition~\ref{prop:transitivity}.(2)), the maps $\varpi_+$ and $\varpi_-$ may be expressed respectively as $g H_0 \mapsto g H_+$ and $g H_0 \mapsto g H_-$. Since every $(n-1)$-plane contained in an isotropic $n$-plane is also isotropic, its clear that $\varpi_+$ and $\varpi_-$ are surjective, and that the fiber above $H_+$ (resp. $H_-$) is the space $\PP(H_+^*)$ of $(n-1)$-planes in $H_+$ (resp. the space $\PP(H_-^*)$ of $(n-1)$-planes in $H_-$), a copy of $\RP^{n-1}$.
\end{proof}

\begin{figure}[h]
\includegraphics[height=4.0cm]{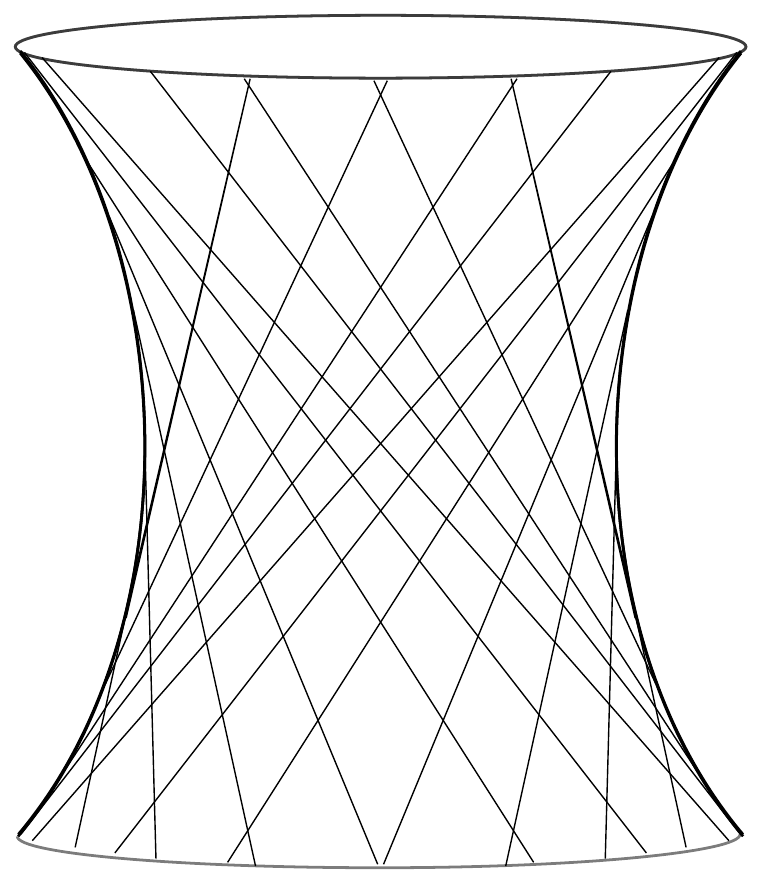}
\caption{The subset $\Gr_1(\RR^{2,2})$ in $\Gr_1(\RR^4) = \RP^3$ is the well-known doubly ruled hyperboloid. The lines of one of the rulings make up $\Gr_2^+(\RR^{2,2})$ while the lines of the other make up $\Gr_2^-(\RR^{2,2})$. The projection map $\varpi_+$ (resp. $\varpi_-$) simply maps a point of $\Gr_1(\RR^{2,2})$ to the line of the $+$ ruling (resp. the $-$ ruling) containing it.}\label{fig:hyperboloid}
\end{figure}

\subsection{Flag manifolds and parabolic subgroups}\label{sec:flags}

For each $1 \leq k < d$, the Grassmannian $\Gr_k(\RR^d) = G/P_k$ is a special example of a flag manifold of $G = \PSL(d,\RR)$ and the stabilizer $P_k < G$ of a $k$-plane in $\RR^d$ is an example of a parabolic subgroup. More generally, by a flag manifold of $G$, we will mean a compact homogeneous space of the form $G/P$ for some parabolic subgroup $P < G$.
Before discussing parabolic subgroups in general, let us first introduce the most important example, the Borel subgroup.

In general, a \emph{Borel subgroup} $\mathbf B$ of an algebraic group $\mathbf G$ is a maximal Zariski closed and Zariski connected solvable subgroup. The Lie groups $G$ that we will work with in this paper are unions of connected components (for the real topology) of the real points $\mathbf G(\RR)$ of some algebraic group $\mathbf G$, so we will understand the term Borel subgroup to mean a subgroup of the form $B = \mathbf B(\RR) \cap G$. 
In the case that $G = \PSL(d,\RR)$, a Borel subgroup $B < G$ is the stabilizer of a \emph{full flag} $F$, \ie a maximal increasing sequence of vector subspaces of $\RR^d$:
\[ F^{(1)} \subset F^{(2)} \subset \cdots \subset F^{(d-1)}\]
where for each $1 \leq k \leq d-1$, $F^{(k)} \in \Gr_k (\RR^d)$ is a $k$-subspace. For example, the \emph{standard full flag} is defined by 
\[ F^{(k)} = \spa \{e_1, \ldots, e_k\}\]
for all $1 \leq k \leq d-1$. The stabilizer of the standard flag is the subgroup of upper triangular matrices, which we will call the standard Borel subgroup. The action of $G$ on the space of full flags is transitive and all Borel subgroups are conjugate. Hence the space of full flags identifies with $\mathcal F_B = G/B$, where $B$ is any Borel subgroup.
For $G' = \PSO(n,n)$, a Borel subgroup is given by the subgroup $B' < G'$ of elements which are upper triangular (\ie the intersection with $G'$ of the standard Borel in $\PSL(2n,\RR)$). 
For $G'' = \SO(n,n-1)$, a Borel subgroup is again given by the subgroup $B'' < G''$ of elements which are upper triangular (\ie the intersection with $G''$ of the standard Borel in $\SL(2n-1, \RR) = \PSL(2n-1, \RR)$.) The associated flag manifolds will be described soon.

\begin{definition}
A \emph{parabolic subgroup} of a semi-simple Lie group $G$ is any subgroup $P$ which contains a Borel subgroup $B$. We call the homogeneous space $\mathcal F_P = G/P$ a \emph{flag manifold}. 
\end{definition}

Two parabolic subgroups $P,Q <G$ are said to be \emph{opposite} if $P \cap Q$ is a reductive subgroup of $G$. Given a parabolic subgroup $P < G$, all parabolic subgroups opposite to $P$ are conjugate to one another. 
For example, the parabolic subgroups opposite to a Borel subgroup $B$ are also Borel subgroups and are conjugate to $B$. 
For another example, if $G = \PSL(d,\RR)$ then the stabilizer $P_k < G$ of the standard $k$-plane $\spa \{e_1, \ldots, e_k\}$ is a parabolic subgroup, whose associated flag manifold is $G/P_k = \Gr_k(\RR^d)$. The stabilizer of any transverse $(d-k)$-plane, for example $\spa \{e_{k+1}, \ldots, e_d\}$, is an opposite parabolic to $P_k$, and any such subgroup is conjugate to $P_{d-k}$.

We shall consider mainly the case of a parabolic subgroup $P < G$ which is conjugate to an opposite of itself.
In this case, the action of $G$ on the product $\mathcal F_P \times \mathcal F_P$ admits a unique open orbit, which we may think of as the subspace of pairs of opposite parabolic subgroups, or alternatively as the pairs of \emph{transverse flags} in $\mathcal F_P$. 
Let us now give some examples of $P, \mathcal F_P,$ and $\mathcal O_P$ in the three settings of interest. The reader may easily verify the following claims.

\begin{example}\label{eg:tangent to grassmannian 1}
Let $G=\PSL(d,\RR)$, and recall that the subgroup $B$ of upper triangular matrices in $G$ is a Borel subgroup. As we saw above, the flag manifold $\mathcal F_B$ is naturally:
\[\mathcal{F}_B =\left\{F=\left(F^{(1)},\dots, F^{(d-1)}\right):
\begin{array}{l}
F^{(k)}\in\Gr_k(\RR^d) \text{ for all } 1 \leq k < d,\\
F^{(i)}\subset F^{(j)}\text{ for all } 1 \leq i \leq j \leq d-1
\end{array}\right\}.\]
The space of transverse flags $\mathcal{O}_B$ is 
\[ \mathcal O_B =\left\{(F_1,F_2)\in\mathcal{F}_B \times \mathcal F_B: F_1^{(k)}+F_2^{(d-k)}=\RR^d\text{ for all }1 \leq k < d\right\}.\]

More generally, let $\mathcal I \subset \{1, \ldots, d-1\}$ be a subset of indices and let $F$ be the standard flag of type $\mathcal I$, meaning $F$ contains the standard subspace $F^{(i)} = \spa \{e_1, \ldots, e_i\}$ of dimension $i$ if and only if $i \in \mathcal I$. Then the stabilizer $P_{\mathcal I}$ of $F$ is a parabolic subgroup of $G$ and $\mathcal F_{\mathcal I} = G/P_{\mathcal I}$ identifies with the space of flags of type $\mathcal I$. Further $P_{\mathcal I}$ is conjugate to its opposite parabolic subgroups if and only if $\mathcal I = \sigma \mathcal I$ for the involution $\sigma: i \mapsto d-i$. In this case the space of transverse pairs of flags is
\[ \mathcal O_{\mathcal I} =\left\{(F_1,F_2)\in\mathcal{F}_{\mathcal I} \times \mathcal F_{\mathcal I}: F_1^{(k)}+F_2^{(d-k)}=\RR^d\text{ for all } k \in \mathcal I\right\}.\]

For $d = 2n$ even, we highlight two important cases. First, if $\mathcal I = \{n\}$, then $\mathcal F_{\mathcal I} = \Gr_n(\RR^{2n})$ is the Grassmannian of $n$-planes in $\RR^{2n}$, and $P_{\mathcal I} = P_n$ is the stabilizer of an $n$-plane. Second, if $\mathcal I = \{ n-1, n+1\}$, then $\mathcal F_{\mathcal I} = \mathcal F_{n-1,n+1}$ is the space of pairs of an $(n-1)$-plane contained in an $(n+1)$-plane and $P_{\mathcal I} = P_{n-1, n+1}$ is the stabilizer of such a flag. 
For $d = 2n-1$ odd, an important case will be that of $\mathcal I = \{n-1, n\}$, for which $\mathcal F_{\mathcal I} = \mathcal F_{n-1,n}$ is the space of pairs of an $(n-1)$-plane contained in an $n$-plane and $P_{\mathcal I} = P_{n-1, n}$ is the stabilizer of such a flag. 
\end{example}

\begin{example}\label{eg:tangent to grassmannian 3}
Let $G'' = \SO(n,n-1)$. Then the subgroup $B'' < G''$ of upper triangular matrices is a Borel subgroup (i.e. $B'' = G'' \cap B$ for $B$ the standard Borel in $\PSL(2n-1, \RR) = \SL(2n-1,\RR)$). The space $\mathcal F_{B''} = G''/B''$ may be described as
\[\mathcal{F}_{B''} = \left\{F=\left(F^{(1)},\dots, F^{(2n-1)}\right):
\begin{array}{l}
F^{(k)}\in\Gr_k(\RR^{n,n-1}),\\
F^{(2n-1-k)} = (F^{(k)})^\perp,\\
F^{(k)}\subset F^{(k+1)}, \text{ for } 1 \leq k \leq 2n-2\\
\end{array}\right\}.
\]
In other words $\mathcal F_{B''}$ is the space of all full flags of $\RR^{2n-1}$ for which each subspace of dimension less than half is isotropic and each subspace of dimension greater than half is the orthogonal space to the isotropic subspace of complementary dimension. Note that all of the data specifying such a flag is contained in the subspaces of dimension less than half. Nonetheless, it is useful to keep track of the subspaces of dimension larger than half as well.
The space of transverse pairs is given by 
\[\mathcal{O}_{B''} = \left\{(F_1,F_2)\in\mathcal{F}_{B''} \times \mathcal F_{B''}:F_1^{(i)}+F_2^{(2n-1-i)}=\RR^{2n-1}\text{ for all }i\right\}.\]
Similarly to the above, a subset $\mathcal I \subset \{1, \ldots, 2n-1\}$ of indices specifies a flag manifold $\mathcal F_{\mathcal I}$ containing the flags of type $\mathcal I$ which obey the orthogonality rules above when applicable. The stabilizer of the standard flag of type $\mathcal I$ is the parabolic subgroup $P_{\mathcal I}$. Unlike above, $P_{\mathcal I} = P_{\overline{\mathcal I}}$, where $\overline{\mathcal I} = \mathcal I \cup \sigma \mathcal I$ denotes the symmetrization of $\mathcal I$ under the involution $\sigma:i \mapsto 2n-1-i$. Indeed all parabolic subgroups of $G''$ are conjugate to their opposites.
Hence, we will always assume $\mathcal I = \overline{\mathcal I}$ is symmetric. The maximal parabolic subgroups of $G''$ are of the form $P''_k := P_{\{k, 2n-1-k\}}$ for $1 \leq k \leq n-1$.
\end{example}

\begin{example}
Let\label{eg:tangent to grassmannian 4} $G'=\PSO(n,n)$. The subgroup $B'$ of upper triangular matrices in $G'$ is again an example of a Borel subgroup. 
Then the associated flag manifold $\mathcal F_{B'} = G/B'$ may be described as the space of flags $F$ of the form
\begin{align}\label{eqn:flagG'}
F^{(1)} \subset \cdots \subset F^{(n-1)} \subset F^{(n)}_+, F^{(n)}_- \subset F^{(n+1)} \subset \cdots \subset F^{(2n-1)},
\end{align}
where
\begin{itemize}
\item $F^{(k)} \in\Gr_k(\RR^{n,n})$ for all $1\leq k \leq 2n-1$, $k \neq n$,
\item $F^{(2n-k)} = (F^{(k)})^\perp$ for all $k\neq n$,
\item $F_+^{(n)}\in\Gr_n^+(\RR^{n,n})$ and $ F_-^{(n)}\in\Gr_n^-(\RR^{n,n})$.
\end{itemize}
As in the previous example, the information given in the flag $F$ is more than needed to specify the associated point of $\mathcal F_{B'}$. Indeed, the subspaces $F^{(1)} \subset \cdots \subset F^{(n-1)}$ entirely determine $F$. However, it will be useful to have notation for the other subspaces as well.

The other parabolic subgroups of $G'$ are each given by the stabilizer of an incomplete flag made up of a subset of the subspaces of~\eqref{eqn:flagG'}. Just as in the case $G'' = \SO(n,n-1)$, the parabolic subgroups of $G'$ are each conjugate to their opposites, and so it suffices to consider symmetric flags. 
However, we note one important difference between $G'$ and $G''$.
The stabilizer $P_n'^+ < G'$ of the positive isotropic $n$-plane $H_+ := \spa \{e_1, \ldots, e_n\}$ and the stabilizer and $P_n'^- < G'$ of the negative isotropic $n$-plane $H_- := \spa \{e_1, \ldots, e_{n-1}, e_{n+1}\}$ are each maximal parabolic subgroups of $G'$. Their intersection $P_n'^+ \cap P_n'^- = P_{n-1}'$ is the stabilizer of the isotropic $(n-1)$-plane $H_0 = \spa \{e_1, \ldots, e_{n-1} \}$, which is a parabolic subgroup, but not a maximal one.
\end{example}

\begin{remark}\label{rem:isotropic transverse}
 In the case $n$ is even, any two transverse isotropic $n$-planes have the same sign, hence $P_n'^+$ and $P_n'^-$ are each conjugate to their opposite parabolic subgroups. In the case $n$ is odd, however, any two transverse isotropic $n$-planes have opposite sign, hence any opposite parabolic subgroup to $P_n'^+$ is conjugate to $P_n'^-$ and vice versa. 
\end{remark}

\subsection{Affine charts for flag manifolds}\label{sec:affine chart}

The flag manifolds of $G = \PSL(d,\RR)$ admit natural affine coordinate charts, which will be useful for the computations  in Sections~\ref{sec:Anosov-repns} and~\ref{sec:main2}.

Let us start with the Grassmannian $\Gr_k(\RR^d) = G/P_k$. For any $Y\in\Gr_{d-k}(\RR^d)$, denote by $U_Y$ the space of $k$-planes which are transverse to $Y$:
\[U_Y:=\{X\in\Gr_k(\RR^d):X\cap Y=0\},\]
an open subset of $\Gr_k(\RR^d)$. Fix $X \in U_Y$. Then any linear map $\psi \in\Hom(X,Y)$ determines another element of $U_Y$, namely the \emph{graph of }$\psi$,
\[\mathcal{G}_\psi\:=\{x+\psi(x)\in\RR^d:x\in X\}.\]
Observe that $\mathcal G_\psi$ is also transverse to $Y$ since the decomposition $\RR^d = X + Y$ is a direct sum.
It is easy to verify that the map $\psi \mapsto \mathcal G_\psi$ is a homeomorphism 
\begin{align}\label{eqn:coords-Gr}
\Hom(X,Y) \cong U_Y.
\end{align}
This equips the chart $U_Y \subset \Gr_k(\RR^d)$ with a linear vector space structure, in which $X$ is the origin. 
This linear structure gives natural coordinates on the tangent space 
\begin{align}\label{eqn:tangent-Grassmannian}
T_X \Gr_k(\RR^d) = T_X U_Y \cong \Hom(X,Y).
\end{align}
Note that choosing a different basepoint, say $Z \in U_Y$, yields a different vector space structure $\Hom(Z,Y) \cong U_Y$, which differs from the first by an affine isomorphism. Hence, independent of basepoint, $U_Y$ is equipped with an affine structure and we call $U_Y$ an \emph{affine chart} of $\Gr_k(\RR^d)$. The affine charts $\{U_Y: Y \in \Gr_{d-k}\}$ cover $\Gr_k(\RR^d)$ and satisfy the invariance property that $gU_Y = U_{gY}$ for all $g \in G$ and $Y \in \Gr_{d-k}(\RR^d)$.

Next, consider the space $\mathcal F_\mathcal I$ of flags of type $\mathcal I = \{i_1, \ldots, i_p\} \subset \{1, \ldots, d\}$. Any flag transverse to a flag of $\mathcal F_{\mathcal I}$ has type $\sigma \mathcal I$, where recall that $\sigma$ denotes the involution $i \mapsto d-i$. Choose $Y \in \mathcal F_{\sigma \mathcal I}$ and let 
\[U_Y:= \{Z \in \mathcal F_{\mathcal I}: Z^{(i_k)} \cap Y^{(d-i_k)} = 0, \text{ for all } 1 \leq k \leq p\}.\]
Fix a basepoint $X \in U_Y$, and let $Z \in U_Y$ another point. Using the above recipe, we may realize each subspace $Z^{(i_k)}$ of $Z$ uniquely as the graph of a linear map $\psi_{i_k}: X^{(i_k)} \to Y^{(d-i_k)}$. 
Further the linear maps are related to one another as follows. Define subspaces $V_1 = X^{(i_1)}$, $V_k = X^{(i_k)} \cap Y^{(d-i_{k-1})}$ for all $1 < k < p$, and $V_p = Y^{(d-i_p)}$.
Note that the dimension of $V_k$ is $i_k - i_{k-1}$ and the subspaces form a direct sum decomposition 
\begin{align}\label{eqn:lumpy-decomp}
\RR^d = V_1 \oplus \cdots \oplus V_{p}
\end{align}
The $i_k$ subspace $X^{(i_k)}$ of $X$ is the direct sum $X^{(i_k)} = V_1 \oplus \cdots \oplus V_{k}$ and the $d-i_k$ subspace of $Y$ is the direct sum $Y^{(d-i_k)} = V_{k+1} \oplus \cdots \oplus V_{p}$. The condition that $Z^{(i_k)} \subset Z^{(i_\ell)}$ for $k < \ell$ is equivalent to the condition that for each $1 \leq i \leq k$ and each $\ell \leq j \leq p$, the projection to the $V_j$ factor of the restriction to $V_i$ is the same for $\psi_{i_\ell}$ as it is for $\psi_{i_k}$. Hence, $Z \in U_Y$ determines unique linear maps $\psi_{i,j}: V_i \to V_j$ for all $1 \leq i < j \leq p$, so that 
\[\psi_{i_k} = \bigoplus_{1 \leq i \leq k < j \leq p} \psi_{i,j}.\]
This gives a homeomorphism
\[ U_Y \cong \bigoplus_{1 \leq i < j \leq p}\Hom(V_i, V_j),\]
which equips $U_Y$ with a linear structure for which $X$ is the origin. This linear structure gives natural coordinates on the tangent space 
\[T_X \mathcal F_{\mathcal I} = T_X U_Y \cong \bigoplus_{1 \leq i  < j \leq p}\Hom(V_i, V_j).\]
As above, note that choosing a different basepoint, say $Z \in U_Y$, yields a different vector space structure on $U_Y$, which differs from the first by an affine isomorphism. Hence, independent of basepoint, $U_Y$ is equipped with an affine structure and we call $U_Y$ an \emph{affine chart} of $\mathcal F_{\mathcal I}$. The affine charts $\{U_Y: Y \in \mathcal F_{\sigma \mathcal I}\}$ cover $\mathcal F_{\mathcal I}$ and satisfy the invariance property that $gU_Y = U_{gY}$ for all $g \in G$ and $Y \in \mathcal F_{\sigma \mathcal I}$.

Let us remark on one special case of the above construction. If $\mathcal I = \{1, \ldots, d\}$, then $\mathcal F_{\mathcal I} = \mathcal F_B$ is the space of full flags. In this case $\sigma \mathcal I = \mathcal I$. Let $X,Y \in \mathcal F_B$ be any transverse pair of flags. Then the decomposition~\eqref{eqn:lumpy-decomp} takes the form
\[ \RR^d = L_1 \oplus \cdots \oplus L_d\]
where $L_i := X^{(i)} \cap Y^{(d-i+1)}$ are lines, so we use the letter $L$ rather than $V$. The linear coordinates above on $U_Y$ for which $X$ is the origin take the form 
\[ U_Y \cong \bigoplus_{1 \leq i  < j \leq d}\Hom(L_i, L_j)\]
and as before these coordinates give coordinates at the tangent space level:
\[T_X \mathcal F_B = T_X U_Y \cong \bigoplus_{1 \leq i  < j \leq d}\Hom(L_i, L_j).\]

Finally, we remark that the flag manifolds for $G' = \PO(n,n)$ and $G'' = \SO(n,n-1)$ do not admit affine coordinates as above. It will be natural for our purposes to embed those flag manifolds in flag manifolds of $G = \PSL(d,\RR)$ (for $d = 2n$ or $d = 2n-1$), and work in the above coordinates.
An important example is the following. The space $\Gr_{n-1}(\RR^{n,n})$ of isotropic $(n-1)$-planes is naturally a smoothly embedded submanifold of the Grassmannian $\Gr_{n-1}(\RR^{2n})$ of all $(n-1)$-planes. The tangent space $T_X \Gr_{n-1}(\RR^{n,n})$ is naturally a subspace of $T_X \Gr_{n-1}(\RR^{2n})$ and may be expressed in the coordinates~\eqref{eqn:tangent-Grassmannian}. In fact, $T_X \Gr_{n-1}(\RR^{n,n})$ corresponds to the homomorphisms $\psi \in \Hom(X,Y)$ which are anti-symmetric, in the sense that 
$\langle \psi(v), w \rangle_{n,n} = -\langle v, \psi(w) \rangle_{n,n}$ holds for all $v,w \in X$.
We conclude this section with a proposition that describes the fibers of the projections $\varpi^\pm: \Gr_n^\pm(\RR^{n,n}) \to \Gr_{n-1}(\RR^{n,n})$ of Proposition~\ref{prop:projections} in these coordinates.

\begin{proposition}\label{prop:isotropic tangent}
Let $X \in \Gr_{n-1}(\RR^{n,n})$ and let $M = \varpi^+(X) \in \Gr_n^+(\RR^{n,n})$. Let $Y\in\Gr_{n+1}(\RR^{n,n})$ be transverse to $X$. Then in the coordinates~\eqref{eqn:tangent-Grassmannian}, the tangent space $T_X \ell_M$ to the fiber $\ell_M = (\varpi^+)^{-1}(M)$ is given by the subspace
\begin{align*}
 \Hom(X, M \cap Y) \subset \Hom(X,Y).
\end{align*}
\end{proposition}

\begin{proof}
Since every $(n-1)$-dimensional subspace in $M$ is isotropic, it follows that $\ell_M$ is simply the Grassmannian of $(n-1)$-dimensional subspaces in $M$. In the coordinates~\eqref{eqn:coords-Gr}, the smaller Grassmannian $\Gr_{n-1}(M) \subset \Gr_{n-1}(\RR^{2n})$ identifies with the subspace $\Hom(X, M \cap Y) \subset \Hom(X,Y)$. 
\end{proof}

\section{Hitchin representations and positivity}\label{sec:3}

Throughout the paper, we fix a closed surface $S$ of genus $g \geq 2$ and denote by $\Gamma = \pi_1 S$ the fundamental group and throughout this section, let $G$ be an adjoint, real split, semi-simple Lie group. 
In this section, we recall what it means for a representation $\Gamma \to G$ to be in the $G$-Hitchin component (Section~\ref{sec:Hitchin}) and explain the important positivity property that such representations satisfy. 
This positivity property was studied by Fock-Goncharov and is based on Lusztig's notion of positivity in~$G$. It will be used to obtain one key ingredient, namely Corollary~\ref{cor:transversality}, for the proof of Theorem~\ref{thm:main2}. A deep understanding of positivity is not needed for the rest of the paper, hence we will avoid giving the rather technical definition until Section~\ref{sec:Positivity} at the end of this section. 
A reader who is not familiar with positive representations may wish to treat Corollary~\ref{cor:transversality} (in Section~\ref{sec:Hitchin}) as a black box, and return to the details of Section~\ref{sec:Positivity}, which is entirely self-contained, after reading the rest of the paper.
Section~\ref{sec:Lie theory} gives some basic Lie theory prerequisites both for Section~\ref{sec:Positivity} and for Section~\ref{sec:Anosov-repns}.

\subsection{$G$-Hitchin representations}\label{sec:Hitchin}

 Let $\mathcal{X}(\Gamma,G):=\Hom(\Gamma,G)/G$, where $G$ acts on $\Hom(\Gamma,G)$ by conjugation.
For $G=\PSL(2,\RR)$, the discrete and faithful representations assemble into two connected components of $\Hom(\Gamma,\PSL(2,\RR))$, and their conjugacy classes form a union of two connected components in $\mathcal{X}(\Gamma,\PSL(2,\RR))$. A representation in either of these components, which are called the \emph{Teichm\"uller components}, corresponds to an oriented hyperbolic structure on the surface $S$, and the orientation distinguishes the two components.
Let us further equip $S$ with an orientation. Then we call the corresponding component of $\mathcal{X}(\Gamma,G)$, \emph{the} Teichm\"uller component of $S$ (and we ignore the other component of discrete faithful representations).
The $G$-Hitchin component is a generalization of the Teichm\"uller component to the setting where $\PSL(2,\RR)$ is replaced with any adjoint, real split, semi-simple Lie group~$G$. 

Let $\mathfrak g$ denote the Lie algebra of $G$. Recall that a \emph{$3$-dimensional subalgebra (TDS)} of $\mathfrak{g}$ is a Lie subalgebra that is isomorphic to $\mathfrak{s}\mathfrak{l}(2,\RR)$. A TDS $\mathfrak{h}\subset\mathfrak{g}$ is called \emph{principal} if every non-zero element $X\in\mathfrak{h}$ is regular, i.e. the dimension of the centralizer of $X$ is minimal among the centralizers of all elements in~$\mathfrak{g}$. By work of Kostant~\cite{Kos59}, $\mathfrak{g}$ contains a principal TDS, and any two principal TDS's are conjugate by an automorphism of $G$. 
Let
\[\tau_G:\PSL(2,\RR)\to G\] 
be a faithful homomorphism whose image is a subgroup of $G$ whose Lie algebra is a principal TDS in~$\mathfrak{g}$. This determines the map
\begin{eqnarray*}
i_G:\mathcal{X}(\Gamma,\PSL(2,\RR))&\to&\mathcal{X}(\Gamma,G)\\
{}[\rho]&\mapsto&[\tau_G\circ\rho].
\end{eqnarray*}
The component of $\mathcal{X}(\Gamma,G)$ containing the image of the Teichm\"uller component was studied by Hitchin~\cite{Hitchin}.
\begin{definition}
The connected component of $\mathcal{X}(\Gamma,G)$ containing the image of the Teichm\"uller component under $i_G$ is called the \emph{$G$-Hitchin component}. A representation whose conjugacy class lies in the $G$-Hitchin component is called a \emph{$G$-Hitchin representation}.
\end{definition}

Note that if $G=\PSL(2,\RR)$, the $G$-Hitchin component is exactly one of the Teichm\"uller components. If there is no ambiguity, we will sometimes refer to a $G$-Hitchin representation simply as a Hitchin representation.

\begin{remark}\label{rem:multiple Hitchin}
The Hitchin component is not quite well-defined. It depends on our choice of orientation on $S$ and also on a conjugacy class of homomorphism $\tau_G$ as above, of which there are finitely many corresponding to the outer automorphism group of $G$. Differing choices may give  distinct Hitchin components of $\mathcal X(\Gamma, G)$ which are mapped isomorphically to one another by pre and/or post composition by outer automorphisms. 	
\end{remark}

\begin{example}\label{eg:HitchinPSLd}
Consider $G = \PSL(d,\RR)$. Then $\tau_G: \PSL(2,\RR) \to \PSL(d,\RR)$ is the irreducible representation, unique up to automorphism of $\PSL(d,\RR)$, obtained from the action of $\SL(2,\RR)$ on the $(d-1)^{st}$ symmetric tensor power $\bigotimes_{sym}^{(d-1)}\RR^2 \cong \RR^d$ of $\RR^2$. It is easy to check that for $h \in \PSL(2,\RR)$ non-trivial, $\tau_G(h)$ is regular, which in this context means simply that $\tau_G(h)$ is diagonalizable with distinct eigenvalues. More specifically, if the eigenvalues of $h$ are $\lambda, \lambda^{-1}$ (well-defined up to $\pm 1$), then the eigenvalues of $\tau_G(h)$ are $\lambda^{d-1}, \lambda^{d-3}, \ldots, \lambda^{-(d-3)}, \lambda^{-(d-1)}$ (also well-defined up to $\pm 1$).
The $G$-Hitchin representations are the continuous deformations in $\Hom(\Gamma,G)$ of $\tau_G\circ j:\Gamma\to G$, where $j:\Gamma\to\PSL(2,\RR)$ is in the Teichm\"uller component.

Let $\omega$ denote the area form on $\RR^2$. Then $\omega$ defines a natural bilinear form $b$ on the tensor power $\bigotimes^{(d-1)}\RR^2$, which may be defined on simple tensors by the formula:
\[b(u_1\otimes \cdots \otimes u_{d-1}, v_1\otimes \cdots \otimes v_{d-1}) = \omega(u_1, v_1) \cdots \omega(u_{d-1},v_{d-1}).\]
Restricting to the subspace $\bigotimes_{sym}^{(d-1)}\RR^2$ of symmetric tensors in $\bigotimes^{(d-1)}\RR^2$ gives a non-degenerate bilinear form which is 
\begin{itemize}
\item anti-symmetric if $d = 2n$ is even, or
\item symmetric, if $d = 2n-1$ is odd, with 
\begin{itemize}
\item signature $(n-1, n)$ if $n $ is even, or
\item signature $(n, n-1)$ if $n$ is odd.
\end{itemize}
\end{itemize}
The image of $\tau_G$ preserves $b$, hence when $d = 2n$ is even, $\tau_G(\PSL(2,\RR))$ is contained in a conjugate of $\mathsf{PSp}(2n,\RR)$ and when $d = 2n-1$ is odd (so that $\PSL(d,\RR) = \SL(d,\RR)$), $\tau_G(\PSL(2,\RR))$ is contained in a conjugate of $\SO(n,n-1)$.
\end{example}

\begin{example}\label{eg:Hitchin(n,n-1)}
Consider $G'' =\SO(n,n-1)$. Thinking of $G'' < G = \PSL(2n-1,\RR) = \SL(2n-1,\RR)$, we may assume the irreducible representation $\tau_G$ from Example~\ref{eg:HitchinPSLd} takes values in $G''$. Further, for each non-trivial element $h \in \PSL(2,\RR)$, $\tau_G(h)$ is regular as an element of $G''$. Hence we may take $\tau_{G''} = \tau_G$. Hence, the natural inclusion $G'' \hookrightarrow G$ induces an inclusion of the $G''$-Hitchin component into the $G$-Hitchin component.
\end{example}

\begin{example}\label{eg:Hitchin(n,n)}
Consider $G'=\PSO(n,n) < G = \PSL(2n,\RR)$. 
Given an orthogonal splitting $\RR^{n,n} = \RR^{n,n-1} \oplus \RR^{0,1}$, the action of $h \in \mathsf{Aut}(\RR^{n,n-1}) = \SO(n,n-1) = G''$ on $\RR^{n,n-1}$ extends to $\RR^{n,n}$ by acting trivially in the $\RR^{0,1}$ factor. 
We denote by $\iota_{n,n} : G'' \to G'$ the composition of the natural inclusion $\SO(n,n-1) \to \SO(n,n)$ with the projection $\SO(n,n) \to \PSO(n,n)$ and note that $\iota_{n,n}$ is injective since the action of $h \in \SO(n,n-1)$ on $\RR^{n,n}$ is never $-1$.

Let $\tau_{G'} = \iota_{n,n} \circ \tau_{G''}$, where $\tau_{G''}$ is as in Example~\ref{eg:Hitchin(n,n-1)}.
Then the image of $\tau_{G'}$ is a principle $\PSL(2,\RR)$ in $G'$. Indeed, for each non-trivial element $h \in \PSL(2,\RR)$, the centralizer of $\tau_{G'}(h)$ in $G'$ is a Cartan subgroup $A' < G'$. Note that if the eigenvalues of $h$ are $\lambda, \lambda^{-1}$ (well-defined up to $\pm 1$), then the eigenvalues of $\tau_{G'}(h)$ are
\[\lambda^{2(n-1)}, \lambda^{2(n-2)}, \ldots, \lambda^2,1,1, \lambda^{-2}, \ldots, \lambda^{-2(n-2)}, \lambda^{-2(n-1)},\] and the eigenvalue $1$ has multiplicity two. Hence $\tau_{G'}(h)$ is \emph{not} regular as an element of $G = \PSL(2n,\RR)$. However, $\tau_{G'}(h)$ is regular in $G'$, since the $1$ eigenspace has signature $(1,1)$ and decomposes into a sum of two null lines which are preserved by (a finite index subgroup of) the centralizer. 

Since $\tau_{G'} = \iota_{n,n} \circ \tau_{G''}$, the inclusion $\iota_{n,n}: G'' \to G'$ induces an inclusion of the $G''$-Hitchin component into the $G'$-Hitchin component. However, the inclusion $\iota_{2n}: G' \to G = \PSL(2n,\RR)$ does not map the $G'$-Hitchin component to the $G$-Hitchin component.
\end{example}

Labourie~\cite{labourie2006}, Guichard~\cite{Guichard}, and Fock-Goncharov \cite{FockGoncharov} established the following characterization of $G$-Hitchin representations. We fix both a hyperbolic metric and an orientation on the surface $S$. The boundary of the group $\partial \Gamma$ then identifies with the visual boundary of the universal cover $\widetilde S \cong \HH^2$ of~$S$. The orientation on $S$ induces an orientation on $\widetilde S$ which in turn induces a cyclic ordering on $\partial\Gamma\cong S^1$. Let $B < G$ denote a Borel subgroup of $G$ and $\mathcal F_B = G/B$ the corresponding flag manifold. 

\begin{theorem}[Labourie, Guichard, Fock-Goncharov]\label{thm: Fock-Goncharov}
Let $\rho: \Gamma \to G$ be a representation. Then $\rho$ is a $G$-Hitchin representation if and only if there exists a continuous $\rho$-equivariant curve $\xi: \partial \Gamma \to \mathcal F_B$ which sends positive triples in $\partial \Gamma$ to positive triples in $\mathcal F_B$.
\end{theorem}

The curve $\xi: \partial \Gamma \to \mathcal F_B$ is called a \emph{positive curve} and turns out to be the same as the Anosov limit curve for $\rho$, see Theorem~\ref{thm:HitchinAnosov} in Section~\ref{sec:Anosov}.
We delay discussion of positivity until Section~\ref{sec:Positivity}, whose main purpose is to prove a transversality statement, Proposition~\ref{prop:SO(n,n)transverse}, about positive triples of flags in $\mathcal F_{B'} = G'/B'$ in the case $G' = \PSO(n,n)$. We remark that in this case, the positive curve of Theorem~\ref{thm: Fock-Goncharov} actually takes any distinct triple (not just a positive triple) to a positive triple of flags, see Appendix~\ref{app:negative triples}. The following result is a direct corollary of Proposition~\ref{prop:SO(n,n)transverse} and may be used as a black box in the rest of the paper.

\begin{corollary}\label{cor:transversality}
Let $G' = \PSO(n,n)$, let $\varrho:\Gamma\to G'$ be a $G'$-Hitchin representation. Then the $\varrho$-equivariant positive curve $\xi: \partial \Gamma \to \mathcal F_{B'}$ satisfies
\[\xi^{(n-1)}(x)+\left(\xi^{(n-1)}(z)\cap \xi^{(n+2)}(y)\right)+\xi_\pm^{(n)}(y)=\RR^{2n}.\]
for all pairwise distinct triples $(y,z,x)$ in $\partial\Gamma$. 
\end{corollary}

\subsection{Lie theory background}\label{sec:Lie theory}
Here we give some brief Lie Theory background needed in particular for Section~\ref{sec:Positivity}, but also for other parts later in the paper such as Section~\ref{sec:char}.

For any opposite pair of Borel subgroups $B^+,B^-\subset G$, let $U^\pm\subset B^\pm$ denote the unipotent radicals. The identity component of $B^+\cap B^-$, denoted $A$, is a maximal, connected, abelian Lie subgroup of $G$, i.e. a Cartan subgroup of $G$. Let $\mathfrak{g}$ be the Lie algebra of $G$, and let $\mathfrak{u}^\pm,\mathfrak{b}^\pm,\mathfrak{a}\subset\mathfrak{g}$ be the Lie subalgebras corresponding to the subgroups $U^\pm,B^\pm,A\subset G$. Then let $\mathfrak{a}^+\subset\mathfrak{a}$ denote the positive Weyl chamber so that the corresponding simple root spaces all lie in $\mathfrak{b}^+$, and let $\Delta$ denote the set of simple roots corresponding to $\mathfrak{a}^+$. For every simple root $\alpha:\mathfrak{a}\to\RR$, let $H_\alpha:\RR\to\mathfrak{a}$ denote the corresponding simple coroot.

Before continuing we give some concrete examples of the Lie theoretic objects defined above in the special cases of interest throughout this paper, namely for the Lie groups $G=\PSL(d,\RR)$, $G' = \PSO(n,n)$, and $G'' = \SO(n,n-1)$. Let $\delta_{i,j}=\delta_{i,j;d}$ denote the $d\times d$ square matrix with $1$ as its $(i,j)$-entry and all other entries are $0$. We will also denote $\delta_i:=\delta_{i,i}$. 

\begin{example}\label{eg:SL(n,R)}
Let $G=\PSL(d,\RR)$. Then the Lie algebra $\mathfrak{g} = \mathfrak{psl}(d,\RR)$ is the set of traceless $d\times d$ real-valued matrices. Let $B^+ < G$ (resp. $B^-<G$) be the subgroup of upper (resp. lower) triangular matrices in $G$, and let $U^\pm<B^\pm$ be the subgroups whose diagonal entries are $1$. Then $(B^+,B^-)$ is an opposite pair of Borel subgroups in $G$, $U^\pm$ is the unipotent radical of $B^\pm$, and $A$ is the set of diagonal matrices in $G$ with positive diagonal entries. The abelian Lie algebra $\mathfrak{a} \subset \mathfrak g$ is the space of traceless diagonal matrices, and $\mathfrak{a}^+$ is the subset of $\mathfrak{a}$ consisting of matrices whose diagonal entries are in weakly decreasing order going down the diagonal. The simple roots are $\Delta=\{\alpha_1,\dots,\alpha_{n-1}\}$ where $\alpha_i:\diag(a_1,\dots,a_n)\mapsto a_i-a_{i+1}$, and the corresponding co-roots are $H_{\alpha_i}(t)=t(\delta_i-\delta_{i+1})$. 
\end{example}

\begin{example}\label{eg:SO(n,n)}
Let $G'=\PSO(n,n) < \PSL(2n,\RR) = G$. The Lie algebra is given by
\[\mathfrak g' = \mathfrak{pso}(n,n):=\{X\in\mathfrak{sl}(2n,\RR):X^T\cdot J_{2n}+J_{2n}\cdot X=0\}.\] 
Again, let $B'^+ < G'$ (resp. $B'^- < G$) be the subgroup of upper (resp. lower) triangular matrices in $G'$ and let $U'^\pm$ be the subgroup of $B'^\pm$ whose diagonal entries are $1$. Then $(B'^+,B'^-)$ is an opposite pair of Borel subgroups, $U'^\pm<B'^\pm$ is the unipotent radical, and
\[A'=\left\{\diag\left(a_1,\dots,a_n,\frac{1}{a_n},\dots,\frac{1}{a_1}\right):a_i>0\right\}\]
is a Cartan subgroup of $G'$.
The Lie algebra of $A'$ is 
\[\mathfrak{a}'=\left\{\diag\left(a_1,\dots,a_n,-a_n,\dots,-a_1\right):a_i\in\RR\right\},\] and the positive Weyl chamber is 
\[\mathfrak{a}'^+=\left\{
\diag(a_1,\dots,a_n,-a_n,\dots,-a_1)\in\mathfrak{a}:a_1\geq\dots\geq a_n,-a_n\geq\dots\geq -a_1\right\},\]
where note that $a_n$ may be either positive, negative, or zero. The simple roots are $\Delta'=\{\alpha'_1,\dots,\alpha'_n\}$ where 
\[\alpha'_i:\diag(a_1,\dots,-a_1)\mapsto 
\left\{\begin{array}{ll}
a_i-a_{i+1}&\text{ if }i=1,\dots,n-1\\
a_{n-1}+a_n&\text{ if }i=n
\end{array}\right.,\] 
and the corresponding co-roots are
\[H_{\alpha'_i}(t)=\left\{\begin{array}{ll}
t(\delta_i-\delta_{i+1}+\delta_{2n-i}-\delta_{2n+1-i})&\text{ if }i=1,\dots,n-1\\
t(\delta_{n-1}+\delta_n-\delta_{n+1}-\delta_{n+2})&\text{ if }i=n.
\end{array}\right..\] 
\end{example}

\begin{example}\label{eg:SO(n,n-1)}
Let $G''=\SO(n,n-1) < \SL(2n-1, \RR) = \PSL(2n-1, \RR)$. The Lie algebra is given by
\[\mathfrak g'' = \mathfrak{so}(n,n-1):=\{X\in\mathfrak{psl}(2n-1,\RR):X^T\cdot J_{2n-1}+J_{2n-1}\cdot X=0\}.\]
Let $B''^+< G$ (resp. $B''^- < G$) be the subgroup of upper (resp. lower) triangular matrices in $G''$, and let $U''^\pm< B''^\pm$ be the subgroups whose diagonal entries are $1$. As before, $(B''^+,B''^-)$ is an opposite pair of Borel subgroups in $G''$, $U''^\pm$ is the unipotent radical of $B''^\pm$, and the Cartan subgroup is given by:
\[A''=\left\{\diag\left(a_1,\dots,a_{n-1},1,\frac{1}{a_{n-1}},\dots,\frac{1}{a_1}\right):a_i>0\right\}.\]
Then the Lie algebra of $A''$ is 
\[\mathfrak{a}''=\left\{\diag\left(a_1,\dots,a_{n-1},0,-a_{n-1},\dots,-a_1\right):a_i\in\RR\right\},\] and the positive Weyl chamber $\mathfrak{a}''^+$ is again the subset of $\mathfrak{a}''$ whose entries are in weakly decreasing order, going down the diagonal. The simple roots are $\Delta''=\{\alpha_1'',\dots,\alpha''_{n-1}\}$ where 
\[\alpha''_i:\diag(a_1,\dots,-a_1)\mapsto 
\left\{\begin{array}{ll}
a_i-a_{i+1}&\text{ if }i=1,\dots,n-2\\
a_{n-1}&\text{ if }i=n-1
\end{array}\right.,\] 
and the corresponding co-roots are give by $H_{\alpha''_i}(t)=t(\delta_i-\delta_{i+1}+\delta_{2n-1-i}-\delta_{2n-i})$. 
\end{example}
\subsection{Positivity}\label{sec:Positivity}
In this section, we recall Lusztig's notion of positivity in an adjoint, real split, semi-simple Lie group $G$, and give some of the basic properties. For more details, refer to Fock-Goncharov \cite{FockGoncharov}, Lusztig \cite{Lusztig}, or Guichard-Wienhard~\cite{GW16}. The theory is easiest to understand in the context of $G = \PSL(d,\RR)$; indeed this is usually the main example given in an introduction to the topic. However, since our goal is Corollary~\ref{cor:transversality}, we will focus here on the lesser known case of $G' = \PSO(n,n)$.

Recall that a \emph{$3$-dimensional subalgebra (TDS)} of $\mathfrak{g}$ is a Lie subalgebra that is isomorphic to $\mathfrak{s}\mathfrak{l}(2,\RR)$. We begin with the following standard fact. 

\begin{proposition}\label{prop:Kostant}
For every $\alpha\in\Delta$, there are linear maps $X_\alpha^+:\RR\to\mathfrak{u}^{+}$, $X_\alpha^-:\RR\to\mathfrak{u}^{-}$ so that
\begin{align}\label{eqn:tds}
[H_\alpha(1),X_\alpha^+(1)] &= 2X_\alpha^+(1),\\ 
[H_\alpha(1),X_\alpha^-(1)] &= -2X_\alpha^-(1), \text{ and }\nonumber\\
[X_\alpha^+(1),X_\alpha^-(1)] &=H_\alpha(1).\nonumber
\end{align}
In particular, $\{H_\alpha(t)+X_\alpha^+(a)+X_\alpha^-(b)\in\mathfrak{g}:a,b,t\in\RR\}$ is a TDS.
\end{proposition}

This motivates the following definition.

\begin{definition}For any $\alpha\in\Delta$, let $x^\pm_\alpha:=\exp\circ X^\pm_\alpha$. The data 
\[\big(B^+,B^-,\{x_\alpha^+\}_{\alpha\in\Delta},\{x_\alpha^-\}_{\alpha\in\Delta}\big)\] 
is a \emph{pinning} of $G$.
\end{definition}

\begin{example}\label{eg:pinning}
Let $G'=\PSO(n,n)$. Choose $B'^\pm$ as in Example~\ref{eg:SO(n,n)}, with $\Delta'=\{\alpha'_1,\dots,\alpha'_n\}$ the corresponding set of simple roots. Then, for all $i=1,\dots,n-1$, define
\begin{align*}
X_{\alpha'_i}^+(t) &=t(\delta_{i,i+1}-\delta_{2n-i,2n+1-i}),\\
 X_{\alpha'_i}^-(t)&=t(\delta_{i+1,i}-\delta_{2n+1-i,2n-i})
 \end{align*}
and for $i= n$, define
\begin{align*}
X_{\alpha'_n}^+(t) &=t(\delta_{n-1,n+1}-\delta_{n,n+2}),\\
X_{\alpha'_n}^-(t) &=t(\delta_{n+1,n-1}-\delta_{n+2,n}).
\end{align*} 
For all $i = 1, \ldots, n$, it is elementary to check that~\eqref{eqn:tds} holds with $\alpha = \alpha'_i$. Let 
\[x_{\alpha'_i}^\pm(t)= \exp(X_{\alpha'_i}^\pm(t)) = \Id_{2n}+X_{\alpha'_i}^\pm(t)\]
where here $\Id_{2n}$ denotes the $(2n)\times (2n)$ identity matrix. 
Then, the data $(B'^+,B'^-,\{x_{\alpha'_i}^+\}_{i=1}^n,\{x_{\alpha'_i}^-\}_{i=1}^n)$ is an example of a pinning of $G'$.
\end{example}

Choose a pinning $(B^+,B^-,\{x_{\alpha}^+\}_{\alpha\in\Delta},\{x_{\alpha}^-\}_{\alpha\in\Delta})$ of $G$, and let $\mathfrak{a}^+\subset\mathfrak{a}$ be the positive Weyl chamber and $\Delta$ the simple roots determined by $B^+, B^-$. For any $\alpha\in\Delta$, let $s_\alpha\in\GL(\mathfrak{a})$ be the reflection about the kernel of $\alpha$ (using the Killing form restricted to $\mathfrak{a}$). Recall that the \emph{Weyl group} $W=W(\mathfrak{a})$ is the subgroup of $\GL(\mathfrak{a})$ generated by $Q:=\{s_\alpha:\alpha\in\Delta\}$. It is well-known that $W(\mathfrak{a})$ is a finite group, and that there is a unique element $w_0\in W(\mathfrak{a})$, usually called the \emph{longest word element}, so that $w_0(\mathfrak{a}^+)=-\mathfrak{a}^+$. Write $w_0=s_{\alpha_{i_1}}\cdots s_{\alpha_{i_m}}$ as a reduced word in $Q$, and define the subset $U^+_{>0}$ by
\begin{align}
U^+_{>0}&:=\{x_{\alpha_{i_1}}^+(t_1)\cdot\dots\cdot x_{\alpha_{i_m}}^+(t_m):t_i>0 \ \forall i=1,\dots,m\} \label{eqn:positive}
\end{align}
The subset $U^+_{>0}\subset U^+$ does not depend on the choice of reduced word representative for $w_0$. In fact $U^+_{> 0}$ is a semi-group (although this is not obvious).

\begin{example}\label{eg:Mkl}
Let $G'=\PSO(n,n)$, and let $(B'^+,B'^-,\{x_{\alpha'_i}^+\}_{i=1}^n,\{x_{\alpha'_i}^-\}_{i=1}^n)$ be the pinning described in Example~\ref{eg:pinning}. The unipotent radicals $U'^\pm\subset B'^\pm$, the corresponding positive Weyl chamber $\mathfrak{a}'^+\subset\mathfrak{a}'$ and the simple roots $\Delta'=\{\alpha'_1,\dots,\alpha'_{n}\}$ are as described in Example~\ref{eg:SO(n,n)}. To simplify notation, let $s_i:=s_{\alpha'_i}$. Then define $\mu_1:=s_{n-1}\cdot s_n\in W(\mathfrak{a}')$, and for all $k=2,\dots,n-1$, define
\[\mu_k:=s_{n-k}\cdot\mu_{k-1}\cdot s_{n-k}.\] 
Then $\mu_1\cdot\mu_2\cdot\dots\cdot \mu_{n-1}$ is a reduced word expression of $w_0$.
Using this expression, one may describe the positive elements $U'^+_{>0}$ of $U'^+$.
For example, if $n=2$, then $w_0 = s_1 s_2$ and hence a typical element of $U'^+_{>0}$ has the form:
\begin{align*}
x^+_{\alpha'_1}(t_1)x^+_{\alpha'_2}(t_2) &= \left(\begin{array}{cccc}
1&t_1&0&0\\
0&1&0&0\\
0&0&1&-t_1\\
0&0&0&1
\end{array}\right)
\left(\begin{array}{cccc}
1&0&t_2&0\\
0&1&0&-t_2\\
0&0&1&0\\
0&0&0&1
\end{array}\right)\\
&= \left(\begin{array}{cccc}
1&t_1&t_2&-t_1t_2\\
0&1&0&-t_2\\
0&0&1&-t_1\\
0&0&0&1
\end{array}\right).
\end{align*}
For larger $n$, we give an inductive formula describing $U'^+_{>0}$ in Appendix \ref{app:matrices}, but the formula is somewhat messy. Luckily, we will be able to avoid working with an explicit description of $U'^+_{>0}$.
\end{example}

Next, we give the definition of positive triple of flags.
\begin{definition}\label{def:positive}
Let $F^+\in\mathcal{F}_B(G)$ (resp. $F^-\in\mathcal{F}_B(G)$) be the flag stabilized by $B^+$ (resp. by $B^-)$. A triple of flags $(F_1,F_2,F_3)\in\mathcal{F}_B(G)^3$ is called \emph{positive} if there is some $g\in G$ so that $g\cdot (F_1,F_2,F_3)=(F^+,u\cdot F^-, F^-)$, for some $u\in U^+_{>0}$.
\end{definition}

\begin{remark}\label{rem: multiple pinnings}
The notion of positivity in Definition~\ref{def:positive} depends on a choice of pinning. Any two conjugate pinnings give the same notion of positivity. However, in this setting it is not always the case that two different pinnings are conjugate (this is exactly the same subtlety, often ignored, that leads to multiple isomorphic Hitchin components in Remark~\ref{rem:multiple Hitchin}). In order for Theorem~\ref{thm: Fock-Goncharov} to hold as stated, one must choose a pinning for $G$ which is compatible with the  choice of representation $\tau_G$ defining the notion of $G$-Hitchin representation. 
On the other hand, since any two pinnings differ by some automorphism of $G$, we may, after applying such an automorphism, assume that a $G$-Hitchin representation satisfies Theorem~\ref{thm: Fock-Goncharov} for the notion of positive determined by any particular pinning we chose to work with. \end{remark}

Note that if $(F_1,F_2,F_3)$ is a positive triple, then in particular the three flags $\{F_1,F_2,F_3\}$ are pairwise transverse. We now prove a stronger transversality result, which is the main technical result of this section. Let $G' = \PSO(n,n)$ and let $B' < G'$ be a Borel subgroup with $\mathcal F_{B'} = G'/B'$ the associated flag manifold. Recall from Example \ref{eg:tangent to grassmannian 4} that an element $F \in \mathcal F_{B'}$ may be regarded as a flag 
\begin{align*}
F^{(1)} \subset \cdots \subset F^{(n-1)} \subset F^{(n)}_+, F^{(n)}_- \subset F^{(n+1)} \subset \cdots \subset F^{(2n-1)},
\end{align*}
where $F^{(i)} \in \Gr_{i}(\RR^{n,n})$ is a null $i$-plane and $F^{(2n-i)} = (F^{(i)})^\perp$ for $1 \leq i \leq n-1$ and $F^{(n)}_+ \in \Gr_{n}^+(\RR^{n,n})$ (resp. $F^{(n)}_- \in \Gr_{n}^-(\RR^{n,n})$) is the unique null $n$-plane which contains $F^{(n-1)}$ and which has positive signature $\tau(F^{(n)}_+) = 1$ (resp. has negative signature $\tau(F^{(n)}_-) = -$). See Section~\ref{sec:grassmannian}. 

\begin{proposition}\label{prop:SO(n,n)transverse}
Let $G' = \PSO(n,n)$ and let $B' < G'$ be a Borel subgroup with $\mathcal F_{B'} = G'/B'$ the associated flag manifold. Then for any positive triple $(Y,Z,X) \in (\mathcal{F}_{B'})^3$, we have:
\begin{align}
X^{(n-1)}+\left(Z^{(n-1)}\cap Y^{(n+2)}\right)+Y_+^{(n)} &=\RR^{2n}, \text{ and } \label{eqn:transversality+}\\
X^{(n-1)}+\left(Z^{(n-1)}\cap Y^{(n+2)}\right)+Y_-^{(n)} &=\RR^{2n}.\label{eqn:transversality-}
\end{align}
\end{proposition}

\begin{proof}  
By Remark~\ref{rem: multiple pinnings}, we may work with any pinning which is convenient. We choose the pinning $(B'^+,B'^-,\{x_{\alpha'_i}^+\}_{i=1}^n,\{x_{\alpha'_i}^-\}_{i=1}^n)$ of Example~\ref{eg:pinning}.
We will prove~\eqref{eqn:transversality+} directly. The other statement~\eqref{eqn:transversality-} is equivalent. To see this, consider the orientation reversing element $g \in \PO(n,n) \setminus \PSO(n,n)$ which exchanges the $n^{th}$ and $(n+1)^{th}$ basis vectors and fixes the other basis vectors. Then, on the one hand, $g$ flips the sign of the isotropic $n$-planes. On the other hand $g$ exchanges the roots $\alpha'_{n-1}$ and $\alpha'_n$ and exchanges the elements $x^+_{\alpha_{n-1}'}(t)$ and $x^+_{\alpha_{n}'}(t)$, leaving all other one parameter subgroups $x_{\alpha_i'}(t)$ of the pinning pointwise fixed. Noting that $x^+_{\alpha_{n-1}'}(t)$ and $x^+_{\alpha_{n}'}(s)$ commute for any $s,t > 0$, we observe that the action of $g$ fixes each point of the set $U'^+_{> 0}$ of Example~\ref{eg:Mkl}. Hence $g$ takes positive triples of flags to positive triples of flags but exchanges the positive and negative isotropic $n$-planes of each flag. Hence~\eqref{eqn:transversality+} implies~\eqref{eqn:transversality-}.

We now prove~\eqref{eqn:transversality+}.
Since the transversality condition~\eqref{eqn:transversality+} is unchanged by multiplication by $g \in G'$, we may assume that $Y = F^+$ and $X = F^-$ are the flags fixed by $B'^+$ and $B'^-$ respectively, and that $Z = u F^-$ for $u \in U'^+_{> 0}$. The element $u$ has the form
\begin{align}\label{eqn:u}
u= x^+_{\alpha_{i_1}}(t_1)\cdot\dots\cdot x^+_{\alpha_{i_m}}(t_m),
\end{align} 
which is difficult to work with directly. We will use a conjugation trick to simplify the proof.

The Cartan subgroup $A' < G'$ stabilizes both $Y$ and $X$, hence for $a \in A'$:
\[a(Y,Z,X) = a(F^+, u F^-, F^-) = (F^+, aua^{-1} F^-, F^-).\]
We shall use a carefully chosen (path of such) element(s) to simplify the situation. Observe that if $1 \leq i \leq n-1$, then
\begin{align*}
ax^+_{\alpha'_i}(t)a^{-1} &= a(\id_{2n} + X^+_{\alpha'_i}(t))a^{-1}\\
&= \id_{2n} + t\left(\frac{a_i}{a_{i+1}}\delta_{i,i+1} - \frac{a_{2n-i}}{a_{2n+1-i}}\delta_{2n-i,2n+1-i}\right)\\
&=  x^+_{\alpha'_i}\left(\frac{a_i}{a_{i+1}} t\right).
\end{align*}
where here $a = \diag\left(a_1,\dots,a_n,\frac{1}{a_n},\dots,\frac{1}{a_1}\right)$.
Similarly, if $i = n$,
\begin{align*}
ax^+_{\alpha'_n}(t)a^{-1} &= a(\id_{2n} + X^+_{\alpha'_n}(t))a^{-1}\\
&= \id_{2n} + t\left(\frac{a_{n-1}}{a_{n+1}}\delta_{n-1,n+1} - \frac{a_{n}}{a_{n+2}}\delta_{n,n+2}\right)\\
&=  x^+_{\alpha'_n}\left(a_na_{n-1} t\right).
\end{align*}
In particular, $aua^{-1}\in U^+_{>0}$. Also, observe that by choosing $a\in A'$ so that $a_1 << a_{2} << \cdots << a_n \leq 1$, we can make each of the finitely many terms $ax^+_{\alpha'_{i}}(t)a^{-1}$ of~\eqref{eqn:u} arbitrarily close to the identity. 
In fact, we may define a path $s \mapsto a^s \in A'$ so that $a^s u (a^{s})^{-1}$ smoothly converges to the identity as $s \to 0$, as follows.
For $s > 0$, let $a^s = \diag(s^n, s^{n-1}, \ldots, s, 1, 1, s^{-1}, \ldots, s^{-n})$.
Then for all $1 \leq i \leq n$,
\begin{align*}
a^s x^+_{\alpha'_i}(t) (a^s)^{-1}  &= x^+_{\alpha'_i}(st),
\end{align*}
and hence
\begin{align*}
u_s := a^s u (a^s)^{-1} &= x_{\alpha_{i_1}}^+(s t_1)\cdot\dots\cdot x_{\alpha_{i_m}}^+(s t_m)
\end{align*}
is a path smoothly converging to the identity element in $G'$ as $s \to 0$. The tangent vector to this path is
\begin{align}\label{eqn:du}
\left.\frac{\D}{\D s}\right|_{s = 0} u_s = X^+_{\alpha'_1}(r_1) + \cdots + X^+_{\alpha'_n}(r_n)
\end{align}
where here $r_j = \sum_{i_k = j} t_k > 0$ for all $1 \leq j \leq n$.

It is sufficient to show that~\eqref{eqn:transversality+} holds to first order for the path $(Y,Z,X) = (Y, Z_s, X) :=  (F^+, u_sF^-, F^-)$, as this will mean that~\eqref{eqn:transversality+} will hold for all $s > 0$ sufficiently small, and hence for $s=1$. We must simply show that the tangent vector to the path $s \mapsto V_s := \left(u_sX^{(n-1)}\right) \cap Y^{(n+2)}$ in $\Gr_1(\RR^{2n}) = \PP(\RR^{2n})$ is transverse to the hyperplane $\PP\left(X^{(n-1)} \oplus Y_+^{(n)}\right)$.
This is straightforward in coordinates:
\begin{align*}
Y_+^{(n)} &= \RR e_1 + \cdots + \RR e_n, \text{ and }\\
X^{(n-1)} &= \RR e_{n+2} + \cdots + \RR e_{2n}, \text{ hence }\\
X^{(n-1)} \oplus Y_+^{(n)} &= \RR e_1 + \ldots + \RR e_n + \RR e_{n+2} + \cdots + \RR e_{2n}.\\
\end{align*}
From~\eqref{eqn:du}, we read off that, 
\begin{align*}
\left.\frac{\D}{\D s} \right|_{s = 0} u_s e_{n+2} = - r_n e_n -  r_{n-1} e_{n+1}.
\end{align*}
Hence, in terms of the identification
\begin{align*}
T_{X^{(n-1)}} \Gr_{n-1}(\RR^{2n}) &= \Hom(X^{(n-1)}, Y^{(n+1)})\\ &= \bigoplus_{2n \geq  i \geq n+2 > j \geq 1} \Hom(\RR e_i, \RR e_j),
\end{align*}
from Section~\ref{sec:affine chart}, we see that $\left. \frac{\D}{\D s} \right|_{s = 0} u_sX^{(n-1)}$ has a non-trivial component in $\Hom(\RR e_{n+2}, \RR e_{n+1})$.
It then follows that our path $s \mapsto V_s$, based at $V_0 = X^{(n-1)} \cap Y^{(n+2)} = \RR e_{n+2}$, has tangent 
vector $\left. \frac{\D V_s}{\D s}\right|_{s=0} \in T_{V_0} \PP(\RR^{2n})$ which is transverse to the hyperplane $\PP\left(X^{(n-1)} \oplus Y_+^{(n)}\right)$.

\end{proof}

Corollary~\ref{cor:transversality} follows immediately from Proposition~\ref{prop:SO(n,n)transverse}.

\section{Anosov representations}\label{sec:Anosov-repns}

Here we review Anosov representations and prove several useful lemmas about them.
Labourie~\cite{labourie2006} introduced the notion of Anosov representation in order to characterize the good dynamical behavior of the representations in the $\PSL(d,\RR)$-Hitchin component. Guichard-Wienhard~\cite{GW} generalized the notion to the setting of representations of word hyperbolic groups in reductive Lie groups and developed the general theory in this setting.  
The quick review of Anosov representations presented here will focus on the more specialized setting of interest, namely representations from a surface group $\Gamma = \pi_1 S$ to an adjoint, real split, semisimple Lie group $G$. As above, there are three Lie groups of interest for our purposes, namely $G = \PSL_{d} \RR$, $G'=\PSO(n,n)$, and $G''=\SO(n,n-1)$.

\subsection{The definition}\label{sec:Anosov}

Throughout, we fix a hyperbolic metric on the surface $S$ and denote by $T^1 S$ the unit tangent bundle of $S$. The boundary $\partial \Gamma$ of the group $\Gamma$ identifies with the visual boundary of the universal cover $\widetilde S \cong \HH^2$ of~$S$.
We choose an orientation on $S$ which induces an orientation on $\widetilde S$ which in turn induces a cyclic ordering on $\partial\Gamma\cong S^1$. We identify the unit tangent bundle of $\widetilde S$ with the space of cyclically ordered triples in $\partial \Gamma$ in the usual way:
\begin{align*}T^1 \widetilde S =\{(y,z,x)\in\partial\Gamma^3:y< z< x < y\}.
\end{align*}
Specifically, if $y < z < x < y$ in $\partial \Gamma$, then there is a unique unit tangent vector $v$ based at a point $p$ of $\widetilde S$ so that $v$ is tangent to the geodesic $(y,x)$ connecting $y$ to $x$, $v$ points away from $y$ toward $x$, and the geodesic ray $[p,z)$ meets $(y,x)$ orthogonally.
The geodesic flow $\varphi_t$ on $T^1 S$ lifts to the geodesic flow $\widetilde \varphi_t$ on $T^1 \widetilde S$, which in these coordinates has the form $\varphi_t(y,z,x)=(y,z(t),x)$, where $z:\RR\to\partial\Gamma$ is a continuous, injective map so that $z(0)=z$, $\lim_{t\to\infty}z(t)=x$ and $\lim_{t\to-\infty}z(t)=y$. Although $\partial \Gamma$ does not have any canonical smooth structure, if $\partial \Gamma$ is endowed with the smooth structure induced by the hyperbolic structure on $\widetilde S$, then the function $z(t)$ is smooth. 
The geodesic flow for a different hyperbolic metric on $S$, written in the same coordinates, is simply a continuous reparameterization of $\varphi_t$, meaning that the flow lines are the same, but the function $z(t)$ is altered by an orientation-preserving homeomorphism of $\RR$. We ignore such subtleties and simply remark that the choice of hyperbolic metric has no meaningful effect on the coming definitions. 
 
The notion of Anosov representation depends on a (conjugacy class) of parabolic subgroup $P < G$. We restrict here to the case that the parabolic subgroup $P$ is conjugate to any opposite parabolic subgroup. 
 This will be the case in the settings of interest and it slightly simplifies the setup.  Recall the flag space $\mathcal{F}_P:=G/P$ defined in Section \ref{sec:flags}. There is a unique open $G$-orbit in the product $\mathcal F_P \times \mathcal F_P$, namely the subspace of transverse pairs, which we denote by $\mathcal{O} \subset \mathcal F_P \times \mathcal F_P$. 
Let $\rho:\Gamma\to G$ a representation. Associated to $\rho$ is the space
\[\mathcal{Y}_\rho:=(T^1\widetilde{S}\times\mathcal O)/\Gamma,\]
where the action on $\mathcal O \subset \mathcal{F}_P\times\mathcal{F}_P$ is the diagonal action by $\rho$ and the action on $T^1 \widetilde S$ is by deck translations. The smooth manifold $\mathcal{Y}_\rho$ is naturally a flat $G$-bundle over $T^1S$ whose fibers are isomorphic to $\mathcal O$ as $G$-sets. 

Now, let \[T^{\mathrm{v}} \mathcal Y_\rho = (T^1\widetilde{S}\times T \mathcal O)/\Gamma\] denote the vertical tangent bundle to $\mathcal Y_\rho$.
The local product structure on $\mathcal O \subset \mathcal F_P \times \mathcal F_P$ determines a splitting $T \mathcal O = E^+ \oplus E^-$ of $T \mathcal O$ into two isomorphic sub-bundles $E^+, E^-$. This splitting induces a splitting \[T^{\mathrm{v}}  \mathcal Y_\rho = \mathcal E_\rho^+ \oplus \mathcal E_\rho^-\] of the vertical tangent bundle into two sub-bundles $\mathcal E_\rho^+ ,\mathcal E_\rho^-$ over $\mathcal Y_\rho$.

The geodesic flow $\widetilde \varphi_t$ on $T^1 \widetilde S$ lifts to a flow, again denoted $\widetilde \varphi_t$ on the product bundle $T^1 \widetilde S \times \mathcal O$ by acting trivially in the second factor: $\widetilde \varphi_t(\nu, o) := (\widetilde \varphi_t \nu, o)$. The differential $\D \widetilde \varphi_t$ defines a lift of the flow $\widetilde \varphi_t$ to the vertical tangent bundle $T^1 \widetilde S \times T\mathcal O$, which is again trivial in the second factor and in particular preserves the product structure on each tangent space of $\mathcal O$. Both flows descend to the bundles $\mathcal Y_\rho$ and $T^{\mathrm{v}}\mathcal Y_\rho$ over $T^1 S$, and are denoted $\varphi_t$ and $\D \varphi_t $ respectively. Indeed, the flow $\varphi_t$ is simply the lift of the geodesic flow on $T^1 S$ to $\mathcal Y_\rho$ using the flat connection, and the flow $\D \varphi_t$ on $T^\mathrm{v} \mathcal Y_\rho$ is its differential.

\begin{definition}\label{def:Anosov}
The representation $\rho:\Gamma\to G$ is \emph{Anosov with respect to $P$} or alternatively, \emph{$P$-Anosov},  if there exists a continuous section $\sigma:T^1S\to\mathcal{Y}_\rho$ of $\mathcal Y_\rho$ that is parallel under the flow $\varphi_t$ and which satisfies the following.
\begin{enumerate}
\item The flow $\D \varphi_t$ \emph{expands} $\mathcal E^+_\rho$ along the section $\sigma(T^1 S)$: There exist constants $a,c\in\RR^+$ so that for any $\nu\in T^1S$ and any non-zero vector $f^+$ in the fiber of $\mathcal E^+_\rho$ over $\sigma(\nu)$,
\[ \| \D \varphi_t(f^+) \|_{\varphi_t \nu} \geq a e^{c t} \| f^+ \|_\nu. \]
\item The flow $\D \varphi_t$ \emph{contracts} $\mathcal E^-_\rho$ along the section $\sigma(T^1 S)$: There exist constants $b,d\in\RR^+$ so that for any $\nu\in T^1S$ and any non-zero vector $f^-$ in the fiber of $\mathcal E^-_\rho$ over $\sigma(\nu)$, 
\[ \| \D \varphi_t(f^-) \|_{\varphi_t \nu} \leq b e^{-d t} \| f^- \|_\nu. \]
\end{enumerate}
\end{definition}
In the above definition $\| \cdot \|$ is any continuously varying family of norms on the (fibers of the) vertical tangent bundle $T^{\mathrm{v}} \mathcal Y_\rho$. Since $T^1 S$ is compact, any two families of norms on $T^{\mathrm{v}} \mathcal Y_\rho$ are equivalent along the section $\sigma(T^1S)$ and the notion of $P$-Anosov does not depend on the choice of norm (although the constants $a,b,c,d$ do). We will often work in the universal cover, where such a family of norms $\| \cdot \|$ lifts to a family of norms, also denoted $\| \cdot \|$, on the product bundle $T^1 \widetilde S \times T \mathcal O$, which is $\rho$-equivariant, meaning $\| \D \rho(\gamma) f \|_{\gamma \cdot \nu} = \| f \|_\nu$ for any $\nu \in T^1 S$ and $f \in T\mathcal O$.
\begin{remark}
In Definition~\ref{def:Anosov}, the contraction condition~(2) follows from the expansion condition~(1) and vice versa, see~\cite{GW}. We will typically work only with condition~(1) here.
\end{remark}

In Definition~\ref{def:Anosov}, the section $\sigma$ is unique~\cite{labourie2006}, and is usually called the \emph{Anosov section}.
The data of the Anosov section can also be captured by a $\rho$-equivariant map to the flag space $\mathcal{F}_P$.

\begin{definition}
Let $\rho:\Gamma\to G$ be a representation, and $\xi:\partial\Gamma\to\mathcal{F}_P$ be a $\rho$-equivariant continuous map. We say $\xi$ is \emph{dynamics preserving} if for every $\gamma\in\Gamma\setminus\{\id\}$, $\xi$ maps the attracting fixed point $\gamma^+ \in \partial \Gamma$  of $\gamma$ to the unique attracting fixed point of $\rho(\gamma)$ in $\mathcal{F}_P$. 
\end{definition}

\begin{fact}[Labourie, Guichard-Wienhard]\label{thm:flag maps}
Let $\rho:\Gamma\to G$ be a $P$-Anosov representation. Then there is a continuous, $\rho$-equivariant, dynamics preserving map $\xi:\partial\Gamma\to\mathcal{F}_P$. Furthermore, the Anosov section $\sigma$ lifts to an equivariant map
\[\widetilde{\sigma}:T^1\widetilde{S}\to T^1\widetilde{S}\times \mathcal{F}_P\times\mathcal{F}_P,\]
which is given by the formula $\widetilde{\sigma}(y,z,x)=(y,z,x,\xi(x),\xi(y))$.
\end{fact}

The map $\xi$ given in Fact~\ref{thm:flag maps} is called the \emph{Anosov limit map} or \emph{Anosov boundary map} of $\rho$. The dynamics preserving and continuity properties ensure that such a map is necessarily unique. Since the (lift of the) Anosov section $\widetilde \sigma$ takes values in the transverse pairs $\mathcal O \subset \mathcal F_P \times \mathcal F_P$, the Anosov limit map $\xi$ is necessarily \emph{transverse}, meaning that for all $x,y\in\partial\Gamma$ distinct, $\xi(x)$ and $\xi(y)$ are transverse points of $\mathcal F_P$, i.e. $(\xi(x),\xi(y))\in \mathcal O$.

There are many examples of Anosov representations of surface groups, including both maximal representations and Hitchin representations, see~\cite{GW}. For recent examples of Anosov representations of right-angled Coxeter groups, see~\cite{DGK5, DGK6}.
Hitchin representations are the main examples of Anosov representations of concern in this paper.

\begin{theorem}[Labourie, Fock-Goncharov]\label{thm:HitchinAnosov}
Every $G$-Hitchin representation $\rho:\Gamma\to G$ is Anosov with respect to the Borel subgroup $B\subset G$, and the $\rho$-equivariant positive curve of Theorem~\ref{thm: Fock-Goncharov} is the Anosov limit map.
\end{theorem} 

One important property of Anosov representations is that the condition is stable under small deformation.

\begin{fact}[Labourie]\label{fact:open}
Let $\rho\in\Hom(\Gamma, G)$ be a $P$-Anosov representation. Then there is an open neighborhood $U\subset\Hom(\Gamma, G)$ of $\rho$ so that every representation in $U$ is also $P$-Anosov. 
\end{fact}

\subsection{$B$-Anosov representations in $\PSL(d,\RR)$}\label{subsec:BSLd}
Let $\rho:\Gamma\to\PSL(d,\RR)$ be Anosov with respect to the Borel subgroup $B \subset \PSL(d,\RR)$ and let $\xi:\partial\Gamma\to \mathcal{F}_B$ be the Anosov limit map. 
We follow Section~\ref{sec:affine chart} to obtain natural coordinates on the fibers of $\mathcal E_\rho^+$.
Let $x,y\in\partial\Gamma$ be distinct. 
For each $1 \leq i \leq d$, let 
\begin{align*}
L_i(x,y) &:=\xi^{(i)}(x)\cap\xi^{(n-i+1)}(y)
\end{align*}
which is a line, since the flags $\xi(x)$ and $\xi(y)$ are transverse (see Example \ref{eg:tangent to grassmannian 1}). The line decomposition $\RR^d = \bigoplus_{i=1}^d L_i(x,y)$ varies continuously as $x,y$ vary and is $\rho$-equivaraint, in the sense that $L_i(\gamma\cdot x,\gamma\cdot y) = \rho(\gamma)L_i(x,y)$ for all $\gamma \in \Gamma$, and distinct points $x,y \in \partial \Gamma$.
Next, for each $x \neq y$ in $\partial \Gamma$, the one-dimensional vector space
\[ \Hom(L_i,L_j)(x,y) := \Hom(L_i(x,y), L_j(x,y))\]
 may be regarded as a subspace of $\mathsf{End}(\RR^d)$ which varies continuously in $x,y$. This gives a $\rho$-equivariant decomposition of the product bundle $T^1 \widetilde S \times \mathsf{End}(\RR^d)$ which descends to the flat $\mathsf{End}(\RR^d)$ bundle over $T^1S$ associated to~$\rho$:
\[\Gamma \backslash (T^1 \widetilde S \times \mathsf{End}(\RR^d)) = \bigoplus_{i,j = 1}^d \Hom(L_i,L_j),\]
where here, by abuse, $\Hom(L_i,L_j)$ denotes the line bundle over $T^1 S$ whose fibers are locally $\Hom(L_i(x,y),L_j(x,y))$.
Next, let $x,y \in \partial \Gamma$ be distinct and let $U_{\xi(y)}$ be the affine chart for $\mathcal F_B$ defined in Section~\ref{sec:affine chart}. We have the identification 
$\bigoplus_{i < j} \Hom(L_i,L_j)(x,y) \xrightarrow[]{\simeq} U_{\xi(y)} \subset \mathcal F_B$ with the origin mapping to $\xi(x)$. In particular, we identify
\begin{align}T_{\xi(x)} \mathcal F_B &= \bigoplus_{i < j} \Hom(L_i,L_j)(x,y). \label{eqn:splitting}
\end{align}
This gives coordinates on $\mathcal E_\rho^+$ along the section $\sigma(T^1S)$: 
\[\sigma^* \mathcal E_\rho^+ = \bigoplus_{i < j} \Hom(L_i,L_j).\]
All of the splittings described above are invariant under the geodesic flow, and the Anosov expansion condition~(1) of Definition~\ref{def:Anosov} 
is satisfied if and only if it is satisfied on each line bundle factor of~\eqref{eqn:splitting}. Therefore condition~(1) of Definition~\ref{def:Anosov} 
is equivalent to the existence of constants $a,c > 0$ so that for any $i < j$, any $\nu \in T^1 S$, and any $f$ in the fiber $\Hom(L_i,L_j)_\nu$ above a point $\nu \in T^1S$ of the bundle $\Hom(L_i,L_j)$, we have
\[\|\D \varphi_t f\|_{\varphi_t\nu}   \geq a e^{c t} \| f \|_{\nu}\]
where here $\| \cdot \|$ is any continuous family of norms on the fibers of $\Hom(L_i,L_j)$ (for example, coming from the restriction of a continuous family of norms on the flat $\mathsf{End}(\RR^d)$ bundle associated to $\rho$).
Equivalently, in the lift to $T^1 \widetilde S$, the condition becomes: there exists $a,c > 0$ so that for any $y < z < x < y$ in $\partial \Gamma$ and any $f \in \Hom(L_i,L_j)(x,y)$,
 \begin{align} \|\D \varphi_t f\|_{\varphi_t(y,z,x)}  = \| f \|_{(y,z(t),x)} \geq a e^{c t} \| f \|_{(y,z,x)}, \label{eqn:expand}
 \end{align}
 where now $\| \cdot \|$ denotes a continuously varying, $\rho$-equivariant family of norms on $\Hom(L_i,L_j)(x,y)$.

Let us now prove a useful proposition.
\begin{proposition}\label{prop:ratio of norms}
Let $\rho \in \Hom(\Gamma, \PSL(d, \RR)$ be $B$-Anosov and define the bundles $\Hom(L_i, L_j)$ as above. Let $1\leq p\leq q<r\leq s\leq d$ be positive integers, with $(p,s) \neq (q,r)$.
 Then there are constants $A,C > 0$, so that for any $(y,z,x) \in T^1 \widetilde S$, where $y < z < x < y$ in $\partial \Gamma$, and any non-zero $f \in \Hom(L_p, L_s)(x,y)$ and $f' \in \Hom(L_{q}, L_{r})(x,y)$, and any $t > 0$ we have: 
\[\frac{\| f \|_{(y,z(t),x)}}{\|f'\|_{(y,z(t),x)}}\geq \frac{ \| f \| _{(y,z,x)}}{\| f'\|_{(y,z,x)}}Ae^{Ct}.\] 
\end{proposition}
For the above proposition, recall the notation $(y,z(t),x) := \varphi_t(y,z,x)$ and that the norms $\| \cdot \|$ on each $\Hom(L_i, L_j)(x,y)$ are a fixed family of $\rho$-equivariant norms varying continuously over $T^1 \widetilde S$.

\begin{proof}
Since $T^1 S$ is compact, all norms on any of the line bundles $\Hom(L_i, L_j)$ are equivalent. Hence, we may choose norms that are convenient to work with.
For each $k$, let $\| \cdot \|^{k,k+1}_{(y,z,x)}$ be a family of norms on the line bundle $\Hom(L_k, L_{k+1})$. Then, for any $i < j$, we may define a family of norms $\| \cdot \|^{i,j}_{(y,z,x)}$ on $\Hom(L_i, L_{j})$ as follows. For any $f \in \Hom(L_i,L_{j})(x,y)$, factor $f$ as a composition $f = f_i \circ f_{i+1} \circ \cdots \circ f_{j-1}$, where each $f_k \in \Hom(L_k, L_{k+1})(x,y)$ and define
\[ \| f\|^{i,j}_{(y,z,x)} := \prod_{k=i}^{j-1} \| f_k \|^{k,k+1}_{(y,z,x)}. \]
Note that, since the $\Hom(L_k, L_{k+1})$ are line bundles, the choice of the factorization of $f$ amounts to choosing scalars in each factor and does not affect the result, hence $\| \cdot \|^{i,j}_{(y,z,x)}$ is well-defined. We now prove the result using these norms.

Let $y < z < x < y$ in $\partial \Gamma$, and let $f \in \Hom(L_p, L_s)(x,y)$ and $f' \in \Hom(L_{q}, L_{r})(x,y)$ both be non-zero.
Next, factor $f$ as the composition 
\begin{align*}
f &= g \circ f' \circ h
\end{align*}
where $g \in \Hom(L_{r}, L_{s})$ and $h \in \Hom(L_p, L_{q})$.
Then, by the definition of the norm above we have that $\|f \|_{(y,z,x)} = \|g \|_{(y,z,x)} \| f' \|_{(y,z,x)} \| h \|_{(y,z,x)}$ at each point $(y,z,x) \in T^1 \widetilde S$. In the case that $p = q$, we may assume $\| h \|_{(y,z,x)} = 1$ constant. Similarly, if $r = s$, we may assume $\| g \|_{(y,z,x)} = 1$ constant.
Then:
\begin{align*}
\frac{\| f \|_{(y,z(t),x)}}{\|f'\|_{(y,z(t),x)}} &=  \|g \|_{(y,z(t),x)} \| h \|_{(y,z(t),x)}\\
& \geq \|g \|_{(y,z,x)} \|h \|_{(y,z,x)} A e^{Ct}
\end{align*}
where we use that  $p < q$ or $r < s$ so that at least one of $\|g \|_{(y,z(t),x)}$,  $\| h \|_{(y,z(t),x)}$ expands as in~\eqref{eqn:expand}, while the other may also expand or otherwise is constant by the discussion above. 
Hence
\begin{align*}
\frac{\| f \|_{(y,z(t),x)}}{\|f'\|_{(y,z(t),x)}} &\geq  \frac{\|g \|_{(y,z,x)} \|f'\|_{(y,z,x)}  \|h \|_{(y,z,x)}}{\|f'\|_{(y,z,x)}} Ae^{Ct}\\
&=  \frac{ \|f\|_{(y,z,x)}}{\| f'\|_{(y,z,x)}} A e^{Ct}.
\end{align*} 
\end{proof}

\subsection{Characterizations of Anosov in terms of Cartan and Lyapunov projections}\label{sec:char}
In this section, we discuss a recent characterization of Anosovness, due independently to Kapovich-Leeb-Porti~\cite{KLP2, KLPa, KLPb} (see also~\cite{KLPc} for a survey of this work) and Guichard-Gueritaud-Kassel-Wienhard~\cite{GGKW}, which will be essential for the main result. We follow~\cite{GGKW}. 
In this section we assume that $G$ is a semi-simple Lie group whose adjoint group $\Ad(G)$ is contained in the group of inner automorphisms of the complexification $\mathfrak g_\CC$ of the Lie algebra $\mathfrak g$. This assumption, which holds in particular for the three groups of primary concern here ($\PSL(d,\RR)$, $\PSO(n,n)$, and $\SO(n,n-1)$, will guarantee that the Cartan projection (defined below) is well-defined, see~\cite[Ch. 7]{Knapp}.

Let $G = K \exp(\mathfrak a^+) K$ be a Cartan decomposition of $G$, where here $K < G$ is a maximal compact subgroup and $\mathfrak a^+$ is a choice of closed positive Weyl chamber contained in a Cartan subalgebra $\mathfrak a$ of the Lie algebra $\mathfrak g$. Then each $g \in G$ may be factored as $g = k \exp(a) k'$, where $k,k' \in K$ and the element $a \in \mathfrak a^+$ is \emph{unique}.
 The associated \emph{Cartan projection} $\mu: G \to \mathfrak a^+$ is the map defined by $\mu(g) = a$. 
  Let $\lambda: G \to \mathfrak a^+$ denote the \emph{Lyapunov projection}, which satisfies $\lambda(g) = \lim_n \mu(g^n)/n$ (and which is defined independent of the choice of Cartan decomposition in the definition of $\mu$). Before stating the alternative characterizations of Anosov, let us give the relevant examples of the two projections. 
  
  \begin{remark}
Note that while the Cartan projection is well-defined for $\PSO(n,n)$, it is not well-defined for the index two supergroup $\PO(n,n)$. Indeed the decomposition $\PO(n,n) = K \exp(\mathfrak{a}^+) K$ holds, but the $\mathfrak a^+$ part is not unique. The reason is that there is an orientation reversing element in the maximal compact $K$ for $\PO(n,n)$ which preserves $\mathfrak a^+$ but acts non-trivially on it. 
  \end{remark}

\begin{example}\label{eg:Jordan projection1}
Let $G=\PSL(d,\RR)$. Recall from Example \ref{eg:SL(n,R)} that the Lie algebra of $G$ is $\mathfrak{psl}(d, \RR)$, the algebra of traceless $d \times d$ real-valued matrices, and we may choose its positive Weyl chamber $\mathfrak a^+$ to be the diagonal matrices of the form $\operatorname{diag}(a_1, \ldots, a_d)$ with $\sum a_i = 0$ and $a_i \geq a_{i+1}$ for all $1 \leq i \leq d-1$. Then for $g \in G$, the Lyapunov projection $\lambda(g) = \operatorname{diag}(\lambda_1(g), \ldots, \lambda_d(g))$ where the diagonal entries $\lambda_i(g)$ are the logarithms of the absolute value of the eigenvalues of $g$ listed in weakly decreasing order. The entries of the Cartan projection $\mu(g) = \operatorname{diag}(\mu_1(g), \ldots, \mu_d(g))$ are the singular values of $g$ listed in weakly decreasing order. 
The simple roots $\alpha_1, \ldots, \alpha_{d-1} \in \Delta$ measure the difference in consecutive singular values: $\alpha_i(\mu(g)) = \mu_i(g) - \mu_{i+1}(g)$. Similarly $\alpha_i(\lambda(g)) = \lambda_i(g) - \lambda_{i+1}(g)$.
\end{example}

\begin{example}\label{eg:Jordan projection}
Let $G' =\PSO(n,n)\subset G = \PSL(2n,\RR)$. Recall from Example \ref{eg:SO(n,n)} that the Cartan subalgebra $\mathfrak a'\subset\mathfrak{pso}(n,n)$ consists of all diagonal matrices of the form $\operatorname{diag}(a_1, \ldots, a_n, -a_n, \ldots, -a_1)$, and is thus realized as a subspace of the Cartan subalgebra $\mathfrak{a}\subset\mathfrak{psl}(n,\RR)$. The postive Weyl chamber $\mathfrak a'^+$, however, may \emph{not} be chosen as a subset of the positive Weyl chamber~$\mathfrak a^+$; observe that the restriction to $\mathfrak a'$ of the simple root $\alpha_n: \mathfrak a \to \RR$ given by $\alpha_n: \operatorname{diag}(a_1,\dots,a_{2n})\mapsto a_n-a_{n+1}$ is \emph{not} a root in $G'$.
Indeed, taking the simple roots $\Delta' = \{\alpha'_1, \ldots, \alpha'_n\}$ for $G'$ as in Example~\ref{eg:SO(n,n)}, the positive Weyl chamber $\mathfrak a'^+$ is the subset of diagonal matrices $\operatorname{diag}(a_1, \ldots, a_n, -a_n, \ldots, -a_1)$ for which $a_i \geq a_{i+1}$ for all $1 \leq i \leq n-2$, and $a_{n-1} \geq, a_n, -a_n$. In particular, for $a_1 \geq \cdots \geq a_n$, the two matrices 
\begin{align*}
\exp \diag(a_1, \ldots, a_{n-1},a_n, -a_n,-a_{n-1}, \ldots, -a_1),\\
\exp \diag(a_1, \ldots, a_{n-1},  -a_n,a_n, -a_{n-1}, \ldots, -a_1)
\end{align*} are not conjugate in $G'$, though they are conjugate in $G$.
For $g \in G'$, if $\lambda'(g) = \operatorname{diag}(\lambda'_1(g), \ldots, \lambda'_n(g), -\lambda'_n(g), \ldots, -\lambda'_1(g))$ is the Lyapunov projection of $g$ in $G'$ and $\lambda(g) = \operatorname{diag}(\lambda_1(g), \ldots, \lambda_{2n}(g))$ is the Lyapunov projection of $g$ in $G$, then 
\begin{itemize}
\item $\lambda'_i(g) = \lambda_i(g)=\lambda_{2n+1-i}(g)$ for all $1 \leq i \leq n-1$,  
\item $\lambda_n(g)=-\lambda_{n+1}(g)$, and $\lambda'_n(g) = \lambda_n(g)$ or $\lambda'_n(g) =- \lambda_n(g)$. 
\end{itemize}
To determine whether $\lambda'_n(g) = \lambda_n(g)$ or $\lambda'_n(g) = -\lambda_n(g)$, let $H$ be the sum of the eigenspaces corresponding to the eigenvalues $\lambda_1(g), \ldots, \lambda_n(g)$. Then $\lambda'_n(g) = \lambda_n(g)$ if $\tau(H) = +1$, \ie $H$ is a positive isotropic $n$-plane, and $\lambda'_n(g) = -\lambda_n(g)$ if $\tau(H) = -1$.
A similar statement holds for the Cartan projections. 
\end{example}

\begin{example}\label{eg:Jordan projection2}
Let $G'' = \SO(n,n-1)$. As we did in Section \ref{sec:grass}, we use the orthogonal splitting $\RR^{n,n} = \RR^{n,n-1} \oplus \RR^{0,1}$ to embed $G''$ as a subgroup of $G' = \PSO(n,n)$ to get $G'' \subset G'\subset G = \PSL(2n, \RR)$. 
Then the Cartan subalgebra $\mathfrak a''$ for $G''$ as described in Example~\ref{eg:SO(n,n-1)} embeds in the Cartan subalgeba $\mathfrak a$ for $G$ as the subspace of $2n \times 2n$ diagonal matrices of the form 
\[\mathfrak a'' = \{\diag (a_1, \ldots, a_{n-1}, 0, 0, -a_{n-1}, \ldots, -a_1)\}.\] The choice of simple roots $\Delta'' = \{\alpha_1'', \ldots, \alpha_{n-1}''\}$ for $G''$ described in Example~\ref{eg:SO(n,n-1)}  are precisely the restrictions of the first $n-1$ simple roots for $G$ as described in Example \ref{eg:SL(n,R)} to $\mathfrak a''$. Hence, the positive Weyl chamber $\mathfrak a''^+$ embeds in the intersection of the positive Weyl chambers $\mathfrak a'^+ \cap \mathfrak a^+$ for $G'$ and $G$, as the subset with $a_1 \geq \cdots \geq a_{n-1} \geq 0 = a_n$. For $g \in G''$, if $(\lambda_1''(g),\dots,\lambda_{n-1}''(g),-\lambda_{n-1}''(g),-\lambda_1''(g))$ is the Lyapunov projection in $G''$ and $(\lambda_1(g),\dots,\lambda_{2n}(g))$ is the Lyapunov projection in $G$, then $\lambda_n(g)=\lambda_{n+1}(g)=0$ and $\lambda_i'(g)=\lambda_i(g)=-\lambda_{2n+1-i}(g)$ for all $i=1,\dots,n-1$. A similar statement holds for the Cartan projections.
\end{example}

Here is the recent characterization of Anosov representations that we will use. It was independently shown by Kapovich-Leeb-Porti~\cite{KLP2,KLPa,KLPb} and by Gu\'eritaud-Guichard-Kassel-Wienhard~\cite{GGKW}. 
First, some brief setup:
There is a well-known bijection between non-empty subsets $\theta_P\subset\Delta$ and conjugacy classes of parabolic subgroups $[P]$ of $G$. Specifically, for any conjugacy class $[P]$ of parabolic subgroups, $\theta_P\subset\Delta$ is the subset with the following property: There is a (necessarily unique) representative $P$ in that conjugacy class $[P]$, called the \emph{standard representative}, whose Lie algebra is spanned by the centralizer $\mathfrak g_0$ of the Cartan subalgebra $\mathfrak a$ (in the cases of interest here $\mathfrak g_0 = \mathfrak a$), each of the positive root spaces, and by the root spaces $\mathfrak g_{-\alpha}$ for all positive roots $\alpha$ not in the span of $\Delta \setminus \theta_P$. 

We state the following result in the special case that $P$ is conjugate to its opposite, in which case the corresponding set of roots $\theta_P$ is invariant under the opposition involution. 
In the following, $| \gamma |$ denotes the word-length of $\gamma \in \Gamma$ with respect to some fixed generating set and $| \gamma|_\infty = \lim_n|\gamma^n|/n$ denotes the stable length, or alternatively the translation length in the Cayley graph of $\Gamma$.

\begin{theorem}\label{thm: GGKW}
Let $G$ be a reductive Lie group, $P$ be a parabolic subgroup, and $\theta_P \subset \Delta$ be the corresponding subset of simple roots. Assume $P$ is conjugate to its opposites.
Then for any hyperbolic group $\Gamma$ and any representation $\rho: \Gamma \to G$, the following are equivalent.
\begin{enumerate}
\item $\rho$ is $P$-Anosov 
\item[(2)] There is a continuous, $\rho$-equivariant, transverse, dynamics preserving map $\xi:\partial\Gamma\to \mathcal{F}_P$, and for any $\alpha\in\theta_P$, $\alpha(\mu(\rho(\gamma))) \to \infty$ as $\gamma \to \infty$ in $\Gamma$.
\item[(2')] There is a continuous, $\rho$-equivariant, transverse, dynamics preserving map $\xi:\partial\Gamma\to \mathcal{F}_P$, and constants $c, C > 0$ so that for any $\alpha\in\theta_P$ and $\gamma \in \Gamma$, $\alpha(\mu(\rho(\gamma))) \geq c| \gamma| - C$.
\item[(3)] There is a continuous, $\rho$-equivariant, transverse, dynamics preserving map $\xi:\partial\Gamma\to \mathcal{F}_P$, and for any $\alpha\in\theta_P$, $\alpha(\lambda(\rho(\gamma))) \to \infty$ as $|\gamma|_\infty \to \infty$ in $\Gamma$.
\item[(3')] There is a continuous, $\rho$-equivariant, transverse, dynamics preserving map $\xi:\partial\Gamma\to \mathcal{F}_P$, and a constant $c > 0$ so that for any $\alpha\in\theta_P$ and $\gamma \in \Gamma$, $\alpha(\lambda(\rho(\gamma))) \geq c| \gamma|_\infty$.
\end{enumerate}
\end{theorem}
\begin{remark}
We will, in this paper, use Conditions~(2) and~(3) to show Anosovness (Condition~(1)) of representations. We included the strengthened versions (2') and (3') of Conditions~(2) and (3) respectively for reference. We also mention that Kapovich-Leeb-Porti \cite{KLP1} proved an even stronger version of the equivalence (1) $\iff$ (2'), namely that for a representation $\rho: \Gamma \to G$ of a finitely generated group $\Gamma$, the group $\Gamma$ is word hyperbolic and $\rho$ is $P$-Anosov if and only if
\begin{enumerate}
\item[(2'')] There are constants $c,C > 0$ so that for any $\alpha\in\theta_P$ and $\gamma \in \Gamma$, $\alpha(\mu(\rho(\gamma))) \geq c| \gamma|-C$.
\end{enumerate}
\end{remark}

As an immediate consequence of Theorem \ref{thm: GGKW}, we have the following useful property that was originally due to Guichard-Wienhard \cite{GW}.
\begin{fact}\label{fact:larger-parabolic}
Let $P,Q\subset G$ be parabolic subgroups, and let $P'$ and $Q'$ be the standard representatives in the conjugacy classes $[P]$ and $[Q]$. Then $P'\cap Q'\subset G$ is a parabolic subgroup, and $\rho
\in\Hom(\Gamma,G)$ is $P'\cap Q'$-Anosov if and only if it is $P$-Anosov and $Q$-Anosov.
\end{fact}


\subsection{The Anosov property under inclusions of Lie groups}\label{sec:Anosovinclusion}
Let $\iota: G' \hookrightarrow G$ be an inclusion of reductive Lie groups. Consider a parabolic subgroup $P'\subset G'$ and a $P'$-Anosov representation $\rho: \Gamma \to G'$. Guichard--Wienhard~\cite[Prop. 4.4]{GW} give a recipe for determining the parabolic subgroups $P$ of $G$ (if any) for which the composition $\iota \circ \rho$ is $P$-Anosov. In particular, it is clear, e.g. from Theorem~\ref{thm: GGKW}, that if the roots in $\theta_P$ restrict to roots in $\theta_{P'}$, then $\iota \circ \rho$ is $P$-Anosov. 

\begin{example}
Consider the inclusion $\iota_{2n-1} : G'' \hookrightarrow G$ for $G'' = \SO(n,n-1)$ and $G = \PSL(2n-1,\RR)$ as in Example~\ref{eg:Jordan projection2} above. Let $B''$ denote the Borel subgroup in $G''$, with corresponding collection of roots $\theta_{B''} = \Delta''$ the full collection of simple roots. Then for our choice of simple roots $\Delta$ and $\Delta''$ for $G$ and $G''$ as in  Example \ref{eg:SL(n,R)} and Example \ref{eg:SO(n,n-1)} respectively, we see that for all simple roots $\alpha_i \in \Delta$, the restriction of $\alpha_i$ to $\mathfrak a''$ is a simple root of $\Delta''$. Specifically, the restriction of $\alpha_i$ to $\mathfrak a''$ is $\alpha''_i$ if $1 \leq i \leq n-1$ or $\alpha''_{2n-1-i}$ if $n \leq i \leq 2n-2$.
Hence the subsets $\theta''$ of $\Delta''$ are in one-one correspondence with the subsets $\theta$ of $\Delta$ which are invariant under the opposition involution. A representation $\rho: \Gamma \to G''$ is $P''$-Anosov if and only if $\iota_{2n-1} \circ \rho$ is $P$-Anosov, where $P'' < G''$ and $P < G$ are the parabolic subgroups whose associated subsets of simple roots $\theta$ and $\theta''$ correspond as above.
\end{example}

\begin{example}\label{eg:includePSOnn}
Consider the inclusion $\iota_{2n} : G' \hookrightarrow G$ for $G = \PSL(2n,\RR)$ and $G' = \PSO(n,n)$ as in Example~\ref{eg:Jordan projection} above. Let $B'$ denote the Borel subgroup in $G'$, so that $\theta_{B'} = \Delta'$ is the full collection of simple roots.
 Then for our choice of simple roots $\Delta$ and $\Delta'$ for $G$ and $G'$ as in Example \ref{eg:SL(n,R)} and Example \ref{eg:SO(n,n)} respectively, we see that the simple roots in $\Delta$ whose restriction to $\mathfrak a'$ are simple roots in $\Delta'$ are $\alpha_1, \ldots, \alpha_{n-1}, \alpha_{n+1}, \ldots, \alpha_{2n-1}$. Here, note that the restriction of $\alpha_{2n-i}$ agrees with that of $\alpha_i$ for all $i=1,\dots,n-1$. Hence, if $\varrho: \Gamma \to G'$ is a $B'$-Anosov representation, then $\iota_{2n} \circ \varrho: \Gamma \to G$ is $P_{\hat n}$-Anosov for $P_{\hat n} \subset G$, the parabolic subgroup whose associated collection of simple roots is $\theta_{P_{\hat n}} = \{\alpha_1, \ldots, \alpha_{n-1}, \alpha_{n+1}, \ldots, \alpha_{2n-1}\}$. This is precisely the stabilizer of a flag which is nearly complete but misses the $n$-dimensional subspace. However, since the restriction of the middle root $\alpha_n$ to $\mathfrak a'$ is equal to $\alpha'_{n} - \alpha'_{n-1}$ which is not a root in $\Delta'$ (and this can not be fixed even with freedom to adjust using the Weyl group), $\iota_{2n} \circ \varrho$ need not be Anosov with respect to the Borel subgroup $B$ in $G$. 
\end{example}

\begin{example}\label{eg:include}
Consider the inclusion $\iota_{n,n}:G'' \to G'$ for $G'' = \SO(n,n-1)$ and $G'= \PSO(n,n)$ as described in Example~\ref{eg:Jordan projection2}. The stabilizer $P_{n-1}'' < G''$ of an isotropic $(n-1)$-plane in $\RR^{n,n-1}$ corresponds to the subset $\theta_{P_{n-1}''}=\{\alpha_{n-1}''\}\subset\Delta''$ of the set of simple roots of $G''$. Similarly, the stabilizer $P_{n-1}' < G'$ of an isotropic $(n-1)$-plane in $\RR^{n,n}$ corresponds to the subset $\theta_{P_{n-1}'}=\{\alpha_{n-1}', \alpha_n'\}\subset\Delta'$ of the set of simple roots of $G''$. From Example~\ref{eg:Jordan projection2}, the $\iota_{n,n}$ embeds the Cartan sub-algebra $\mathfrak a''$ for $G''$ into the Cartan subalgebra $\mathfrak a'$ for $G'$ and we observe that the restrictions of $\alpha'_{n-1}$ and $\alpha'_n$ to $\mathfrak a''$ each coincide with $\alpha''_{n-1}$. Hence $\rho: \Gamma \to G''$ is $P_{n-1}''$-Anosov if and only if $\iota_{n,n} \circ \rho$ is $P_{n-1}'$-Anosov. 
\end{example}

The following proposition gives the condition under which a $B'$-Anosov representation in $G' = \PSO(n,n)$ becomes Anosov with respect to the Borel subgroup $B$ in $G = \PSL(2n, \RR)$ under inclusion. In the following denote by $P_n \subset G$ the stabilizer of an $n$-plane, whose corresponding collection of simple roots $\theta_{P_n} = \alpha_n$.

\begin{proposition}\label{prop:Anosov relations}
Let $\varrho:\Gamma\to G' = \PSO(n,n)$ and let $\iota_{2n}: G'\hookrightarrow G= \PSL(2n,\RR)$ be the inclusion. Suppose that $\rho$ is $B'$-Anosov and that $\iota_{2n} \circ \varrho$ is $P_n$-Anosov. Then
\begin{enumerate} 
\item the $\varrho$-equivariant limit curve $\xi':\partial\Gamma\to\mathcal{F}_{B'}$ and the $\iota_{2n}\circ\varrho$-equivariant curve $\xi:\partial\Gamma\to\mathcal{F}_{P_n} =\Gr_n(\RR^{2n})$ satisfy that $\xi=\xi'^{(n)}_+$  or $\xi=\xi'^{(n)}_-$ (where $\xi'^{(n)}_\pm(x)=(\xi'(x))^{(n)}_\pm$ are as in Example~\ref{eg:tangent to grassmannian 4}).
\item $\iota_{2n} \circ \varrho$ is $B$-Anosov, and the associated $\iota_{2n} \circ \varrho$-equivariant limit curve $\xi'':\partial\Gamma\to\mathcal{F}_B$ in the space of complete flags, satisfies that $\xi''^{(i)}=\xi'^{(i)}$ for all $i\neq n$, and either $\xi''^{(n)}=\xi'^{(n)}_+$ or $\xi''^{(n)}=\xi'^{(n)}_-$.
\end{enumerate}
\end{proposition}

\begin{proof}
Proof of (1): Let $\gamma \in \Gamma$ non-trivial, and let $\gamma^+ \in \partial \Gamma$ be the attracting fixed point for the action of $\gamma$ on $\partial \Gamma$. Since $\iota_{2n}\circ\varrho$ is $P_n$-Anosov,  $\xi(\gamma^+)$ is the unique attracting $n$-plane for the action of $\iota_{2n}\circ\varrho(\gamma)=\varrho(\gamma)$ on $\Gr_n(\RR^{2n})$ and we observe that  $\xi(\gamma^+)$ must be isotropic, since the eigenvectors of $\varrho(\gamma)$ for eigenvalues larger than one are null and pairwise orthogonal. Further, again by simple eigenvalue considerations, the attracting fixed point $\xi'^{(n-1)}(\gamma^+)$ of $\varrho(\gamma)$ in the isotropic Grassmannian $\Gr_{n-1}(\RR^{n,n})$ is also the unique attracting fixed point in the full Grassmannian $\Gr_{n-1}(\RR^{2n})$. It follows that $\xi'^{(n-1)}(\gamma^+) \subset \xi(\gamma^+)$. By density of the points $\gamma^+$ in $\partial \Gamma$, it follows that $\xi'^{(n-1)}(\eta) \subset \xi(\eta)$ for all $\eta \in \partial \Gamma$. Further, for each $\eta \in \partial \Gamma$, $\xi(\eta)$ is an isotropic $n$-plane containing the isotropic $(n-1)$-plane $\xi'^{(n-1)}(\eta)$, so $\xi(\eta) = \xi'^{(n)}_+(\eta)$ or $\xi(\eta) = \xi'^{(n)}_-(\eta)$ and hence by continuity of $\xi$, we have that $\xi = \xi'^{(n)}_+$ or $\xi = \xi'^{(n)}_-$ on all of $\partial \Gamma$.

Proof of (2): By Example \ref{eg:includePSOnn}, we see that $\iota_{2n}\circ\varrho$ is $P_{\hat{n}}$-Anosov. Also, note that $\theta_{P_{\hat{n}}}\cup\theta_{P_n}=\Delta$, so the intersection of the standard representatives of $[P_{\hat{n}}]$ and $[P_n]$ is a Borel subgroup. Since $\iota_{2n}\circ\varrho$ is assumed to be $P_n$-Anosov, Fact \ref{fact:larger-parabolic} implies (2).
\end{proof}

\section{Proof of Theorem~\ref{thm:main2}}\label{sec:main2}

We now prove Theorem~\ref{thm:main2}. Let $\varrho:\Gamma\to G'=\PSO(n,n)$ be a $\PSO(n,n)$-Hitchin representation, and let $\iota_{2n}: \PSO(n,n) \hookrightarrow \PSL(2n,\RR)$ be the inclusion. 
Assume for contradiction that $\iota_{2n} \circ \varrho$ is Anosov with respect to the stabilizer $P_n < \PSL(2n,\RR)$ of an $n$-plane in $\RR^{2n}$.
By Theorem~\ref{thm:HitchinAnosov}, $\varrho$ is Anosov with respect to the Borel subgroup $B'$ of $\PSO(n,n)$. Let $\xi': \partial \Gamma \to \mathcal F_{B'}$ denote the Anosov limit map. 
By Proposition~\ref{prop:Anosov relations}, $\iota_{2n} \circ \varrho$ is Anosov with respect to the Borel subgroup $B$ of $\PSL(2n,\RR)$. Further, the Anosov limit map $\xi:\partial\Gamma\to\mathcal{F}_B$ satisfies that $\xi^{(i)}= \xi'^{(i)}$ for all $i\neq n$, and either $\xi^{(n)}={\xi'}^{(n)}_+$ or $\xi^{(n)}=\xi'^{(n)}_-$. Assume without loss of generality (see Remark \ref{rem:switch orientation}) that the former holds. First note that if $n$ is odd, then Remark \ref{rem:isotropic transverse} implies that for any $x,y\in\partial\Gamma$, the $n$-planes $\xi^{(n)}_+(x)$ and $\xi^{(n)}_+(y)$ fail to be transverse, a contradiction which completes the proof in the case $n$ is odd. We now give the proof in the more interesting case that $n$ is even.

The strategy of the proof will be to use the Anosov dynamics, plus the extra transversality condition provided by Corollary~\ref{cor:transversality}, to show:
\begin{lemma}\label{lem:tangent}
The subset $\xi^{(n-1)}(\partial\Gamma)\subset\Gr_{n-1}(\RR^{n,n})$ is a differentiable sub-manifold 
that is everywhere tangent to the fibers of the natural projection $\varpi^+: \Gr_{n-1}(\RR^{n,n}) \to \Gr_{n}^+(\RR^{n,n})$ from Proposition~\ref{prop:projections}, and is therefore contained in a single fiber.
\end{lemma}

Lemma~\ref{lem:tangent} immediately gives a contradiction which completes the proof of  Theorem~\ref{thm:main2} as follows.

\begin{proof}[Proof of Theorem~\ref{thm:main2}]
Assuming Lemma~\ref{lem:tangent}, we have that $\xi^{(n-1)}(\partial\Gamma)$ is contained in a single fiber of $\varpi^+$ and it follows that $\xi'^{(n)}_+(x_1)=\xi'^{(n)}_+(x_2)$ for all $x_1,x_2\in\partial\Gamma$. We assumed that $\xi^{(n)}=\xi'^{(n)}_+$, so
\[\xi^{(n)}(x_1) = \varpi^+(\xi^{(n-1)}(x_1)) = \varpi^+(\xi^{(n-1)}(x_2))= \xi^{(n)}(x_2),\] 
which contradicts the injectivity of $\xi^{(n)}$ (and the transversality of $\xi^{(n-1)}$).
This concludes the proof of Theorem \ref{thm:main2}.
\end{proof}

We now focus on proving Lemma~\ref{lem:tangent}. Recall from Section~\ref{subsec:BSLd} the line decomposition $\RR^{2n} = \bigoplus_{i=1}^{2n} L_i(x,y)$ associated to a pair of distinct points $x,y \in \partial \Gamma$. Observe that 
\begin{align*}
\xi^{(n-1)}(x) &= L_1(x,y) \oplus \cdots \oplus L_{n-1}(x,y)\\
\xi^{(n+1)}(y) &= L_n(x,y) \oplus \cdots \oplus L_{2n}(x,y)
\end{align*}
Let $U_{\xi^{(n+1)}(y)}$ denote the affine chart of $\Gr_{n-1}(\RR^{2n})$ consisting of all $(n-1)$-planes transverse to $\xi^{(n+1)}(y)$. We use the identification in Section \ref{sec:affine chart} to obtain local coordinates for $U_{\xi^{(n+1)}(y)}$:
\begin{align}\label{eqn:coords}
U_{\xi^{(n+1)}(y)} &\xrightarrow[]{\simeq} \Hom(\xi^{(n-1)}(x), \xi^{(n+1)}(y)) = \bigoplus_{1 \leq i < n \leq j \leq 2n} \Hom(L_i, L_j)(x,y)
\end{align} 
Then for each $y < z \leq x < y$ in $\partial \Gamma$ we observe that $\xi^{(n-1)}(z) \in U_{\xi^{(n+1)}(y)}$ and we express it in coordinates as
\begin{align*}
\xi^{(n-1)}(z) \mapsto \left( u_{ij}(y,z,x) \right)_{1 \leq i < n \leq j \leq 2n}.
\end{align*}
Then $u_{ij}(y,x,x) = 0$ for all $1 \leq i < n \leq j \leq 2n$ and we wish to calculate the ``derivatives" of the $u_{ij}(y,z,x)$ as $z \to x$. 

\begin{lemma}\label{lem:transversality}
For all $y<z<x<y$ in $\partial \Gamma$, we have that $u_{n-1,n}(y,z,x) \neq 0$.
\end{lemma}

\begin{proof}
Note that 
\[ L_{n-1}(z,y) = \xi^{(n-1)}(z) \cap \xi^{(n+2)}(y) = \xi^{(n-1)}(z) \cap (L_{n-1}(x,y) \oplus \xi^{(n+1)}(y)) \]
is exactly the graph of the linear map 
\[\bigoplus_{n \leq j \leq 2n}u_{n-1,j}(y,z,x) : L_{n-1}(x,y) \to \xi^{(n+1)}(y).\]
The condition that $u_{n-1,n}(y,z,x) = 0$ is exactly the condition that $L_{n-1}(z,y) \subset L_{n-1}(x,y) \oplus \xi^{(n)}(y)$. However, this cannot happen, since it would violate the transversality statement of Corollary~\ref{cor:transversality} (a consequence of positivity of the limit curve $\xi'$):
\[ \xi^{(n-1)}(x) + L_{n-1}(z,y) + \xi^{(n)}(y) = \RR^{2n}.\]
\end{proof}

Next, for each $i,j$, choose continuously varying norms $\| \cdot \|$ on the fibers of the bundle $\Hom(L_i,L_j)$, defined in Section~\ref{subsec:BSLd}.  The lift of these norms, again denoted $\| \cdot \|$, gives a continuously varying, $\rho$-equivariant family of norms on $\Hom(L_i,L_j)(x,y)$ depending on a cyclically ordered triple $y < z< x<y$.

\begin{lemma}\label{lem:bounds}
There exists $C>0$ so that the following hold for all $y < z < x < y$ in $\partial\Gamma$.
\begin{enumerate}
\item $\| u_{i,j}(y,z,x) \|_{(y,z,x)} \leq C$ for all $1 \leq i < n \leq j \leq 2n$.
\item $\displaystyle\frac{1}{C}\leq \| u_{n-1,n}(y,z,x)\|_{(y,z,x)}\leq C$.
\end{enumerate}
\end{lemma}

\begin{proof}
Observe that the definition of $u_{i,j}(y,z,x)$ is $\rho$-equivariant, hence $u_{i,j}$ defines a section of the bundle $\Hom(L_i, L_j)$ over $T^1 S$. Statement~(1) then follows from the compactness of $T^1 S$. The lower bound of Statement~(2) also follows from the compactness of $T^1 S$ in light of the fact that the section $u_{n-1,n}$ is nowhere zero by Lemma~\ref{lem:transversality}.
\end{proof}

Next, fix the points $x,y \in \partial \Gamma$ and let $C_x \subset T_{\xi^{(n-1)}(x)} \Gr_{n-1}(\RR^{2n}) $ denote the tangent cone to the curve $\xi^{(n-1)}(\partial \Gamma)$.
The linear coordinates on the patch $U_{\xi^{(n+1)}(y)}\subset \Gr_{n-1}(\RR^{2n})$ give coordinates for $T_{\xi^{(n-1)}(x)}\Gr_{n-1}(\RR^{2n})$, and in those coordinates $C_x$ consists of all $(v_{i,j})_{1 \leq i < n \leq j \leq 2n}$ so that there exists a sequence $z_k \to x$ in $\Gamma$ and $s_k \to \infty$ in $\RR$ so that 
\begin{align}\label{eqn:converge}
s_k{u_{i,j}(y,z_k,x)} &\to v_{i,j} & \text{for all } 1 \leq i < n \leq j \leq 2n
\end{align}

\begin{lemma}\label{lem:C1}
Let $(v_{i,j})_{1 \leq i < n \leq j \leq 2n} \in C_x$ be non-zero. Then $v_{i,j} = 0$ for all $(i,j) \neq (n-1,n)$. 
\end{lemma}

\begin{proof}
Let $z_k \to x$ in $\partial \Gamma$ and let $s_k \to \infty$ in $\RR$ so that~\eqref{eqn:converge} holds. 
 For each $i,j$, choose a non-zero element $b_{i,j}$ spanning $\Hom(L_i,L_j)(x,y)$ and write 
 \[u_{i,j}(y,z,x) =: \zeta_{i,j}(z) b_{i,j}.\]
 Let $p < n \leq q$ with $(p,q) \neq (n-1,n)$. Then
 \begin{align*}
 \frac{\| u_{p,q}(y,z_k,x) \|_{(y,z_k,x)}}{ \| u_{n-1,n}(y,z_k,x) \|_{(y,z_k,x)}} &= \frac{ |\zeta_{p,q}(z_k)|}{| \zeta_{n-1,n}(z_k)|} \frac{ \| b_{p,q} \|_{(y,z_k,x)}}{\| b_{n-1,n}\|_{(y,z_k,x)}}.\\
 \end{align*}
By Proposition~\ref{prop:ratio of norms}, the term $\frac{ \| b_{p,q} \|_{(y,z_k,x)}}{\| b_{n-1,n}\|_{(y,z_k,x)}}$ on the right-hand side goes to infinity.
But by Lemma~\ref{lem:bounds}, the left-hand side is bounded. It follows that
\[ \frac{ |\zeta_{p,q}(z_k)|}{| \zeta_{n-1,n}(z_k)|} \to 0 \]
hence we must have $s_k \zeta_{p,q}(z_k) \to 0$, else $s_k \zeta_{n-1,n}(z_k) \to\pm\infty$ which means $s_k u_{n-1,n}(y,z_k,x)$ diverges and that contradicts the definition of $s_k$.  Hence $s_k u_{p,q}(y,z_k,x) \to 0$ showing that $v_{p,q} = 0$.
\end{proof}

Lemma~\ref{lem:C1} implies that the tangent cone $C_x$ to $\xi^{(n-1)}(\partial \Gamma)$ at the point $\xi^{(n-1)}{(x)}$ is contained in the line corresponding, in the coordinates~\eqref{eqn:coords}, to $\Hom(L_{n-1},L_n)(x,y)$ (note this does not depend on $y$), so it is equal to that line or to a ray contained in the line. Let us show now that $C_x$ is the full line.

\begin{lemma}\label{lem:inj}
Let $z_1,z_2\in\partial\Gamma\setminus\{y\}$ be distinct. Then 
\[u_{n-1,n}(y,z_1,x)\neq u_{n-1,n}(y,z_2,x).\] 
\end{lemma}

\begin{proof}
This follows easily from Lemma~\ref{lem:transversality} and the following formula for changing coordinates on the affine chart $U_{\xi^{(n+1)}(y)}$ to move the origin from $\xi^{(n-1)}(x)$ to $\xi^{(n-1)}(z_1)$: 
\[ u_{n-1,n}(y,z_2,z_1)  = \left( u_{n-1,n}(y,z_2,x) - u_{n-1,n}(y,z_1,x) \right) \circ \Pi_{x}^{z_1}\]
where $\Pi_{x}^{z_1}: \xi^{(n-1)}(z_1) \to \xi^{(n-1)}(x)$ is the projection induced by the splitting $\xi^{(n-1)}(x)\oplus\xi^{(n+1)}(y) = \RR^{2n}$.
 In particular, if $u_{n-1,n}(y,z_2,x) = u_{n-1,n}(y,z_1,x)$, then $u_{n-1,n}(y,z_2,z_1) = 0$, which would contradict Lemma~\ref{lem:transversality}.
 \end{proof}

We now prove Lemma~\ref{lem:tangent}.
\begin{proof}[Proof of Lemma~\ref{lem:tangent}]
Lemma~\ref{lem:C1} implies that the tangent cone $C_x$ to $\xi^{(n-1)}(\partial \Gamma)$ at the point $\xi^{(n-1)}{(x)}$ is contained in the line corresponding, in the coordinates~\eqref{eqn:coords}, to $\Hom(L_{n-1},L_n)(x,y)$, which varies continuously with $x \in \partial \Gamma$. Lemma~\ref{lem:inj} implies that $C_x$ is the entire line, and not just a ray. It now follows that $\xi^{(n-1)}(\partial \Gamma)$ is a differentiable sub-manifold of dimension one (although the parameterization of $\xi^{(n-1)}(\partial\Gamma)$ by $\partial\Gamma$ is not necessarily $C^1$).

Since $\Gr_{n-1}(\RR^{n,n})$ is smoothly embedded in $\Gr_{n-1}(\RR^{2n})$, we may work in the coordinates~\eqref{eqn:coords} on $\Gr_{n-1}(\RR^{2n})$. 
Proposition \ref{prop:isotropic tangent} tells us that in these coordinates, the tangent space to the fiber $\ell_{\xi^{(n-1)}(x)}$ above $\xi^{(n-1)}(x)$ of the projection $\varpi^+: \Gr_{n-1}(\RR^{n,n}) \to \Gr_n^+(\RR^{n,n})$ is $\Hom(\xi^{(n-1)}(x), L_n(x,y))$. Hence Lemma~\ref{lem:C1} and Lemma ~\ref{lem:inj} imply that $\xi^{(n-1)}(\partial \Gamma)$ is tangent to $\ell_{\xi^{(n-1)}(x)}$ at $\xi^{(n-1)}(x)$.  
Since this holds for all points $\xi^{(n-1)}(x)$ on $\xi^{(n-1)}(\partial \Gamma)$, we conclude that $\xi^{(n-1)}(\partial \Gamma)$ is contained in a single fiber, concluding the proof of Lemma~\ref{lem:tangent}. 
\end{proof}

\section{Properly discontinuous actions on $\HH^{n,n-1}$}\label{sec:thm:main-H}
We now prove Theorem~\ref{thm:main-H} which states that the action of a surface group on the pseudo-Riemannian hyperbolic space $\HH^{n,n-1}$ by a $\PSO(n,n)$-Hitchin representation is not properly discontinuous. 
Theorem~\ref{thm:main-H} follows directly from Theorem~\ref{thm:main2} and from the following theorem. Let $\iota_{2n}: \PSO(n,n) \to \PSL(2n,\RR)$ denote the inclusion.

\begin{theorem}\label{thm:gen-H}
Suppose $\varrho: \Gamma \to \PSO(n,n)$ is Anosov with respect to the stabilizer $P_{n-1}'$ of an isotropic $(n-1)$-plane. Then the $\varrho$-action of $\Gamma$ on $\HH^{n,n-1}$ is properly discontinuous if and only if $\iota_{2n} \circ \varrho:\Gamma\to\PSL(2n,\RR)$ is Anosov with respect to the stabilizer $P_n$ of an $n$-plane.
\end{theorem}
We will now prove this theorem. We will use the techniques of Gu\'eritaud-Guichard-Kassel-Wienhard~\cite{GGKW}. 

Let us first recall (a version of) the properness criterion due independently to Benoist and to Kobayashi. In the following, $\mathfrak a'$ denotes a Cartan sub-algebra in the Lie algebra $\mathfrak g'$ of a semi-simple Lie group $G'$, and $\|\cdot\|$ is any norm on $\mathfrak a'$. We assume, as in Section~\ref{sec:char}, that the adjoint group $\Ad(G')$ is contained in the group of inner automorphisms of the complexification $\mathfrak g_\CC$ of the Lie algebra $\mathfrak g$ so that the Cartan projection $\mu': G' \to \mathfrak a'^+$ is well-defined. For the case $G' = \PSO(n,n)$, see Example~\ref{eg:Jordan projection}. 

\begin{theorem}[Benoist~\cite{Benoist}, Kobayashi~\cite{kob}]\label{thm:benkob}
Let $G'$ be a semi-simple Lie group and $H' < G'$ a reductive subgroup. Let $\varrho: \Gamma \to G'$ be a discrete faithful representation of a finitely generated group $\Gamma$. Then the $\varrho$-action of $\Gamma$ on $G'/H'$ is properly discontinuous if and only if $\|\mu(\varrho(\gamma)) - \mu(H')\| \to \infty$ as $\gamma \to \infty$ in~$\Gamma$.
\end{theorem}

In the setting of interest, $G' = \PSO(n,n)$ and $H' \cong \OO(n,n-1)$ is the subgroup which stabilizes the orthogonal splitting $\RR^{n,n} = \RR^{n,n-1}\oplus \RR^{0,1}$, so that $G'/H' = \HH^{n,n-1}$ (see Section~\ref{sec:intro:negcurv}). Recall from Example~\ref{eg:Jordan projection} that the positive Weyl chamber $\mathfrak a'^+$ for $G'$ may be thought of as the subset of the diagonal matrices of the form $\mathrm{diag}(a_1, \ldots, a_n, -a_n, \ldots, -a_1)$ where 
\begin{align*}
a_1 \geq a_2 \geq \cdots \geq a_{n-1} \geq a_n, -a_n
\end{align*}
but $a_{n}$ is allowed to have either sign. The Cartan projection of $H'$ is then given by
\begin{align*}
\mu(H') &= \{\mathrm{diag}(a_1, \ldots, a_n, -a_n, \ldots, -a_1) \in \mathfrak a^+ : a_n = 0\}.
\end{align*}
Hence, in this setting, the criterion for properness of the action of $\Gamma$ on $G'/H'$ in Theorem~\ref{thm:benkob} reduces to the simple condition that the $n^{th}$ diagonal entry of the Cartan projection $\mu'_n(\varrho(\gamma))$ escapes all compact subsets of $\RR$ as $\gamma \to \infty$ in $\Gamma$. Note that $\mu'_n(\varrho(\gamma))$ does not necessarily need to be positive, unlike $\mu'_i(\varrho(\gamma))$ for $i<n$. However, by the following result of Kassel~\cite{kassel}, we can deduce that $\mu'_n(\varrho(\gamma))$ diverges to infinity in a consistent direction (i.e. always positive or always negative). Note that here $\mu'(H)$ separates $\mathfrak a'^+$ into two connected components.
\begin{theorem}[Kassel]\label{thm:kasco1}
Let $G', H', \Gamma,$ and $\varrho$ be as in Theorem~\ref{thm:benkob} and suppose further that $G'$ and $H'$ are both connected, that $\Gamma$ is not virtually cyclic, and that $\mathrm{rank}_\RR(H') = \mathrm{rank}_\RR(G') -1$. If the $\varrho$-action of $\Gamma$ on $G'/H'$ is proper, then all but finitely many points of $\mu(\varrho(\Gamma))$ lie in a single component of the complement $\mathfrak a'^+ \setminus \mu(H')$.
\end{theorem}

Observe that 
\begin{align*}
 \mathrm{rank}_\RR \OO(p,q) &= \mathrm{rank}_\RR \PSO(p,q)=\min(p,q),
\end{align*}
and that rank is invariant under taking finite index subgroups. 
 Hence the theorem applies in the case $G' = \PSO_0(n,n)$ and $H' = \OO(n,n-1) \cap \PSO_0(n,n)$. However, it is easy to check that the same result continues to hold in the case of interest here, namely $G' = \PSO(n,n)$ and $H' = \OO(n,n-1)$ and we will apply the theorem in this case without further remark.
 
\begin{proof}[Proof of Theorem~\ref{thm:gen-H}]
Let $\varrho: \Gamma \to \PSO(n,n)$ be $P_{n-1}'$-Anosov. 
We begin with the reverse implication, which is straightforward. Suppose $\iota_{2n} \circ \varrho$ is $P_n$-Anosov. 
Let $\mu: \PSL(2n,\RR) \to \mathfrak{a}^+$ denote the Cartan projection of $G = \PSL(2n,\RR)$ as in Example~\ref{eg:Jordan projection1}. Then by Theorem~\ref{thm: GGKW}, $\mu_n(\iota_{2n}\circ\varrho(\gamma)) \to \infty$ as $\gamma \to \infty$ in $\Gamma$. Since $\mu_n(\iota_{2n}g) = |\mu'_n(g)|$ for all $g \in \PSO(n,n)$, it follows that $|\mu'_n(\iota_{2n}\circ\varrho(\gamma))| \to \infty$ as $\gamma \to \infty$ in $\Gamma$. Hence, the $\varrho$-action on $\HH^{n,n-1} = G'/H'$ is proper by Theorem~\ref{thm:benkob}, since $\mu'_n(H') = 0$.

We now prove the forward implication. 
Let $\xi^{(n-1)}: \partial \Gamma \to \Gr_{n-1}(\RR^{n,n-1})$ be the Anosov limit curve, and 
let $\xi^{(n)}_+:  \partial \Gamma \to \Gr_n^+(\RR^{n,n})$ and $\xi^{(n)}_-: \partial \Gamma \to \Gr_n^-(\RR^{n,n})$ denote the $\varrho$-equivariant, continuous embeddings defined by $\xi^{(n)}_\pm = \varpi_\pm \circ \xi^{(n-1)}$, where $\varpi_+$ (resp. $\varpi_-$) is the projection taking an isotropic $(n-1)$-plane to the unique positive (resp. negative) istropic $n$-plane containing it, see Proposition~\ref{prop:projections}.

Now assume that $\varrho$ determines a proper action of $\Gamma$ on $\HH^{n,n-1}$. Then by Theorem~\ref{thm:benkob} and the discussion just above, we have that $\mu'_n(\varrho(\gamma))$ leaves every compact set as $\gamma \to \infty$ in $\Gamma$. Further, by Theorem~\ref{thm:kasco1}, either $\mu'_n(\varrho(\gamma)) \to +\infty$ or $\mu'_n(\varrho(\gamma)) \to -\infty$ and the sign is consistent for all escaping sequences in $\Gamma$. Without loss in generality, we assume $\mu'_n(\varrho(\gamma)) \to +\infty$ whenever $\gamma \to \infty$ in $\Gamma$. It then follows that for any $\gamma \in \Gamma \setminus \{1\}$, the $n^{th}$ value $\lambda'_n(\varrho(\gamma)) = \lim_{m \to \infty} \mu'_n(\varrho(\gamma^m))/m$ of the Lyapunov projection is non-negative. 
Now, fix $\gamma \in \Gamma \setminus \{1\}$, and observe that $\xi^{(n-1)}(\gamma^+)^\perp \cap \xi^{(n-1)}(\gamma^-)^\perp$ is a $\varrho(\gamma)$-invariant subspace on which the restriction of the inner product has signature $(1,1)$, where here $\gamma^\pm = \lim_{m \to \pm \infty} \gamma^m \in \partial \Gamma$. It follows that the restriction of $\varrho(\gamma)$ to this $(1,1)$ subspace is diagonalizable, and the corresponding eigenvalues are precisely the exponentials of $\pm \lambda'_n(\varrho(\gamma))$. 
If $\lambda'_n(\varrho(\gamma)) = 0$, then $\xi^{(n-1)}(\gamma^+)^\perp \cap \xi^{(n-1)}(\gamma^-)^\perp$
projects to a line in $\HH^{n,n-1}$ which is point-wise fixed by the action of $\varrho(\gamma)$, contradicting properness of the action. Hence $\lambda'_n(\varrho(\gamma)) > 0$. It then follows that the $n$-plane $\xi^{(n)}_+(\gamma^+)$ is the attracting fixed point for the action of $\varrho(\gamma)$ on the \emph{full} Grassmannian $\Gr_n(\RR^{2n})$ of $n$-planes in $\RR^{2n}$. Hence, composing with the inclusion $\Gr_n^+(\RR^{n,n}) \hookrightarrow \Gr_n(\RR^{2n})$, the map $\xi^{(n)}_+$ determines a continuous embedding $\partial \Gamma \to \Gr_n(\RR^{2n})$ which is equivariant and dynamics preserving for the representation $\iota_{2n} \circ \varrho: \Gamma \to \PSL(2n, \RR)$.
Hence, the implication (2)~$\implies$~(1) in Theorem~\ref{thm: GGKW} shows that $\iota \circ \varrho$ is Anosov with respect to the stabilizer $P_n$ of an $n$-plane in $\RR^{2n}$. 
\end{proof}

Theorem~\ref{thm:main-H} follows directly from Theorem~\ref{thm:gen-H}, Theorem~\ref{thm:main2}, and the fact that Hitchin representations are Anosov with respect to the Borel subgroup (Theorem \ref{thm:HitchinAnosov}).

\section{Constant curvature geometry in signature $(n,n-1)$}\label{sec:length-functions}
We now turn to some of the geometry needed for Theorem~\ref{thm:main-flat}. In order to understand properly discontinuous actions by isometries of the psuedo-Riemannian Euclidean space $\EE^{n,n-1}$, we recall the notion of signed translation length in $\EE^{n,n-1}$, known as the \emph{Margulis invariant} (Section~\ref{sec:Margulis-invariant}). The proof of Theorem~\ref{thm:main-flat} involves deforming into pseudo-Riemannian hyperbolic geometry $\HH^{n,n-1}$, and it will be important to have a theory of signed translation length in that setting as well.  We develop the notion of signed translation length in each of $\EE^{n,n-1}$ and $\HH^{n,n-1}$ in parallel.

Before we proceed, we will perform a change of basis on $\RR^{2n}$ that we will use for the rest of this article. Recall that in Section \ref{sec:grass}, we specified the bilinear form $\langle\cdot,\cdot\rangle_{n,n}$ on $\RR^{2n}$ using the matrix $J_{2n}$ in the standard basis $e_1,\dots,e_{2n}$ of $\RR^{2n}$: if $x,y\in\RR^{2n}$ are written as $x=(x_1,\dots,x_{2n})^T$ and $y=(y_1,\dots,y_{2n})^T$  in the standard basis of $\RR^{2n}$, then 
\[\langle x,y\rangle_{n,n}=\sum_{i=1}^{2n}x_iy_{2n+1-i}.\]
Let $e'_1,\dots,e'_{2n}$ be the basis of $\RR^{2n}$ defined by
\[e_i':=\left\{
\begin{array}{ll}
e_i+e_{2n+1-i}&\text{if }i\leq n\\
e_{i-n}-e_{3n+1-i}&\text{if }i\geq n+1\\
\end{array}\right..\]
If $x,y\in\RR^{2n}$ are written as $x=(x_1,\dots,x_{2n})^T$ and $y=(y_1,\dots,y_{2n})^T$  in the basis $e_1',\dots,e_{2n}'$, then 
\[\langle x,y \rangle_{n,n} = \sum_{i=1}^n x_i y_i - \sum_{i=n+1}^{2n} x_i y_i.\]
In Section \ref{sec:length-functions} and Section \ref{sec:geometric limit}, we will think of $e_1',\dots,e_{2n}'$ as the standard basis of $\RR^{2n}$ instead of $e_1,\dots,e_{2n}$, as this will be more convenient. Henceforth, all coordinates, matrices, and vectors will be written using the basis $e_1',\dots,e_{2n}'$.

\subsection{$\HH^{n,n-1}$ and $\EE^{n,n-1}$ as real projective geometries}\label{sec:subgeometries}
Both $\HH^{n,n-1}$ and $\EE^{n,n-1}$ naturally embed in real projective geometry.
Indeed, the projective model for $\HH^{n,n-1}$ is given by:
\begin{align*}
\HH^{n,n-1} := \left\{ [x] \in \PP(\RR^{2n}): \langle x, x \rangle_{n,n} < 0\right\}.
\end{align*}
The projective orthogonal group $\PO(n,n) < \PGL(2n, \RR)$ for this inner product preserves $\HH^{n,n-1}$ and is the isometry group of a geodesically complete pseudo-Riemannian metric $\metric^\HH$ of signature $(n,n-1)$. The metric $\metric^\HH$ is the natural metric coming from restriction of $\langle \cdot, \cdot \rangle_{n,n}$ to the tangent spaces of the hyperboloid $\langle x,x \rangle_{n,n} = 1$, which double covers $\HH^{n,n-1}$.

The restriction of $\langle \cdot, \cdot \rangle_{n,n}$ to the vector space $\RR^{2n-1}=\mathsf{span}\{e_1',\dots,e_{2n-1}'\}$, determines a complete, flat metric $\metric^{\EE}$ of signature $(n,n-1)$ on the parallel affine hyperplane defined by $x_{2n} = 1$, and hence on the corresponding affine chart of projective space. We henceforth identify this affine chart with $\EE^{n,n-1}$:

\begin{align*}
\EE^{n,n-1} &:= \left\{ [x_1: \ldots: x_{2n-1}: 1]\right\} \subset  \PP(\RR^{2n}).
\end{align*}
The subgroup of the projective general linear group $\PGL(2n,\RR)$ that preserves this affine chart and its flat metric gives the isometry group of~$\EE^{n,n-1}$:
\begin{align} \label{eqn:isom-flat}
\Isom(\EE^{n,n-1}) &= \left\{ \begin{bmatrix} A & v \\ 0 & 1 \end{bmatrix} \in \PGL(2n, \RR): A \in \OO(n,n-1), v \in \RR^{2n-1} \right\},
\end{align}
where here $\OO(n,n-1)$ denotes the orthogonal group for the restriction, to be denoted $\langle \cdot, \cdot \rangle_{n,n-1}$, of $\langle \cdot, \cdot \rangle_{n,n}$ to $\RR^{2n-1}$. The vector subspace $\RR^{2n-1}$ together with inner product $\langle \cdot, \cdot \rangle_{n,n-1}$ is denoted $\RR^{n,n-1}$ as usual.
We will henceforth restrict to the orientation-preserving isometry groups $\PSO(n,n)$ of $\HH^{n,n-1}$ and $\Isom^+(\EE^{n,n-1})$ of $\EE^{n,n-1}$, which consists of the elements as in~\eqref{eqn:isom-flat} with $A \in \SO(n,n-1)$.  The reason for this is that the discussion of properly discontinuous actions, in Section~\ref{sec:length-function-H} and Section~\ref{sec:Margulis-invariant}, will make important use of the orientation. A theory of properly discontinuous actions in the general setting will follow from elementary considerations, but is not needed for the main goal of the paper.

We fix once and for all an orientation on $\RR^{n,n}$ defined by the $n$-form $e_1'\wedge\dots\wedge e_{2n}'$, and an orientation on $\RR^{n,n-1}$ defined by the $n$-form $e_1'\wedge\dots\wedge e_{2n-1}'$. The diffeomorphism $\EE^{n,n-1}\to\RR^{n,n-1}$ given by $[x_1:\dots:x_{2n-1}:1]\mapsto(x_1,\dots,x_{2n-1})$ then defines an orientation on $\EE^{n,n-1}$.

\subsection{Translation lengths in $\HH^{n,n-1}$}\label{sec:length-function-H}
We follow the conventions from Example~\ref{eg:Jordan projection} and think of $G' = \PSO(n,n)$ as embedded in $G= \PSL(2n,\RR)$, denoting by $\lambda'$ and $\lambda$ the respective Lyapunov projections.

Consider an element $g \in G'$ whose Lyapunov projection $\lambda'(g)$ satisfies that
\begin{align}\label{eqn:evals-g}
\lambda_1'(g) \geq  \cdots \geq \lambda_{n-1}'(g) > \lambda_n'(g), -\lambda_n'(g) > -\lambda_{n-1}'(g) \geq \cdots \geq -\lambda_1'(g)
\end{align}
where here $\lambda_n'(g)$ may be positive, in which case $\lambda_n'(g) = \lambda_n(g)$, or negative, in which case $\lambda_n'(g) = -\lambda_n(g)$, or zero. 
With our future application in mind, we note that this assumption holds for all non-trivial elements of a $\PSO(n,n)$-Hitchin representation.
Thinking of $g$ as an element of the (projective) matrix group $G$, the entries~\eqref{eqn:evals-g} of the Lyapunov projection are the logarithms of the moduli of the eigenvalues of $g$. Let $V^+_{n-1}(g)$ denote the sum of the generalized eigenspaces associated to the $\lambda_1'(g), \ldots, \lambda_{n-1}'(g)$, and let $V^-_{n-1}(g)$ denote the sum of the generalized eigenspaces associated to the $-\lambda_{n-1}'(g), \ldots, -\lambda_{1}'(g)$. Then $(V^+_{n-1}(g), V^-_{n-1}(g))$ is a pair of transverse isotropic $(n-1)$-spaces. The orthogonal complement of $V^+_{n-1}(g) \oplus V^-_{n-1}(g)$ is a $(1,1)$-subspace $L^+_n(g) \oplus L^-_n(g)$, where $L_n^+(g), L_n^-(g)$ are defined as follows. In the case that $\lambda'_n(g) \neq -\lambda'_n(g)$, $L^+_n(g)$ (resp. $L^-_n(g)$) denotes the eigenspace for the eigenvalue $\exp \lambda'_n(g)$ (resp. $\exp(- \lambda'_n(g))$), and we note that by definition of $\lambda'$, the subspace $V^+_{n-1}\oplus L^+_n(g)$ is a positive isotropic $n$-plane; it is precisely this convention that defines the sign of $\lambda'_n(g)$. If $\lambda'_n(g) = -\lambda'_n(g) =0$, then $L^+_n(g) \oplus L^-_n(g)$ is a decomposition of the $1 = \exp(0)$ eigenspace into null lines so that $V^+_{n-1}(g)\oplus L^+_n(g)$ is a positive isotropic $n$-plane.

Here is a geometric picture of the action of $g$ on $\HH^{n,n-1}$. Each of the subspaces $\PP(V^+_{n-1}), \PP(V^-_{n-1})$ in $\PP(\RR^{2n})$ are contained in the ideal boundary
\begin{align*}
\partial \HH^{n,n-1} &= \left\{ [x] \in \PP(\RR^{2n}): \langle x, x \rangle_{n,n} = 0 \right\}.
\end{align*}
of $\HH^{n,n-1}$.
The subspace $\PP(V^+_{n-1}(g) \oplus V^-_{n-1}(g))$ intersects $\HH^{n,n-1}$ in a totally geodesic copy of $\HH^{n-1,n-2}$ and the action of $g$  repels from $\PP(V^-_{n-1}(g))$ and attracts toward $\PP(V^+_{n-1}(g))$. For example, if $g$ is diagonalizable with distinct eigenvalues, then for each $1 \leq i \leq n-1$, $\exp \lambda'_i(g)$ is an eigenvalue of $g$ with eigenline $L_i^+$ and $\exp(-\lambda'_i(g))$ is an eigenvalue with eigenline $L^-_i$ such that $L^+_i \oplus L^-_i$ has signature $(1,1)$. The projection $\PP(L^+_i \oplus L^-_i)$ to $\PP(\RR^{2n})$ intersects $\HH^{n,n-1}$ in a line with ideal endpoints $\PP(L_i^+), \PP(L_i^-) \in \partial \HH^{n,n-1}$, which is invariant under $g$, is Riemannian, and has a well-defined orientation defined by labeling $\PP(L_i^+)$ the positive endpoint. The picture of the action on $\PP(V^+_{n-1}(g) \oplus V^-_{n-1}(g))$ is slightly more complicated in the case that $g$ is not diagonalizable and we do not attempt a thorough description here.

The important behavior we wish to observe is in the $g$-invariant Riemannian line $\axisH = \axisH(g) := \PP(L^+_n \oplus L^-_n)\cap \HH^{n,n-1}$ with endpoints $\PP(L^+_n),  \PP(L^-_n) \in \partial \HH^{n,n-1}$.
The translation along the axis $\axisH$, which is sometimes referred to as the \emph{slow axis}, may be either toward or away from $\PP(L_n^+)$, depending on the sign of $\lambda_n'(g)$. Hence, the translation amount
\begin{align}\label{eqn:length-H}
\mathscr L(g) := 2\lambda_n'(g)
\end{align}
has a well-defined sign. Note that under the same assumptions on $g \in \PSO(n,n)$ as above, the action of the cyclic group $\langle g \rangle$ on $\HH^{n,n-1}$ is properly discontinuous if and only if $\mathscr L(g) \neq 0$.

\begin{remark}\label{rem:sign-change}
If $g \in \PSO(n,n)$ has Lyapunov projection $\lambda'(g)$ as in~\eqref{eqn:evals-g} above, then $\mathscr L(g) = (-1)^n \mathscr L(g^{-1})$. This follows easily because $V^\pm_{n-1}(g^{-1}) = V^\mp_{n-1}(g)$, but $L^\pm_n(g^{-1}) = L^\mp_n(g)$ if $n$ is even while $L^\pm_n(g^{-1}) = L^\pm_n(g)$ if $n$ is odd. 
\end{remark}

\begin{remark}
It follows from Theorem~\ref{thm:kasco1} that if $g, h \in \PSO(n,n)$ have Lyapunov projections $\lambda'(g), \lambda'(h)$ as in~\eqref{eqn:evals-g} above, if $\langle g, h\rangle$ is not virtually cyclic, and if further $\mathscr L(g)$ and $\mathscr L(h)$ have opposite sign, then  $\langle g, h\rangle$ does not act properly discontinuously on $\HH^{n,n-1}$. This is the analogue of Margulis's Opposite Sign Lemma from the setting of affine geometry, see Lemma~\ref{lem:osl} below.
In particular, in light of Remark~\ref{rem:sign-change}, if $n$ is odd, then the only groups which admit proper actions by isometries of $\HH^{n,n-1}$ are virtually cyclic, see Benoist~\cite{Benoist}.
\end{remark}

\subsection{Translation lengths in $\EE^{n,n-1}$: the Margulis invariant}\label{sec:Margulis-invariant}

Recall that an element $g \in \Isom_+(\EE^{n,n-1}) < \PSL(2n,\RR)$ has the form
\begin{align}
g = \begin{bmatrix} A_g & v_g \\ 0 & 1 \end{bmatrix} \in \PSL(2n,\RR)
\end{align}
where $v_g \in \RR^{2n-1}$ is called the \emph{translational part} and $A_g \in \SO(n,n-1)$ is called the \emph{linear part}. Here we think of $\SO(n,n-1)$ as the subgroup of $\PSL(2n,\RR)$ which preserves the vector space $\RR^{2n-1}$ spanned by the first $2n-1$ coordinate basis vectors of $\RR^{2n}$, and which preserves the form $\langle \cdot, \cdot \rangle_{n,n}$, and hence preserves its restriction, denoted $\langle \cdot, \cdot \rangle_{n,n-1}$, to $\RR^{2n-1}$. The form $\langle \cdot, \cdot \rangle_{n,n-1}$ on $\RR^{2n-1}$ makes the affine hyperplane $x_{2n} = 1$, and hence the corresponding affine chart of projective space $\PP(\RR^{2n})$, into a copy of $\EE^{n,n-1}$, whose orientation-preserving isometry group has the form above.

Let $g \in \Isom_+(\EE^{n,n-1}) < \PSL(2n,\RR)$ and note that the Lyapunov projection $\lambda(g)$ is equal to the Lyapunov projection $\lambda''(A_g)$ of the linear part $A_g$, where we follow the convention of Example~\ref{eg:Jordan projection2} and think of $G'' = \SO(n,n-1)$ as embedded in $G' = \PSO(n,n)$ with both embedded in $G = \PSL(2n,\RR)$.
With our future application to actions on $\EE^{n,n-1}$ whose linear part is Hitchin, 
let us assume that the Lyapunov projection $\lambda(g)$ satisfies:
\begin{align}\label{eqn:evals-affine}
\lambda_1(g) \geq \cdots \geq \lambda_{n-1}(g) > \lambda_n(g) = 0 = -\lambda_n(g) >  -\lambda_{n-1}(g) \geq \cdots \geq -\lambda_1(g).
\end{align}
Then the affine transformation $g$ has a unique invariant line $\axisR$, which we will describe now.
The values listed in~\eqref{eqn:evals-affine} are precisely the logarithms of the moduli of the eigenvalues of $g$, repeated with multiplicity.
Let $V^+_{n-1}(g)$ denote the sum of the generalized eigenspaces associated to the $\lambda_1(g), \ldots, \lambda_{n-1}(g)$, and let $V^-_{n-1}(g)$ denote the sum of the generalized eigenspaces associated to the $-\lambda_{n-1}(g), \ldots, -\lambda_{1}(g)$. In fact, $V^+_{n-1}(g), V^-_{n-1}(g)$ are contained in $\RR^{2n-1} \subset \RR^{2n}$ and are sums of generalized eigenspaces for the linear part $A_g$ of~$g$. Each of $V^+_{n-1}(g), V^-_{n-1}(g)$ is an isotropic $(n-1)$-plane for the form $\langle \cdot, \cdot \rangle_{n,n-1}$ and the pair $(V^+_{n-1}(g), V^-_{n-1}(g))$ is transverse, meaning the span has signature $(n-1,n-1)$.
The generalized eigenspace $V_0(g)$ of $g$ for the eigenvalue $1 = \exp(0)$ is two-dimensional and contains the eigenline $L_0(g) \subset \RR^{2n-1}$ for the eigenvalue $1 = \exp(0)$ of $A_g$. Since $V_0 \cap \RR^{2n-1} = L_0$, we have that $\axisR = \axisR(g) := \PP(V_0) \cap \EE^{n,n-1}$ is an affine line parallel to the direction of $L_0$.
We may orient $L_0$ and hence $\axisR$ as follows. Choose a positively oriented basis 
\begin{align}\label{eqn:basis-affine}
(f_1^+, \ldots, f_{n-1}^+, f_0, f_{n-1}^-, \ldots, f_1^-) \subset (\RR^{2n-1})^{2n-1}
\end{align}
for $\RR^{2n-1}$ so that 
\begin{align*}
\mathsf{span}(f_1^+, \ldots, f_{n-1}^+) &= V^+_{n-1}\\
\mathsf{span}(f_1^-, \ldots, f_{n-1}^-) &= V^-_{n-1}\\
\RR f_0 &= L_0.
\end{align*}
Then $\langle f_i^+, f_j^+ \rangle = \langle f_i^-, f_j^- \rangle = \langle f_i^+, f_0 \rangle = \langle f_i^-, f_0\rangle= 0$ for all $1 \leq i, j \leq n-1$, and
$\langle f_0, f_0 \rangle > 0$. Further we may arrange that 
\begin{align*}
\langle f_i^+, f_j^- \rangle \ \ &\left\{\begin{array}{rr} = 0 & \text{ if } i \neq j \\ < 0 & \text{ if } i =j \end{array} \right..
\end{align*}
This together with the positive orientation of the basis determines the direction of $f_0$ and we orient $L_0$ so that the $f_0$ direction is positive.
This determines an orientation on any parallel affine line, in particular on the translation axis $\mathcal A$.

\begin{remark}\label{rem:orientation}
Alternatively, we may orient the line $L_0(g)$ as follows. 
Since $L_0(g)$ is positive for $\langle \cdot, \cdot \rangle_{n,n-1}$, the two-plane $L_0 \oplus \RR e_{2n}$ has signature $(1,1)$ for the form $\langle \cdot, \cdot \rangle_{n,n}$ and hence splits as a direct sum of null lines $L^+_0 \oplus L^-_0$ where we choose the labeling so that the isotropic $n$-plane $V^+_{n-1} \oplus L^+_0$ is positive. Then, 
there is a unique $\ell \in L_0(g)$ so that $\ell + e_{2n} \in L^+_0$. We orient $L_0(g)$ in the direction of $\ell$. This agrees with the orientation defined above.
\end{remark}

Since the line $\axisR$ is Riemannian, oriented, $\langle g\rangle$-invariant, and its direction is given by $f_0$, we may measure the signed translation distance of $g$ along $\axisR$ by the formula,
\begin{align}\label{eqn:Margulis-invt}
\alpha(g) := \langle g\cdot\mathbf x -\mathbf x, f_0 \rangle_{n,n-1},
\end{align}
where $\mathbf x$ is any point in $\axisR$.

\begin{remark}\label{rem:any-x}
A simple computation shows that the right-hand side of Equation~\ref{eqn:Margulis-invt} yields the same quantity for any $\mathbf x \in \EE^{n,n-1}$ (not just for $\mathbf x\in \axisR$):
\begin{align}
\alpha(g) = \langle g\cdot\mathbf x - \mathbf x, f_0\rangle_{n,n-1} &= \langle v, f_0\rangle_{n,n-1}\label{eqn:simplify-alpha}\\
 &= d_{\axisR}(\Pi_{\axisR}(\mathbf x),\Pi_{\axisR} (g\cdot\mathbf x))\label{eqn:alpha-projection}
\end{align}
where here $\Pi_{\axisR}: \EE^{n,n-1} \to \axisR$ denotes the orthogonal projection and $d_{\axisR}(y, x)$ denotes plus or minus the Riemannian distance between $y$ and $x$ along $\axisR$ with positive sign if and only if the pair $(y,x)$ is positive for the orientation induced by $f_0$.
\end{remark}

The quantity $\alpha(g)$ is often called the \emph{Margulis invariant} of the transformation $g$. It plays a crucial role in determining proper discontinuity of group actions on $\EE^{n,n-1}$. Indeed, the action of the cyclic group $\langle g \rangle$ is properly discontinuous if and only if $\alpha(g) \neq 0$. The following lemma, known as the Opposite Sign Lemma, goes back to Margulis's original work~\cite{mar83} on properly discontinuous groups of isometries of $\EE^{2,1}$. 

\begin{lemma}[Margulis~\cite{mar83}, Abels-Margulis-Soifer~\cite{AMS1}]\label{lem:osl}
Assume that $g, h \in \Isom_+(\EE^{n,n-1})$ have linear parts $A_g, A_h$ as above. Assume further that $\langle g,h \rangle$ is not virtually cyclic. If $\alpha(g), \alpha(h)$ have opposite sign, then $\langle g,h \rangle$ does not act properly discontinuously on $\EE^{n,n-1}$.
\end{lemma}

In order to prove Theorem~\ref{thm:main-flat}, we will need a full properness criterion for actions on $\EE^{n,n-1}$. The converse of Lemma~\ref{lem:osl} is not true (\cite{GLMM} give an example in $\EE^{2,1}$). However, a modified version of the converse does hold in the context of interest to us here. It is phrased in terms of geodesic currents.

\begin{remark}\label{rem:AMS}
Suppose that $g \in \Isom_+(\EE^{n,n-1}) $ has the property that the Lyapunov projection $\lambda''(A_g)$ of the linear part of $g$ satisfies
\[\lambda_1''(A_g) \geq \cdots \geq \lambda_{n-1}''(A_g)>0 >  -\lambda_{n-1}''(A_g) \geq \cdots \geq -\lambda_1''(A_g).\]
Observe that $\alpha(g) = (-1)^n \alpha(g^{-1})$. In particular, Lemma \ref{lem:osl} implies that if $\langle g,h \rangle\subset\Isom_+(\EE^{n,n-1})$ is not virtually cyclic, then $\langle g,h \rangle$ cannot act properly discontinuously on $\EE^{n,n-1}$ when $n$ is odd.
\end{remark}

\subsection{The space of geodesic currents}
We now return to our surface group $\Gamma = \pi_1 S$.
As in Section~\ref{sec:Anosov}, we fix a hyperbolic metric on the surface $S$ for this entire discussion. We let $\varphi_t$ denote the geodesic flow on $T^1 S$.
\begin{definition}\label{def:current}
An \emph{geodesic current} $\mu$ is a finite, $\varphi_t$-invariant, Borel measure on the unit tangent bundle $T^1 S$. We denote the space of geodesic currents on $S$ by $\mathcal{C}(S)$. 
\end{definition}

\begin{remark}
Geodesic currents were  introduced by Bonahon \cite{Bonahon1988} in his description of the Thurston boundary of Techm\"uller space. Definition~\ref{def:current}, which follows Goldman-Labourie-Margulis~\cite{GLM}, is slightly different than Bonahon's original definition in that the currents of Definition~\ref{def:current} are oriented, while those from Bonahon's setting are not.
\end{remark}

The most basic example of a geodesic current is the current associated to an oriented closed geodesic $c$ on $S$. Denote by $\mu_c$ the geodesic current that is uniformly supported on the tangent field of $c$ and whose total mass is
\[\int_{T^1S}d\mu_c=\ell(c),\]
where $\ell(c)$ denotes the length of $c$. This defines a map from the oriented closed geodesics $\mathcal{CG}(S)$ into the space $\mathcal C(S)$ of geodesic currents.

As a consequence of the Banach-Alaoglu Theorem, we have the following fact.

\begin{fact}\label{thm:Bonahon}
Equip $\mathcal{C}(S)$ with the weak-* topology. The space of probability currents 
\[\mathcal{C}_1(S):=\left\{\mu\in\mathcal{C}(S):\int_{T^1 S}d\mu=1\right\}.\]
is compact. 
\end{fact}

\subsection{The Margulis invariant for currents and the properness criterion}\label{sec:Marg-currents}
Here we will discuss a properness criterion for actions on $\EE^{n,n-1}$, due originally to Goldman-Labourie-Margulis~\cite{GLM} in the case of free and surface groups acting on $\EE^{2,1}$, and extended by Ghosh-Trieb~\cite{GT} to the case of word hyperbolic groups acting with Anosov linear part in any $\EE^{n,n-1}$. This is one of several key tools needed for Theorem~\ref{thm:main-flat}. We shall discuss the properness criterion in the context of interest, namely $\Gamma = \pi_1 S$ is the fundamental group of a closed surface $S$ of negative Euler characteristic. As in the previous section we equip $S$ with a fixed hyperbolic metric.

Let $(\rho, u): \Gamma \to \Isom_+(\EE^{n,n-1}) = \SO(n,n-1) \ltimes \RR^{2n-1}$ be an action of the group $\Gamma$ by isometries of $\EE^{n,n-1}$. Here $\rho: \Gamma \to \SO(n,n-1) =: G''$ denotes the \emph{linear part} of the action, a homomorphism, and $u: \Gamma \to \RR^{2n-1}$ denotes the \emph{translational part}, which is a $\rho$-\emph{cocycle}: 
\[u(\gamma_1 \gamma_2) = u(\gamma_1) + \rho(\gamma_1)u(\gamma_2).\]

Suppose the linear part $\rho: \Gamma \to \SO(n,n-1)= G''$  is Anosov with respect to the stabilizer $P''_{n-1}$ of an isotropic $(n-1)$-plane in $\RR^{n,n-1}$. Then each non-trivial element $g = (\rho(\gamma), u(\gamma))$ satisfies~\eqref{eqn:evals-affine} and therefore the Margulis invariant $\alpha(\rho(\gamma),u(\gamma))$ is defined. Recall that oriented closed geodesics $c \in \mathcal{CG}(S)$ are in one-one correspondence with non-trivial conjugacy classes $[\gamma] \subset \Gamma$. Since the Margulis invariant is invariant under conjugation, we may naturally associate to the oriented closed geodesic $c = [\gamma]$, the Margulis invariant $\alpha((\rho(\gamma),u(\gamma)))$.

\begin{theorem}[Goldman-Labourie-Margulis, Ghosh-Trieb]\label{thm:GLM-GT}
Suppose the linear part $\rho: \Gamma \to \SO(n,n-1)$ of the affine action $(\rho,u)$ is $P''_{n-1}$-Anosov. Then:
\begin{enumerate}
\item\label{item:extend} There exists a unique continuous linear functional $\alpha_{(\rho, u)}: \mathcal C(S) \to \RR$ such that for each $c =[\gamma] \in \mathcal{CG}(S)$, 
\begin{align}\label{eqn:alpha-function}
\alpha_{(\rho,u)}(\mu_c) = \alpha((\rho(\gamma), u(\gamma))).
\end{align}
\item\label{item:properness} The action $(\rho,u)$ of $\Gamma$ on $\EE^{n,n-1}$ is properly discontinuous if and only if $\alpha_{(\rho,u)}(\mu) \neq 0$ for all $\mu \in \mathcal C(S)$. 
\end{enumerate}

\end{theorem}

Note that Theorem~\ref{thm:GLM-GT}.\eqref{item:properness} implies the Opposite Sign Lemma~\ref{lem:osl}, since the space $\mathcal C(S)$ of currents is connected.

In order to study properly discontinuous affine actions with Anosov linear part, as in Theorem~\ref{thm:GLM-GT}, Ghosh-Trieb~\cite{GT} generalize the ideas of~\cite{GLM} and introduce a notion of \emph{affine Anosov}.
We will not recall their definition here. However, since it will be useful for the proof of Theorem~\ref{thm:main-flat}, let us explain the construction of the functional $\alpha_{(\rho,u)}$ in Theorem~\ref{thm:GLM-GT}.\eqref{item:extend}. 
Let $(\rho,u)$ as in the theorem statement.
In order to discuss Anosov properties of representations in $\Isom_+(\EE^{n,n-1})$, which is not reductive, we think of $\Isom_+(\EE^{n,n-1})$ as a subgroup of $\PSL(2n,\RR)$. Observe that since $\rho$ is $P_{n-1}''$-Anosov, the representation $(\rho,0)$, when viewed as a representation into $\PSL(2n,\RR)$, is Anosov with respect to the stabilizer $P_{n-1,n+1}$ in $\PSL(2n,\RR)$ of a flag made up of an $(n-1)$-space contained in a $(n+1)$-space, see Section~\ref{sec:Anosovinclusion}. 
Since $(\rho,u)$ is conjugate in $\PSL(2n,\RR)$ to $(\rho, \eps u)$ for any $\eps > 0$, and since $(\rho, \eps u) \to (\rho,0)$ as $\eps \to 0$, it follows from Fact \ref{fact:open} that $(\rho,u)$, when viewed as a representation into $\PSL(2n,\RR)$, is also $P_{n-1,n+1}$-Anosov. 
Let $\xi^{(n-1)}: \partial \Gamma \to \Gr_{n-1}(\RR^{2n})$ and $\xi^{(n+1)}: \partial \Gamma \to \Gr_{n+1}(\RR^{2n})$ denote the corresponding boundary maps. Then note that $\xi^{(n-1)}$ does not depend on the translational part~$u$; it is simply the composition of the boundary map for the $P_{n-1}''$-Anosov representation $\rho: \Gamma \to \SO(n,n-1)$ with the inclusion $\Gr_{n-1}(\RR^{n,n-1}) \to \Gr_{n-1}(\RR^{2n})$ induced by the inclusion $\RR^{n,n-1} = \RR^{2n-1} \to \RR^{2n}$ as the subspace orthogonal to $e_{2n}'$. However, $\xi^{(n+1)}$ does depend on~$u$. 

Together, the Anosov boundary maps define a splitting of the flat $\RR^{2n}$-bundle associated to $(\rho,u)$ into sub-bundles that are invariant under the geodesic flow $d\varphi_t$:
\begin{align}\label{eqn:decomp-2n}
V_{(\rho,u)} &= \Gamma \backslash(T^1 \widetilde S \times \RR^{2n}) =  V^+ \oplus V_0 \oplus V^-
\end{align}
where $V^+, V^-$ have rank $n-1$ and $V_0$ has rank two. Thought of as $(\rho,u)$-equivariant maps $V^\pm: T^1\widetilde{S}\to\Gr_{n-1}(\RR^{2n})$ and $V_0: T^1\widetilde{S}\to\Gr_2(\RR^{2n})$, the three maps depend only on the $y$ and $x$ coordinates of the point $(y,z,x) \in T^1 \widetilde S$.  Explicitly, $V^+(y,x) = \xi^{(n-1)}(x)$, $V^-(y,x) = \xi^{(n-1)}(y)$ and $V_0(y,x) = \xi^{(n+1)}(y) \cap \xi^{(n+1)}(x)$. 

Since $V^+(y,x), V^-(y,x)$ are each contained in $\RR^{2n-1}$, it follows that $V_0(y,x) \cap \RR^{2n-1} =: L_0(y,x)$ is one-dimensional. Indeed the decomposition $V^+(y,x) \oplus L_0(y,x) \oplus V^-(y,x) = \RR^{2n-1}$ does not depend on the translational part~$u$; it precisely induces the decomposition of the flat $\RR^{2n-1}$ bundle $V_\rho$ associated to $\rho$ coming from the Anosov boundary map $\xi^{(n-1)}$:
\begin{align}
V_\rho =  \Gamma \backslash(T^1 \widetilde S \times \RR^{2n-1}) =  V^+ \oplus L_0 \oplus V^-.
\end{align}

Note that $V^+(y,x), V^-(y,x)$ are null subspaces of $\RR^{n,n-1}$ and $L_0(y,x)$ is a positive line (meaning the restriction of $\langle \cdot, \cdot \rangle_{n,n-1}$ is positive definite) which is orthogonal to $V^+(y,x) \oplus V^-(y,x)$ in $\RR^{n,n-1}$.

For $(y,x) = (\gamma^-, \gamma^+)$ the pair of repelling and attracting fixed points for an element $\gamma \in \Gamma$, the subspaces $V^+, V^-, V_0, L_0$ corresponds precisely to those coming from the decomposition into generalized eigenspaces for $g = (\rho(\gamma, u(\gamma))$ of Section~\ref{sec:Margulis-invariant}:
\begin{align*}
V^\pm(\gamma^-, \gamma^+) &= V_{n-1}^\pm((\rho(\gamma), u(\gamma)))\\
V_0(\gamma^-, \gamma^+) &= V_0((\rho(\gamma), u(\gamma)))\\
L_0(\gamma^-, \gamma^+) &= L_0((\rho(\gamma), u(\gamma))).
\end{align*} 
By the discussion in Section~\ref{sec:Margulis-invariant}, $L_0(\gamma^-, \gamma^+)$ is an oriented line and using the same convention we define an orientation on $L_0(y,x)$ for all $y \neq x$ in $\partial \Gamma$. Hence, there is a unique positive unit vector $f_0(y,x)$ in each line $L_0(y,x)$, which defines the \emph{neutral section} $f_0: T^1 S \to V_\rho$.

Now, consider the flat $\EE^{n,n-1}$-bundle over $T^1S$,
\begin{align}
\mathsf{E}_{(\rho,u)} := \Gamma \backslash(T^1 \widetilde S \times \EE^{n,n-1}),
\end{align}
where here $\Gamma$ acts on the $\EE^{n,n-1}$ factor by the (affine) isometries $(\rho,u)$.
The geodesic flow $\varphi_t$ lifts in the usual way to a flow on $\mathsf{E}_{(\rho,u)}$ which is locally constant in the fiber. 
Let $s: T^1 S \to \mathsf{E}_{(\rho,u)}$ be a section which is differentiable along flow lines. The derivative $\nabla_{\varphi} s$ along the geodesic flow takes values in the vertical tangent bundle $T^{\mathrm{v}} \mathsf{E}_{(\rho,u)}$ of $\mathsf{E}_{(\rho,u)}$ which canonically identifies with the vector bundle $V_\rho$. Then for a current $\mu \in \mathcal C(S)$, define:
\begin{align}\label{eqn:defn-alpha}
\alpha_{(\rho,u)}(\mu) := \int_{\nu \in T^1 S} \left\langle \nabla_{\varphi}s , f_0 \right\rangle \ d\mu
\end{align}
where here $\langle \cdot, \cdot \rangle$ is the inner product on $V_\rho$ coming from the inner product $\langle \cdot, \cdot \rangle_{n,n-1}$ on $\RR^{n,n-1}$.
Let us  see Theorem~\ref{thm:GLM-GT}.\eqref{item:extend}: that~\eqref{eqn:defn-alpha} satisfies~\eqref{eqn:alpha-function} in the case that $\mu = \mu_c$ is the current associated to the closed geodesic $c \in \mathcal{CG}(S)$. In the following, $\D c: [0, \ell(c)] \to T^1 S$ is the tangent vector to the path traversing the geodesic $c$ at unit speed.
\begin{align}\label{eqn:calculus}
\alpha_{(\rho,u)}(\mu_c) &= \int_{\tau = 0}^{\ell(c)}  \left\langle (\nabla_\varphi s)(\D c(\tau)), f_0(\D c(\tau)) \right\rangle \ d\tau 
\end{align}
and the right-hand side may be evaluated by lifting to $T^1 \widetilde S$ where the bundles in consideration become products. Let $\widetilde s: T^1 \widetilde S \to \EE^{n,n-1}$ be the lift of the section $s$, a $(\rho, u)$-equivariant map. Choose a lift $\widetilde c$ of $c$ to $\widetilde S$ and let $\D \widetilde{c}: [0,\ell(c)] \to T^1 \widetilde S$ be the tangent vector to the unit speed parameterization of $\widetilde c$. Then the right-hand side of~\eqref{eqn:calculus} becomes
\begingroup
\addtolength{\jot}{0.5em}
\begin{align*}
\alpha_{(\rho,u)}(\mu_c)&=  \int_{\tau = 0}^{\ell(c)}  \left\langle (\nabla_\varphi \widetilde s)(\D \widetilde c(\tau)), f_0(\D \widetilde c(\tau)) \right\rangle \ d\tau
\\
&= \int_{\tau = 0}^{\ell(c)}  \left\langle \left.\frac{\D}{\D t}\right|_{t=0}\widetilde s(\D \widetilde{c}(\tau+t)), f_0(\D \widetilde{c}(\tau)) \right\rangle \ d\tau \\ &= \int_{\tau = 0}^{\ell(c)}  \left.\frac{\D}{\D t}\right|_{t=0}\left\langle \widetilde s(\D \widetilde{c}(\tau+t)), f_0(\D \widetilde{c}(\tau)) \right\rangle \ d\tau\\
 &=  \left\langle \widetilde s(\D\widetilde{c}(\ell(c))) - \widetilde s(\D \widetilde{c}(0)) , f_0(\gamma^-, \gamma^+) \right\rangle \\
 &= \left\langle \widetilde s(\gamma\cdot\D \widetilde{c}(0)) - \widetilde s(\D \widetilde{c}(0)) , f_0(\gamma^-, \gamma^+) \right\rangle\\ 
 &= \left\langle  (\rho(\gamma), u(\gamma))\cdot\widetilde s(\D \widetilde{c}(0)) - \widetilde s(\D \widetilde{c}(0)) , f_0(\gamma^-, \gamma^+) \right\rangle\\
 &= \alpha_{(\rho,u)}(c)
\end{align*}
\endgroup
where here $\gamma \in \Gamma$ is the element corresponding to the chosen lift $\widetilde c$ of $c$, we observe that $f_0(\D \widetilde{c}(\tau))= f_0(\gamma^-, \gamma^+)$ is independent of $\tau$, and we note the final equality follows from~\eqref{eqn:simplify-alpha}.

Before continuing to an analogous theory in the $\HH^{n,n-1}$ setting, let us first give a useful interpretation of Formula~\eqref{eqn:defn-alpha}. Formula~\eqref{eqn:defn-alpha} says that to calculate $\alpha_{(\rho,u)}(\mu)$, one first measures the infinitesimal signed progress (along the geodesic flow on $T^1S$) made by a section $s$ of $\mathsf{E}_{(\rho,u)}$ in the neutral direction $f_0$ above each point of $T^1 S$ , and then integrate it against $\mu$. 
We will now interpret the neutral vector $f_0$ as the vector field on $\EE^{n,n-1}$ whose pairing with a vector $v$ based at any point $\mathbf x \in \EE^{n,n-1}$ measures the projection of $v$ to an oriented \emph{translation axis} $\axisR \subset \EE^{n,n-1}$ parallel to~$f_0$.
We define the translation axis $\axisR(\nu)$ above $\nu\in T^1 S$ using the middle sub-bundle $V_0$ from the decomposition~\eqref{eqn:decomp-2n}. More precisely, suppose that $\nu\in T^1S$ lifts to $\widetilde{\nu}=(y,z,x)\in T^1\widetilde{S}$. Then
\begin{align}
\axisR(y,x) := \PP(V_0(y,x)) \cap \EE^{n,n-1}
\end{align}
is an affine line $\EE^{n,n-1}$ whose direction is $L_0(y,x)$.
Since $(y,x) \mapsto \axisR(y,x)$ is $(\rho,u)$-equivariant, $\axisR$ defines an affine Riemannian line, denoted by $\axisR(\nu)$, in the fiber $\EE^{n,n-1}_\nu$ of $\mathsf{E}_{(\rho,u)}$ which varies continuously with $\nu$.
The neutral vector $f_0(\nu)$ at $\nu \in T^1 S$, which is tangent to $\axisR(\nu)$, defines an orientation of the corresponding line $\axisR(\nu)$.
We call $\axisR(\nu)$ the \emph{translation axis} associated to $\nu \in T^1S$. Note that $\axisR(\nu)$ is locally constant under the geodesic flow $\varphi_t$.

Next consider any oriented Riemannian line $\axisR$ in $\EE^{n,n-1}$ and let $f_0 \in \RR^{n,n-1}$ denote a non-zero tangent vector to $\axisR$. Let $\Pi_\axisR: \EE^{n,n-1} \to \axisR$ denote the orthogonal projection, defined by the property that $\Pi_\axisR(\mathbf x)$ is the unique point in $\axisR$ so that $\langle \mathbf x - \Pi_\axisR(\mathbf x), f_0 \rangle_{n,n-1} = 0$.
Note that $\Pi_{\axisR}$ satisfies the equivariance property that for any $g \in \Isom_+(\EE^{n,n-1})$, $\Pi_{g\cdot\axisR}(g\cdot\mathbf x) = g\cdot\Pi_{\axisR}(\mathbf x)$. In particular, if $\axisR$ is invariant under $g$, then $\Pi_{\axisR}(g\cdot\mathbf x) = g\Pi_{\axisR}(\mathbf x)$.
Note that the function $(y,z,x) \mapsto \Pi_{\axisR(y,x)}$ is $(\rho,u)$-equivariant and hence descends to $T^1S$ giving a continuous assignment of a projection map $\Pi_{\axisR(\nu)}: \EE^{n,n-1}_\nu \to \axisR(\nu)$ in the fiber above $\nu \in T^1S$.
Observe that for any vector $v \in T_{\mathbf x} \EE^{n,n-1}$, 
\begin{align*}
 \metric^\EE_{\mathbf x}(v, f_0) 
 &= \metric^\EE_{\Pi_{\axisR}\mathbf x}(\D \Pi_{\axisR}v, f_0)
 \end{align*}
 where $\metric^\EE$ denotes the flat metric on $\EE^{n,n-1}$, and we interpret $f_0 \in \RR^{n,n-1}$ as a parallel vector field on $\EE^{n,n-1}$. Hence we may rewrite formula~\eqref{eqn:defn-alpha} as follows (see Figure~\ref{fig:diagram1}):
 \begin{align}\label{eqn:defn-alpha-alt}
\alpha_{(\rho,u)}(\mu) &= \int_{\nu \in T^1 S} \metric^\EE (\nabla_{\varphi}s , f_0 ) \ d\mu\\
&= \int_{\nu \in T^1 S} \metric^\EE \left(\D \Pi_{\axisR(\nu)} \left((\nabla_{\varphi}s)(\nu)\right) , f_0(\nu) \right) \ d\mu\nonumber.
\end{align}

\begin{figure}[h]
\centering

\vspace{0.4cm}
\def\svgwidth{7.0cm}
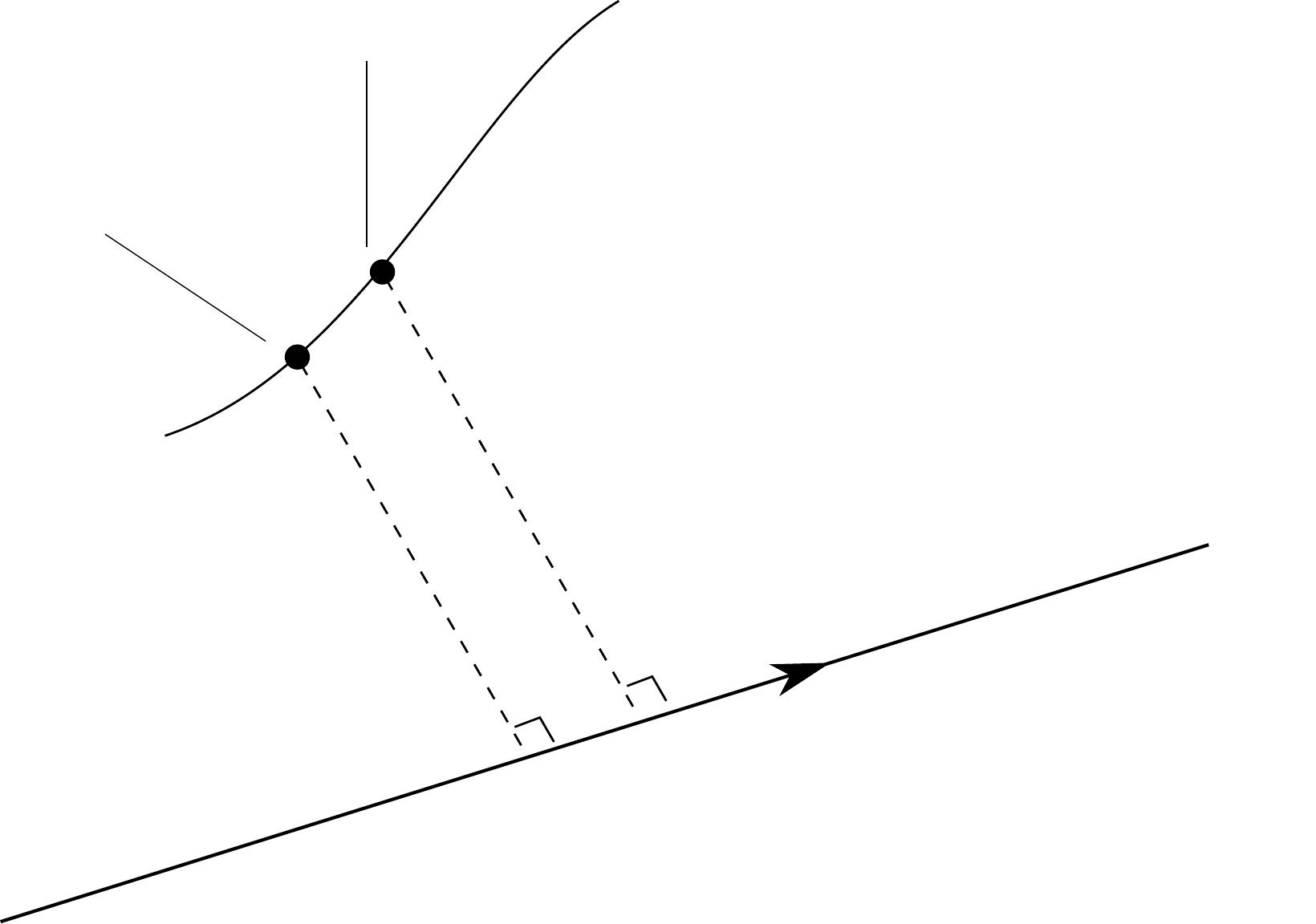

\caption[]{The Margulis invariant $\alpha_{(\rho,u)}(\mu)$ is the rate at which the projection of $s(\nu)$ to the translation axis $\axisR(\nu)$ makes progress under the geodesic flow, averaged over $\nu \in T^1 S$ against the current $\mu$. Here we use the flat connection to identify the fibers of $\mathsf E_{(\rho,u)}$ above the flow line $\varphi_t \nu$ with a fixed copy of $\EE^{n,n-1}$ and note that the translation axis $\axisR(\varphi_t \nu) = \axisR(\nu)$ is constant in $t$.} \label{fig:diagram1}
\end{figure}

\subsection{Extending the length function $\mathscr L$ for $\HH^{n,n-1}$ to currents}
\label{sec:n-1n+1}

We now give an analogue of the construction from the previous section in the setting of $\HH^{n,n-1}$ geometry.
We follow the notation conventions of Section~\ref{sec:length-function-H} and Example~\ref{eg:Jordan projection}, thinking of $G' = \PSO(n,n)$ as embedded in $G= \PSL(2n,\RR)$.

Consider a representation $\varrho:\Gamma\to G'$ which is Anosov with respect to the stabilizer $P_{n-1}'$ of an isotropic $(n-1)$-plane and form the flat $\HH^{n,n-1}$ bundle associated to $\varrho$:
\begin{align*}
\mathsf{H}_\varrho := \Gamma \backslash (T^1 \widetilde S \times \HH^{n,n-1}).
\end{align*}
As usual, we lift the geodesic flow $\varphi_t$ to $\mathsf{H}_\varrho$ so that it is locally constant in the fiber.
Suppose now that there is a differentiable section $s: T^1 S \to \mathsf{H}_\varrho$ (such a section exists in the setting where we will apply this later). The present goal will be to work by analogy to~\eqref{eqn:defn-alpha-alt} and use the variation of the section $s$ along the geodesic flow to define a continuous length functional $\mathscr L_\varrho$ on the space of geodesic currents $\mathcal C(S)$, that satisfies
\begin{align*}
\mathscr L_\varrho(\mu_c) = \mathscr L(\varrho(\gamma)).
\end{align*}
Here, $[\gamma] = c$ and the function $\mathscr L$ of Section~\ref{sec:length-function-H} is well-defined on $\varrho(\gamma)$ since $\varrho$ is $P_{n-1}'$-Anosov. 

Let $\xi^{(n-1)}: \partial \Gamma \to \Gr_{n-1}(\RR^{n,n})$ and $\xi^{(n+1)}: \partial \Gamma \to \Gr_{n+1}(\RR^{n,n})$ be the associated Anosov boundary maps and let $\xi^{(n)}_+:  \partial \Gamma \to \Gr_{n}^+(\RR^{n,n})$ and $\xi^{(n)}_-: \partial \Gamma \to \Gr_{n}^-(\RR^{n,n})$ be the maps
defined by $\xi^{(n)}_\pm = \varpi_\pm \circ \xi^{(n-1)}$, where $\varpi_+$ (resp. $\varpi_-$) is the projection taking an isotropic $(n-1)$-plane to the unique positive (resp. negative) istropic $n$-plane containing it, see Proposition~\ref{prop:projections}. Then for each pair $(y,x)$ of distinct points in $\partial \Gamma$, define
\begin{align}
L_n^+(y,x) &= \xi^{(n)}_+(x) \cap \xi^{(n+1)}(y)\\
L_n^-(y,x) &= \xi^{(n)}_-(x) \cap \xi^{(n+1)}(y).
\end{align}

Note that in the case that $(y,x) = (\gamma^-, \gamma^+)$ are the repelling and attracting fixed points for an element $\gamma \in \Gamma$, we have:
\begin{align*}
 L_n^+(\gamma^-, \gamma^+) &= L_n^+(\varrho(\gamma))\\
 L_n^-(\gamma^-, \gamma^+) &= L_n^-(\varrho(\gamma))
 \end{align*} where $L_n^\pm(g)$ is as defined in Section~\ref{sec:length-function-H}.
Now, define 
\begin{align}
\axisH_\varrho(y,x)=\axisH(y,x) := \PP( L_n^+(y,x) \oplus L_n^-(y,x)) \cap \HH^{n,n-1}
\end{align}
Then $\axisH(\gamma^-, \gamma^+)$ is the slow axis for $\varrho(\gamma)$ as in Section~\ref{sec:length-function-H}. For any point $(y,z,x) \in T^1 \widetilde S$, we call $\axisH(y,x)$ the \emph{slow axis} associated to $(y,z,x)$; it is invariant under $\varphi_t$ (\ie it is independent of $z$). Further, by the same convention as for $\axisH(\gamma^-, \gamma^+)$, described in Section~\ref{sec:length-function-H}, the axis $\axisH(y,x)$ is endowed with a natural orientation, namely that for which $\PP(L_n^+(y,x))$ is the forward endpoint and $\PP(L_n^-(y,x))$ the backward endpoint. We equip $\axisH(y,x)$ with the (anti-symmetric) signed distance function $d_{\axisH(y,x)}(\cdot, \cdot)$, where $d_{\axisH(y,x)}([v],[w])$ is plus/minus the Riemannian distance with positive sign if and only $([v],[w])$ is positive for the orientation. Let $\nu\in T^1S$ be the point that lifts to $\widetilde{\nu}=(y,z,x)\in T^1\widetilde{S}$. Since the construction of $\axisH(y,x)$ is equivariant, $\axisH$ descends to $T^1 S$, giving a smooth assignment of an oriented Riemannian axis $\axisH(\nu)$ in the fiber $\HH^{n,n-1}_\nu$ above $\nu$.

Next, consider any oriented Riemannian geodesic axis $\axisH$ in $\HH^{n,n-1}$. Write $\axisH = \PP(L^- \oplus L^+)$ where $\PP(L^-), \PP(L^+) \in \partial \HH^{n,n-1}$ are the negative and positive endpoints of $\axisH$ respectively. 
Choose $f^+ \in L^+$ and $f^- \in L^-$ so that $\langle f^+, f^- \rangle = -1$. Define $\mathcal U(\axisH) \subset \HH^{n,n-1}$ to be the open subset of points $[v] \in \HH^{n,n-1}$ so that $\langle v, f^+\rangle\langle v, f^- \rangle > 0$ and note that $\mathcal U(\axisH)$ is independent of the choice of $f^+, f^-$ as above. The region $\mathcal U(\axisH)$ is a maximal neighborhood of the axis $\axisH$ on which the following ``nearest point" projection is defined. Let  $\Pi_{\axisH}: \mathcal U(\axisH) \to \axisH$ be given by the formula:
\begin{align}
\Pi_{\axisH}([w]) := \left[-\langle w, f^+ \rangle f^- - \langle w, f^-\rangle f^+\right]
\end{align}
Again, note that $\Pi_{\axisH}$ is independent of the choice of $f^+, f^-$ as above, and note also that $\Pi_{\axisH}$ is smooth and varies smoothly as $\axisH$ varies.
Further, note that $\Pi_{\axisH}$ satisfies the equivariance property that for any $g \in \PSO(n,n)$, $\Pi_{g\axisH}(g\cdot[v]) = g\cdot\Pi_{\axisH}([v])$. In particular, if $\axisH$ is invariant under $g$, then $\Pi_{\axisH}(g\cdot[v]) = g\Pi_{\axisH}([v])$.
Note that the functions $(y,z,x) \mapsto \mathcal U(\axisH(y,x))$ and $(y,z,x) \mapsto \Pi_{\axisH(y,x)}$ are $\varrho$-equivariant and hence descend to $T^1S$ giving a smooth assignment of an open neighborhood $\mathcal U(\axisH(\nu))$ of $\axisH(\nu)$, and a projection map $\Pi_{\axisH(\nu)}$ of the fiber above $\nu \in T^1S$ to the axis $\axisH(\nu)$ in that fiber.

Next, let $\axisH$ denote any oriented Riemannian line in $\HH^{n,n-1}$. 
We define the vector field $f = f_\axisH$ on $\mathcal U(\axisH)$ to be the extension of the unit tangent field to $\axisH$ that satisfies that $\D \Pi_{\axisH} f_{\mathbf x} = f_{\Pi_{\axisH}\mathbf x}$ for any $\mathbf x \in \mathcal U(\axisH)$, and that $f$ is orthogonal to the kernel of the projection $\D \Pi_{\axisH}$. 
Hence
\begin{align}\label{eqn:f}
\metric^\HH_{\mathbf x}(v, f) &= \metric^\HH_{\Pi_{\axisH} \mathbf x}( \D \Pi_{\axisH} v, f )
\end{align}
holds for any tangent vector $v \in T_{\mathbf x} \mathcal U(\axisH)$, where
here $\metric^\HH$ denotes the invariant metric on $\HH^{n,n-1}$ of constant curvature $-1$. 
Then for a path $\mathbf x(t)$ in $\mathcal U(\axisH)$, the amount of infinitesimal signed progress the projection $\Pi_{\axisH}(\mathbf x(t))$ is making along $\axisH$ at time $t = \tau$ may be expressed as follows:
\begin{align}
\left.\frac{\D}{\D t}\right|_{t=0} d_{\axisH}\left(\Pi_{\axisH} \mathbf x(\tau),\Pi_{\axisH} \mathbf x(\tau + t)\right)
&= \metric^\HH_{\Pi_{\axisH} \mathbf x(\tau)} ( \D \Pi_{\axisH} \mathbf x'(\tau), f ) \nonumber\\
& = \metric^\HH_{\mathbf x(\tau)}( \mathbf x'(\tau), f ). \label{eqn:derivative}
\end{align}

%

We now give the definition of the length function. Suppose the differentiable section $s: T^1 S \to \mathsf{H}_\varrho$ satisfies that $s(\nu) \subset \mathcal U(\axisH(\nu))$ for all $\nu \in T^1 S$.
Define the function $\mathscr L_\varrho: \mathcal C(S) \to \RR$ by the formula
\begin{align}\label{eqn:defn-L}
\mathscr L_\varrho(\mu) := \int_{\nu \in T^1 S} \metric^\HH( \nabla_\varphi s, f ) d \mu 
\end{align}
where here $\nabla_\varphi s$ is the derivative of $s$ in the flow direction using the flat connection, $f = f(\axisH(\nu))$ is the vector field defined as above in the subset $\mathcal U(\axisH(\nu))$ of the fiber above $\nu \in T^1 S$, and $\metric^\HH$ is the natural metric of constant curvature $-1$ on the fiber $\HH^{n,n-1}_\nu$ above $\nu$.
See Figure~\ref{fig:diagram2}.
\begin{figure}[h]
\centering

\def\svgwidth{4.8cm}
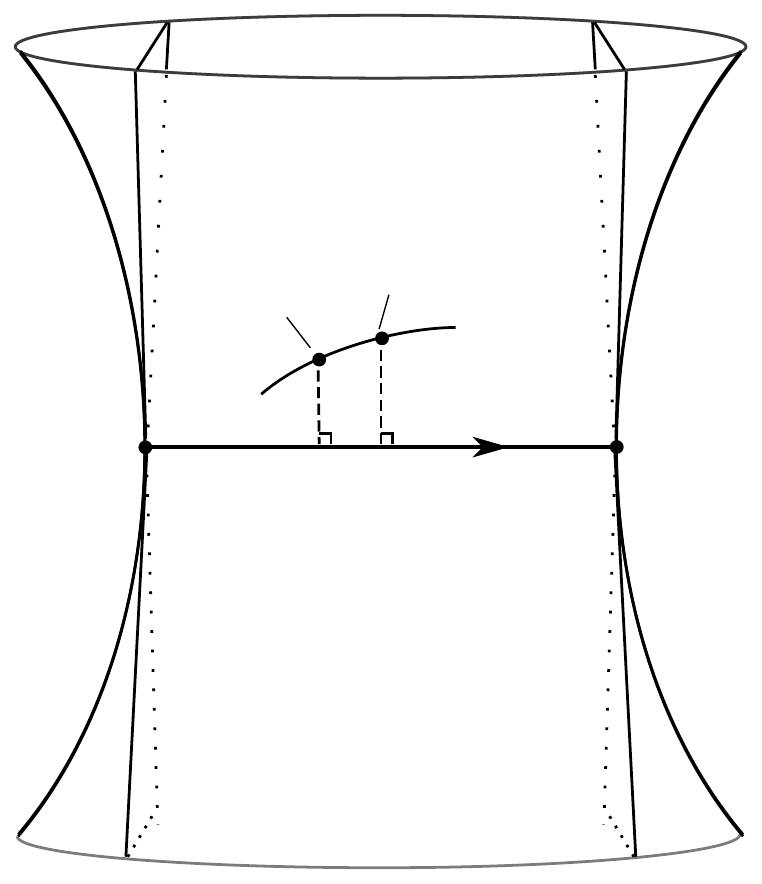

\caption[]{The length function $L_\varrho(\mu)$ is the rate at which the projection of $s(\nu)$ to the slow axis $\axisH(\nu)$ makes progress under the geodesic flow, averaged over $\nu \in T^1 S$ against the current $\mu$. Here we use the flat connection to identify the fibers of $\mathsf H_{(\rho,u)}$ above the flow line $\varphi_t \nu$ with a fixed copy of $\HH^{n,n-1}$ and note that the slow axis $\axisH(\varphi_t \nu) = \axisH(\nu)$ is constant in $t$.} \label{fig:diagram2}
\end{figure}

The function $\mathscr L$ is clearly continuous and linear. Further:

\begin{proposition}\label{prop:Lperiods}
Suppose $\varrho$ is $P_{n-1}'$-Anosov and $s: T^1 S \to \mathsf{H}_\varrho$ is a section so that $s(\nu) \in \mathcal U(\axisH(\nu))$ for all $\nu \in T^1 S$ as above. Then, for any $c = [\gamma] \in \mathcal{CG}(S)$, 
\begin{align*}
\mathscr L_\varrho(\mu_c) = \mathscr L(\varrho(\gamma)).
\end{align*}
\end{proposition}

\begin{proof}
In the following, $\D c: [0, \ell(c)] \to T^1 S$ is the tangent vector to the path traversing the geodesic $c$ at unit speed. Then
\begin{align}\nonumber
\mathscr L_\varrho(\mu_c) &= \int_{\nu \in T^1 S} \metric^\HH( \nabla_\varphi s, f ) \ d \mu_c 
\\
&= \int_{\tau = 0}^{\ell(c)} \metric^\HH\left( (\nabla_\varphi s)(\D c(\tau)), f(\axisH(\D c(\tau))) \right) d \tau\nonumber\\
\label{eqn:calculus2} &= \int_{\tau = 0}^{\ell(c)}  \left.\frac{\D}{\D t}\right|_{t=0} d_{\axisH(\D c(\tau))}(\Pi_{\axisH(\D c(\tau))} s(\D c(\tau+t)), \Pi_{\axisH(\D c(\tau))} s(\D c(\tau)) \ d\tau 
\end{align}
where the last equality follows from~\eqref{eqn:derivative}. The right-hand side of~\eqref{eqn:calculus2} may be evaluated by lifting to $T^1 \widetilde S$ where the bundle in consideration becomes a product. Let $\widetilde s: T^1 \widetilde S \to \HH^{n,n-1}$ be the $\varrho$-equivariant map lifting~$s$. Choose a lift $\widetilde c$ of $c$ to $\widetilde S$, corresponding to an element $\gamma \in \Gamma$ with $[\gamma] = c$ and let $\D \widetilde{c}: [0,\ell] \to T^1 \widetilde S$ be the tangent vector to the unit speed parameterization of $\widetilde c$, where here $\ell = \ell(c)$ is the length of the closed geodesic $c$ on $S$. Then the right-hand side of~\eqref{eqn:calculus2} becomes
\begingroup
\addtolength{\jot}{0.5em}
\begin{align*}
\mathscr L_\varrho(\mu_c) &=
 \int_{\tau = 0}^{\ell}  \left.\frac{\D}{\D t}\right|_{t=0} d_{\axisH(\D \widetilde c(\tau))}(\Pi_{\axisH(\D \widetilde c(\tau))} \widetilde s(\D \widetilde c(\tau+t)), \Pi_{\axisH(\D \widetilde c(\tau))}  \widetilde s(\D \widetilde c(\tau)) \ d\tau 
 \\
 &= d_{\axisH(\gamma^-, \gamma^+)}(\Pi_{\axisH(\gamma^-, \gamma^+)}  \widetilde s(\D \widetilde c(\ell)), \Pi_{\axisH(\gamma^-, \gamma^+)}  \widetilde s(\D \widetilde c(0)))\\
 &=  d_{\axisH(\gamma^-, \gamma^+)}(\Pi_{\axisH(\gamma^-, \gamma^+)}  \widetilde s(\gamma.\D \widetilde c(0)), \Pi_{\axisH(\gamma^-, \gamma^+)}  \widetilde s(\D \widetilde c(0)))\\
&=  d_{\axisH(\gamma^-, \gamma^+)}\left(\Pi_{\axisH(\gamma^-, \gamma^+)} \varrho(\gamma)\cdot \widetilde s(\D \widetilde c(0)), \Pi_{\axisH(\gamma^-, \gamma^+)}  \widetilde s(\D \widetilde c(0))\right)\\
&= d_{\axisH(\gamma^-, \gamma^+)}\left(\varrho(\gamma)\cdot\Pi_{\axisH(\gamma^-, \gamma^+)}  \widetilde s(\D \widetilde c(0)), \Pi_{\axisH(\gamma^-, \gamma^+)}  \widetilde s(\D \widetilde c(0))\right)\\
&=  \mathscr L(\varrho(\gamma)).
\end{align*}
\endgroup
where here we observe that the axis $\axisH(\D \widetilde c(\tau)) = \axisH(\gamma^-, \gamma^+)$ is constant in the integral and the fundamental theorem of calculus is applied in the first step above to the signed distance function $d_{\axisH(\gamma^-, \gamma^+)}(\cdot, \D \widetilde c(0))$ on the axis $\axisH(\gamma^-, \gamma^+)$. The final two equalities follow respectively from the $\varrho$-equivariance of $s$ and the equivariance property of $\Pi_\axisH$ discussed above.
\end{proof}

Finally, let $G'' = \SO(n,n-1)$ be embedded in $G' = \PSO(n,n)$ via the inclusion $\iota_{n,n}: G'' \hookrightarrow G'$ as described in Example~\ref{eg:Jordan projection2}. Recall from Example \ref{eg:include} that if $\rho: \Gamma \to G''$ is Anosov with respect to the stabilizer $P_{n-1}'' < G''$ of an isotropic $(n-1)$-plane in $\RR^{n,n-1}$, then $\iota_{n,n} \circ \rho$ is Anosov with respect to the stabilizer $P_{n-1}' < G'$ of an isotropic $(n-1)$-plane in $\RR^{n,n}$.

 \begin{lemma}\label{lem:compact-section}
Let $\rho: \Gamma \to G''$ be $P_{n-1}''$-Anosov. Then for $\varrho: \Gamma \to G'$ close enough to $\iota_{n,n} \circ \rho$, there exists a differentiable section $s: T^1S \to \mathsf{H}_\varrho$ such that $s(\nu) \in \mathcal U(\axisH(\nu))$ for all $\nu \in T^1 S$, and hence $\mathscr L_\varrho: \mathcal C(S) \to \RR$ is well-defined. 
 \end{lemma}
 
 \begin{proof}
 For $\varrho_0 = \iota \circ \rho$, such a section exists, namely the projection $s_0$ of the $(\iota \circ \rho)$-equivariant map $\widetilde s_0: T^1 \widetilde S \to \HH^{n,n-1}$, defined by $\widetilde s_0(\nu) = [e_{2n}]$ constant. Indeed, in this case $[e_{2n}] \in \axisH(y,x) \subset \mathcal U(\axisH(y,x))$ for all pairs $(y,x)$ of distinct points in $\partial \Gamma$, because $[e_{2n}]=[f_++f_-]$.

Now consider a path $\varrho_t: \Gamma \to G'$ based at $\varrho_0 = \iota_{n,n} \circ \rho$. The bundles $\mathsf{H}_{\varrho_t}$ are all isomorphic as smooth fiber bundles, so we may regard $\mathsf{H}_{\varrho_t}$ as a fixed fiber bundle with continuously varying flat structure. Hence, the section $s_0$ may be regarded as a differentiable  section of any of the bundles~$\mathsf{H}_{\varrho_t}$. The open subsets $\mathcal U(\axisH(\nu))$ in the fiber over $\nu$ of $\mathsf{H}_{\varrho_t}$ also depend on~$t$, and the dependence is continuous because the dependence of the Anosov boundary map $\xi_{\varrho_t}: \partial \Gamma \to \Gr_{n-1}(\RR^{n,n})$ on $t$ is continuous. More precisely, the union $\bigcup_{\nu\in T^1S}\HH^{n,n-1}_\nu\setminus\mathcal U(\axisH(\nu))$ is a closed subset of the bundle $\mathsf{H}_{\varrho_t}$ that varies continuously in $t$ in the topology of uniform convergence on compact subsets. Hence, since $s_0(T^1 S)$ is compact, it remains contained in the union $\bigcup_{\nu\in T^1S}\mathcal U(\axisH(\nu))$, provided that $t$ is sufficiently small. \end{proof}
 
 \begin{remark}
By the same argument given in Section 6.2 of Goldman-Labourie-Margulis \cite{GLM}, one shows that $\mathscr L_\varrho$ does not depend on the section~$s$. We do not give that argument here as we do not need it for our purposes. 
\end{remark}

\section{$\EE^{n,n-1}$ as a geometric limit of $\HH^{n,n-1}$ \label{sec:geometric limit}}
We now give the crucial geometric input needed for the main result, namely the understanding of group actions by isometries of the psuedo-Riemannian Euclidean space $\EE^{n,n-1}$ as limits of group actions by isometries on the pseudo-Riemannian hyperbolic space $\HH^{n,n-1}$.
This geometric transition interpretation, which follows the work of Danciger-Gu\'eritaud-Kassel~\cite{DGK1} in the setting of free groups acting on $\RR^{2,1}$, will be used to make a connection to Theorem~\ref{thm:main2} in order to prove Theorem~\ref{thm:main-flat} and eventually Theorem~\ref{thm:main-affine}.

\subsection{$\EE^{n,n-1}$ as a limit of $\HH^{n,n-1}$ in real projective geometry}

We continue to work with the coordinates of Section~\ref{sec:subgeometries}, in which $\EE^{n,n-1}$ and $\HH^{n,n-1}$ are embedded in $\PP(\RR^{2n})$ with the isometry groups $\Isom_+(\EE^{n,n-1}) = \SO(n,n-1) \ltimes \RR^{2n-1}$ and $\Isom_+(\HH^{n,n-1}) = \PSO(n,n)$ embedded in $\PSL(2n,\RR)$.

Consider a differentiable path $r \mapsto g_r$ in $\PSO(n,n)$ based at $g_0 = \iota(h)$, where $h \in \SO(n,n-1)$ and where $\iota = \iota_{n,n}: \SO(n,n-1) \hookrightarrow \PSO(n,n)$ is the inclusion as in Section~\ref{sec:subgeometries}. We write $g_r$ in the form:
\begin{align*}
g_r =\left[\begin{array}{cc}
A_r &v_r\\
w_r^T&b_r
\end{array}\right],
\end{align*}
where $A_r$ is a $(2n-1)\times(2n-1)$ matrix, $v_r , w_r \in \RR^{2n-1}$,and $b_r \in \RR$. (These are well-defined up to simultaneously changing signs.) Since $g_0=\iota(h)$, we see that $A_0= h$, $v_0=0$, $w_0=0$ and $b_0=1$. 

Let $c_r:\PP(\RR^{2n})\to\PP(\RR^{2n})$ be the projective transformation given by the matrix
\[c_r=\left[\begin{array}{cc}
\frac{1}{r}\Id_{2n-1}&0\\
0&1
\end{array}\right],\]
where $\Id_{2n-1}$ is the $(2n-1)\times(2n-1)$ identity matrix. Then observe that
\begin{align}\lim_{r\to 0} c_r g_r c_r^{-1} &=
\lim_{r\to 0}\left[\begin{array}{cc}
A_r&\frac{1}{r}\cdot v_r\\
r\cdot w_r^T&b_r
\end{array}\right]
=\left[\begin{array}{cc}
h & u \\
0&1
\end{array}\right],\label{eqn:limit-rep}
\end{align}
is the element $(h,u) \in \SO(n,n-1) \ltimes \RR^{2n-1} =  \Isom_+(\RR^{n,n-1})$, where $u:=\frac{d}{dr}\Big|_{r=0}v_r$ (where $v_r$ is chosen with the appropriate sign). 
This (essentially) shows that $c_r \PSO(n,n) c_r^{-1}$ converges as $r \to 0$ to $\SO(n,n-1) \ltimes \RR^{2n-1}$ in the Chabauty topology on closed subgroups of $\PSL(2n,\RR)$. In fact, the action of $\PSO(n,n)$ on $\HH^{n,n-1}$ converges to the action of $\SO(n,n-1) \ltimes \RR^{2n-1}$ on $\EE^{n,n-1}$ under conjugation by $c_r$ in the following sense. 
Let $r\mapsto \mathbf x_r$ be a differentiable path in $\HH^{n,n-1}$ based at the basepoint $\mathbf x_0 = [0:\ldots: 0: 1] \in \HH^{n,n-1}$ which is stabilized by $\iota(\SO(n,n-1))$. For sufficiently small $r$,  $c_r\mathbf x_r $ lies in $\EE^{n,n-1}$, so $c_r\mathbf x_r \to \mathbf x'$ as $r \to 0$ for some $\mathbf x'\in \EE^{n,n-1}$. Thinking of $\RR^{n,n-1} = T_{\mathbf x_0}\HH^{n,n-1}$, the tangent vector to the path $\mathbf x_r$ at $r = 0$ is precisely the displacement vector between $\mathbf x'$ and the basepoint.  Next,
\begin{align*}
c_r g_r \mathbf x_r &= c_r g_r c_r^{-1} (c_r \mathbf x_r)\\
&\to (h,u) \cdot \mathbf x'
\end{align*} 
as $r\to 0$. Hence the geometry $\EE^{n,n-1}$ is a \emph{geometric limit} of $\HH^{n,n-1}$ as sub-geometries of real projective geometry, in the sense of Cooper-Danciger-Wienhard~\cite{CDW}.

Now, let $\rho:\Gamma\to\SO(n,n-1)$ be a representation and let $\varrho_r:\Gamma\to\PSO(n,n)$ be a differentiable path of representations so that $\varrho_0=\iota \circ\rho$.
Define ${\varrho}_r^{c_r}:\Gamma\to\PSL(2n,\RR)$ by \[{\varrho}_r^{c_r}(\gamma):=c_r\cdot\varrho_r(\gamma)\cdot c_r^{-1}.\]
By the above calculation, 
\begin{align*}
\lim_{r\to 0} \varrho_r^{c_r} &= (\rho, u)
\end{align*}
is a representation into $\SO(n,n-1) \ltimes \RR^{2n-1}$ with linear part $\rho$ and translational part the $\rho$-cocycle $u: \Gamma \to \RR$.

\begin{lemma}\label{lem:smooth-point}
If $(\rho,u): \Gamma \to \SO(n,n-1) \ltimes \RR^{2n-1}$ is any surface group representation with irreducible linear part $\rho$, then there exists a path $\varrho_r: \Gamma \to \PSO(n,n)$ so that $\varrho_0 = \iota \circ \rho$ and $\lim_{r\to 0} \varrho_r^{c_r} = (\rho, u)$ as above.
\end{lemma}
\begin{proof}
Since $\rho$ is irreducible, $\iota \circ \rho: \Gamma \to \PSO(n,n)$ has finite centralizer. Hence $\iota \circ \rho$ is a smooth point of $\Hom(\Gamma, \PSO(n,n))$ by~Goldman~\cite{Goldmansymplectic}. Hence any tangent direction to $\iota \circ \rho$ is integrable, in particular the tangent direction defined by the $\mathfrak{pso}(n,n)$ valued cocycle
\begin{align*}
\gamma &\mapsto  \left(\begin{array}{cc}
0 & u(\gamma) \\
u^T(\gamma)J  &0
\end{array}\right) \in \mathfrak{pso}(n,n),
\end{align*}
where $J = \Id_n \oplus (-\Id_{n-1})$ is the matrix for the form $\langle \cdot, \cdot \rangle_{n,n-1}$ (in the basis $e_1',\dots,e_{2n-1}'$). Any path $\varrho_r$ tangent to this direction satisfies the conclusion of the lemma.
\end{proof}

\subsection{The derivative formula}

We now state and prove the key lemma. In the following $P_{n-1}'' < G'' = \SO(n,n-1)$ denotes the stabilizer of an isotropic $(n-1)$-plane in $\RR^{n,n-1}$ and $P_{n-1}' < G' = \PSO(n,n)$ denotes the stabilizer of an isotropic $(n-1)$-plane in $\RR^{n,n}$.

\begin{lemma}\label{lem:key}
Let $(\rho,u): \Gamma \to \SO(n,n-1) \ltimes \RR^{2n-1}$ be any representation whose linear part $\rho: \Gamma \to \SO(n,n-1)$ is $P_{n-1}''$-Anosov. Let $\varrho_r: \Gamma \to \PSO(n,n)$ be a differentiable path based at $\varrho_0 = \iota \circ \rho$ and satisfying $\lim_{r\to 0} \varrho_r^{c_r} = (\rho, u)$. Then the length functions $\alpha_{(\rho, u)}, \mathscr L_{\varrho_r}: \mathcal C(S) \to \RR$ of Sections~\ref{sec:Marg-currents} and~\ref{sec:n-1n+1} satisfy:
\begin{align}\label{eqn:key}
\lim_{r \to 0} \frac{1}{r} \mathscr L_{\varrho_r}(\cdot) = \alpha_{(\rho,u)}(\cdot)
\end{align}
and the convergence is uniform on compact subsets of $\mathcal C(S)$, in particular on the probability currents $\mathcal C_1(S)$.
\end{lemma}

Note that for $r$ sufficiently small, $\varrho_r$ is $P_{n-1}'$-Anosov by Fact~\ref{fact:open}, and $\mathscr L_{\varrho_r}$ is well-defined by Lemma~\ref{lem:compact-section}.

It is easy to verify that, in the context of the Lemma, 
\begin{align*}
\lim_{r \to 0} \frac{1}{r} \mathscr L(\varrho_r(\gamma)) = \alpha((\rho,u)(\gamma)),
\end{align*}
for any $\gamma \in \Gamma$, hence the formula~\eqref{eqn:key} holds pointwise on currents $\mu_c$ supported on closed geodesics. The difficulty is to show the uniform convergence. In order to do this, we show that the integrand from Equation~\eqref{eqn:defn-L} defining $\mathscr L$ may be arranged to, after rescaling, converge uniformly to the integrand from Equation~\eqref{eqn:defn-alpha}.

We must first examine a bit more carefully the notion that the geometries $c_r \HH^{n,n-1}$ converge to $\EE^{n,n-1}$, as sub-geometries of real projective geometry.
First we note that $c_s \HH^{n,n-1} \subset c_r \HH^{n,n-1}$ whenever $0 < r < s$, and that 
\begin{align}\label{eqn:union}
\bigcup_{r \to 0} c_r \HH^{n,n-1} \supset \EE^{n,n-1}.
\end{align}

Next, we prove a statement about convergence of metrics, analogous to \cite[\S 7.4]{DGK1}.
The space $\EE^{n,n-1}$ (respectively $\HH^{n,n-1}$) admits a pseudo-Riemannian metric of zero (respectively constant negative) curvature which is invariant under the group $\Isom_+(\EE^{n,n-1}) = \SO(n,n-1) \times \RR^{2n-1}$ (respectively $\Isom_+(\HH^{n,n-1}) = \PSO(n,n)$). As in Section~\ref{sec:length-functions}, we denote these metrics by $\metric^{\EE}$ and $\metric^\HH$, and we view $\Isom_+(\EE^{n,n-1})$ and $\Isom_+(\HH^{n,n-1}) = \PSO(n,n)$ as subgroups of $\PSL(2n,\RR)$. Since the stabilizer of the basepoint $\mathbf x_0 = [0,\ldots,0,1]$ is the same in both isometry groups, we  arrange that $\metric^{\EE}_{\mathbf x_0} = \metric^\HH_{\mathbf x_0}$.
For $r>0$, consider the metric $\metric^r$ defined on $c_r \cdot \HH^{n,n-1} \subset \PP(\RR^{2n})$ by
$$\metric^r := (c_r)_* \, \metric^{\HH},$$
where $(c_r)_*$ is the pushforward by~$c_r$.

\begin{lemma}\label{lem:metrics}
The sequence of metrics $r^{-2}\,\metric^r$ converges to $\metric^{\EE}$ uniformly on compact subsets of $\EE^{n,n-1}$ as $r \to 0$, where for a given compact set~$C\subset\EE^{n,n-1}$ we only consider $r$ small enough so that $\metric^r$ is defined on~$C$.
\end{lemma}

\begin{proof}
In what follows, we use the trivialization of the tangent bundle $T\EE^{n,n-1}$ to the affine chart $\EE^{n,n-1}$, denoting the associated parallel transport of a vector $v \in T_{\mathbf x} \EE^{n,n-1}$ to $T_{\mathbf y} \EE^{n,n-1}$ again by~$v$.
First, note that for any tangent vector $v \in T_\mathbf x \EE^{n,n-1}$ we have $(c_r^{-1})_*v = r v \in T_{c_r^{-1}(\mathbf x)}\EE^{n,n-1}$.
Thus, for $v,w \in T_{\mathbf x} \EE^{n,n-1}$,
\begin{align*}
r^{-2}\,\metric^r_{\mathbf x}(v,w) &= r^{-2}\,\big((c_r)_* \metric^\HH\big)_{\mathbf x} (v,w)\\
&= r^{-2}\ \metric^\HH_{c_r^{-1}({\mathbf x})}\big((c_r^{-1})_*v, (c_r^{-1})_* w\big)\\
&= \metric^\HH_{c_r^{-1}({\mathbf x})}(v,w).
\end{align*}
Given a compact set $C\subset \EE^{n,n-1}$, the projective transformation $c_r^{-1}$ maps $C$ into arbitrarily small neighborhoods of the basepoint~$\mathbf x_0$ as $r \to 0$.
Therefore, by continuity of~$\metric^\HH$,
$$r^{-2}\,\metric^r_{\mathbf x} = \metric^\HH_{c_r^{-1}({\mathbf x})} \,\xrightarrow[r\to 0]{}\, \metric^\HH_{\mathbf x_0} = \metric^{\EE}_{\mathbf x_0} = \metric^{\EE}_{\mathbf x}$$
uniformly for ${\mathbf x}\in C$ (where we use again the trivialization of $T\EE^{n,n-1}$).
\end{proof}

We also need a statement about convergence of the vector fields used to calculate the translation length functions in $\HH^{n,n-1}$ and $\RR^{n,n-1}$.
\begin{lemma}\label{lem:vector-fields-converge}
Let $\axisH_r$ be a continuous path of oriented Riemannian lines in $\HH^{n,n-1}$ so that $\axisH_0 \ni \mathbf x_0$. Let $f_{\axisH_r}$ denote the vector field~\eqref{eqn:f} defined on $\mathcal U(\axisH_r)$.  
\begin{enumerate}
\item \label{item:all-of-Rd} The open sets $c_r \mathcal U(\axisH_r)$ converge to $\EE^{n,n-1}$ in the sense that for any compact subset $C \subset \EE^{n,n-1}$, there exists $r_0 > 0$, so that $c_r \mathcal U(\axisH_r) \supset C$ for all $ 0 < r < r_0$.
\item \label{item:vector-fields}  On any compact subset $C \subset \EE^{n,n-1}$, the vector fields $r (c_r)_* f_{\axisH_r}$ converge uniformly to the parallel unit vector field $f_0$ on $\EE^{n,n-1}$ which agrees with the positive unit vector in the direction of $\axisH_0$ in the tangent space $T_{\mathbf x_0} \PP(\RR^{2n}) = T_{\mathbf x_0} \HH^{n,n-1} = T_{\mathbf x_0} \EE^{n,n-1}$.
\end{enumerate}
\end{lemma}

\begin{proof}
For~\eqref{item:all-of-Rd} simply observe that a small open neighborhood $U$ of the basepoint $\mathbf x_0$ is contained in $\mathcal U(\axisH_0)$ and hence in all $\mathcal U(\axisH_r)$ for $r$ sufficiently small. The open sets $c_r U$, which in the affine chart $\EE^{n,n-1}$ are just dilated copies of $U$, eventually contain any compact subset of the affine chart $\EE^{n,n-1}$.

We now prove~\eqref{item:vector-fields}. 
Consider $\mathbf x \in C$ and $v \in T_{\mathbf x} \EE^{n,n-1}$ and suppose $r > 0$ is sufficiently small so that $c_r \mathcal U(\axisH_r) \supset C$.
As in the proof of Lemma~\ref{lem:metrics}, we again use the trivialization of the tangent bundle $T\EE^{n,n-1}$ to the affine chart $\EE^{n,n-1}$, and denote again by $v$ the constant vector field which agrees with the given $v \in T_{\mathbf x} \EE^{n,n-1}$. Then $(c_r^{-1})_*v = r v \in T_{c_r^{-1}({\mathbf x})}\EE^{n,n-1}$.
Next, observe that
\begin{align*}
r^{-2}\metric^r_{\mathbf x}(v, r(c_r)_*f_{\axisH_r}) &= r^{-2}\metric^\HH_{c_r^{-1}({\mathbf x})}((c_r)^{-1}_* v, rf_{\axisH_r})\\
&= r^{-2}\metric^\HH_{c_r^{-1}({\mathbf x})}(rv, rf_{\axisH_r})\\
&= \metric^\HH_{c_r^{-1}({\mathbf x})}(v, f_{\axisH_r})\\
&\xrightarrow[]{r\to0} \metric^\HH_{{\mathbf x}_0}(v, f_{\axisH_0})\\
&= \metric^\EE_{{\mathbf x}_0}(v, f_0) = \metric^\EE_{{\mathbf x}}(v, f_0).
\end{align*}
But on the other hand, the vector field $r(c_r)_*f_{\axisH_r}$ is bounded on $C$, independent of $r$, hence by Lemma~\ref{lem:metrics}, the quantity $r^{-2} \metric^r_{\mathbf x}(v, r(c_r)_*f_{\axisH_r})$ differs from $\metric^\EE_{\mathbf x}(v, r(c_r)_*f_{\axisH_r})$ by a uniform constant tending to zero with~$r$. We conclude that the vector field $r(c_r)_*f_{\axisH_r}$ converges to $f_0$ uniformly on $C$.
\end{proof}

The final ingredient needed for Lemma~\ref{lem:key} is a statement about convergence of sections of the bundles associated to the convergent path $\varrho_r^{c_r} \to (\rho, u)$. In the context of Lemma~\ref{lem:key}, define for each $r \geq 0$, the flat projective space bundle 
\begin{align*}
\mathsf{P}_r &= \Gamma \backslash T^1 \widetilde S \times \PP(\RR^{2n})
\end{align*} 
where for $r > 0$, the action of $\Gamma$ on $\PP(\RR^{2n})$ is by $\varrho_r^{c_r}$ and for $r = 0$, the action of $\Gamma$ on $\PP(\RR^{2n})$ is by $(\rho,u)$.  For each $r > 0$, the map $c_r: \PP(\RR^{2n}) \to \PP(\RR^{2n})$ induces a fiberwise embedding $c_r: \mathsf{H}_{\varrho_r} \to \mathsf{P}_r$. 
Further, the fiberwise action of $\PSO(n,n)$ on $\mathsf{H}_{\varrho_r}$ is taken by $c_r$ to the fiberwise action of $c_r \PSO(n,n) c_r^{-1}$ on $\mathsf{P}_r$.
For $r = 0$, the embedding $\EE^{n,n-1} \hookrightarrow \PP(\RR^{2n})$ induces a fiberwise embedding $\mathsf{E}_{(\rho,u)} \to \mathsf{P}_0$ which is invariant under the fiberwise action of $(\rho,u)$. Further, the fiberwise action of $c_r\PSO(n,n)c_r^{-1}$ on $\mathsf{P}_r$ converges to the action of $\SO(n,n-1) \ltimes \RR^{2n-1}$ on $\mathsf{P}_0$.

Let $s: T^1 S \to \mathsf{E}_{(\rho,u)}$ be any differentiable section of the $\EE^{n,n-1}$ bundle associated to $(\rho,u)$.
 Using the embedding $\mathsf{E}_{(\rho,u)} \to \mathsf{P}_0$, we regard $s$ as a section of $\mathsf{P}_0$.
The path $\mathsf{P}_r$ is a continuous family of flat bundles. The underlying bundles are (smoothly) isomorphic to a fixed projective space bundle and the path $\mathsf{P}_r$ may be thought of as a continuously varying family of flat connections on that fixed bundle. Hence the section $s$ determines a family of sections $s_r: T^1S \to \mathsf{P_r}$ which lift to a family of maps
\[ \widetilde s_r: T^1 \widetilde S \to \PP(\RR^{2n})\]
which vary continuously in the compact open topology and satisfy that 
\begin{itemize}
\item $\widetilde s_0(T^1 \widetilde S) \subset \EE^{n,n-1}$, and $s_0$ is $(\rho,u)$-equivariant.
\item $\widetilde s_r$ is $\varrho_r^{c_r}$-equivariant.
\end{itemize}

\begin{lemma}\label{lem:sections}
Let $(\rho,u)$ and $\varrho_r$ be as in Lemma~\ref{lem:key}. Let $s_r: T^1S \to \mathsf{P}_r$ be a continuously varying family of sections as above. Then for $r > 0$ sufficiently small, we have: 
\begin{enumerate}
\item $s_r(T^1 S) \subset c_r\mathsf{H}_{\varrho_r}$, or in other words $\widetilde s_r(T^1 \widetilde S) \subset c_r \HH^{n,n-1}$, and
\item for any $(y,z,x) \in T^1 \widetilde S$, $\widetilde{s}_r(y,z,x) \subset c_r \mathcal U(\axisH_{\varrho_r}(y,x))$.
\end{enumerate}
\end{lemma}
\begin{proof}
Note that (1) will follow from (2). We prove (2).
Let $\mathscr F \subset T^1 \widetilde S$ be a compact fundamental domain. 
For all sufficiently small $r > 0$, $\widetilde s_r(\mathscr F)$ is contained in a uniform neighborhood $U$ of the compact subset $\widetilde s_0(\mathscr F)$ in  $\EE^{n,n-1}$.
For a fixed $(y,z,x) \in \mathscr F$, the Riemannian line $\axisH_{\varrho_0}(y,x)$ contains the basepoint ${\mathbf x}_0 \in \HH^{n,n-1}$, so by Lemma~\ref{lem:vector-fields-converge}.\eqref{item:all-of-Rd}, $c_r \mathcal U(\axisH_{\varrho_r}(y,x))$ contains $U$ for all $r > 0$ sufficiently small. 
By compactness of $\mathscr F$, we have that $ \bigcap_{(y,z,x) \in \mathscr F} c_r \mathcal U(\axisH_{\varrho_r}(y,x)) \supset U$ for all $r > 0$ sufficiently small. Hence (2) holds for all $(y,z,x) \in \mathscr F$, and hence over all of $T^1 \widetilde S$ by equivariance.
\end{proof}

We now give the proof of Lemma~\ref{lem:key}.

\begin{proof}[Proof of Lemma~\ref{lem:key}]
Let $s_r: T^1S \to \mathsf{P}_r$ be a continuously varying family of sections and assume $r > 0$ is sufficiently small as in Lemma~\ref{lem:sections} above. We may use the section $c_r^{-1} s_r: T^1 S \to \mathsf{H}_{\varrho_r}$ to calculate the length function $\mathscr{L}_{\varrho_r}$ via Formula~\eqref{eqn:defn-L}:
\begin{align*} \mathscr L_{\varrho_r}(\mu) := \int_{\nu \in T^1 S} \metric^\HH( \nabla_\varphi (c_r^{-1}s_r), f ) d \mu 
\end{align*}
Let us calculate the integral by lifting everything to the product bundle $T^1 \widetilde S \times \HH^{n,n-1}$. Let $\mathscr F \subset T^1 \widetilde S$ be a fundamental domain for the action of $\Gamma = \pi_1 S$ and let $\widetilde \mu$ denote the pullback of $\mu$ to $T^1 \widetilde S$. Then 
\begin{align} r^{-1}\mathscr L_{\varrho_r}(\mu) &= r^{-1}\int_{\widetilde \nu \in \mathscr F} \metric^\HH\left( (\nabla_\varphi (c_r^{-1}\widetilde s_r))(\widetilde \nu), f_{\axisH_{\varrho_r}(\widetilde \nu)} \right) d \widetilde \mu \nonumber\\
&= r^{-1}\int_{\widetilde \nu \in \mathscr F} \metric^\HH\left( ((c_r^{-1})_* \nabla_\varphi \widetilde s_r )(\widetilde \nu), f_{\axisH_{\varrho_r}(\widetilde \nu)} \right) d \widetilde \mu \nonumber \\
&= r^{-1}\int_{\widetilde \nu \in \mathscr F} \metric^r\left( (\nabla_\varphi \widetilde s_r )(\widetilde \nu), (c_r)_* f_{\axisH_{\varrho_r}(\widetilde \nu)} \right) d \widetilde \mu \nonumber \\
&=  \int_{\widetilde \nu \in \mathscr F} r^{-2}\metric^r\left( (\nabla_\varphi \widetilde s_r )(\widetilde \nu), r(c_r)_* f_{\axisH_{\varrho_r}(\widetilde \nu)} \right) d \widetilde \mu \label{eqn:giraffe}
\end{align}

Let us now examine the integrand of right-hand side of~\eqref{eqn:giraffe}. The oriented axes $\axisH_{\varrho_r}(\widetilde 
\nu)$ converge to $\axisH_{\varrho_0}(\widetilde \nu)$, where $\varrho_0 = \iota \circ \rho$, simply because the Anosov boundary maps associated to the path $\varrho_r$ of Anosov representations vary continuously in the $C^0$ topology~\cite[Theorem 5.13]{GW}. The convergence is uniform over $\widetilde \nu$ in the compact fundamental domain $\mathscr F$. Writing $\widetilde \nu = (y,z,x)$, the axis $\axisH_{\varrho_0}(\widetilde \nu) = \xi^{n-1}(y)^\perp \cap \xi^{n-1}(x)^\perp$ contains the basepoint ${\mathbf x}_0 = [0:\ldots:0:1]$ and the positive unit tangent direction of $\axisH_{\varrho_0}$ at $\mathbf x_0$ is precisely the neutral vector $f_0(y,x) =f_0(\widetilde \nu) \in \RR^{2n-1} = T_{\mathbf x_0} \HH^{n,n-1}$ associated to $(y,x)$ for the representation $\rho$. Hence, by Lemmas~\ref{lem:metrics} and~\ref{lem:vector-fields-converge},
\begin{align*}
\lim_{r \to 0} r^{-2}\metric^r\left( (\nabla_\varphi \widetilde s_r )(\widetilde\nu), r(c_r)_* f_{\axisH_{\varrho_r}(\widetilde\nu)} \right) &=\metric^\EE( (\nabla_\varphi \widetilde s_0)(\widetilde \nu), f_0(\widetilde \nu) )
\end{align*}
and the convergence is uniform over $\mathscr F$. 
Hence
\begin{align*}
\lim_{r \to 0} r^{-1}\mathscr L_{\varrho_r}(\mu) &= \int_{\widetilde p \in \mathscr F}\metric^\EE( (\nabla_\varphi \widetilde s_0)(\widetilde\nu), f_0(\widetilde\nu) ) d \widetilde \mu\\
&= \int_{T^1 S} \metric^\EE( \nabla_{\varphi} s_0, f_0 ) d \mu\\ &= \alpha_{(\rho,u)}(\mu)
\end{align*}
and the convergence is uniform for $\mu$ varying in a compact subset of $\mathcal C(S)$.
\end{proof}

\begin{remark}
In the context of Lemma~\ref{lem:key}, the slow axes for $\varrho_r^{c_r}$ converge to the translation axis for $(\rho,u)$: for each $(y,x) \in (\partial \Gamma)^2$, $c_r \axisH_{\varrho_r}(y,x) \to \axisR(y,x)$. This follows easily since the axes are constructed directly from Anosov boundary maps. Similarly, the associated projection maps converge: $c_r \Pi_{\axisH_{\varrho_r}(y,x)} \to \Pi_{\axisR(y,x)}$. However, we did not use this convergence in the proof of Lemma~\ref{lem:key}. We used only the statement that the vector field $f_{\axisH_{\varrho_r}}(y,x)$ converges as $r\to 0$ to the neutral vector field $f_0(y,x)$. This is a slightly weaker statement because while the vector field $f_{\axisH}$ determines the line $\axisH$ in $\HH^{n,n-1}$, a parallel vector field $f_0$ does not determine one line, but only a family of parallel lines.
 \end{remark}

We now have the ingredients needed to prove Theorem~\ref{thm:main-flat} and finally Theorem~\ref{thm:main-affine}.

\subsection{Proof of Theorems~\ref{thm:main-flat} and~\ref{thm:main-affine}}

Let $\iota_{n,n}: \SO(n,n-1) \to \PSO(n,n)$ and $\iota_{2n}: \PSO(n,n) \to \PSL(2n,\RR)$ be the inclusions in the examples in Section~\ref{sec:Anosovinclusion}. Theorem~\ref{thm:main-flat} follows from the more general statement:
\begin{theorem}\label{thm:general}
Let $\rho: \Gamma \to \SO(n,n-1) = G''$ be Anosov with respect to the stabilizer $P_{n-1}''$ of an isotropic $(n-1)$-plane. Let $u: \Gamma \to \RR^{2n-1}$ be a $\rho$-cocycle so that the affine action $(\rho,u)$ on $\RR^{n,n-1}$ is properly discontinuous. Let $\varrho_r: \Gamma \to \PSO(n,n)$ be any path so that $\varrho_0 = \iota_{n,n} \circ \rho$ and $\varrho_r^{c_r}$ converges to $(\rho,u)$ as in Lemma~\ref{lem:smooth-point}. Then for all $r > 0$ sufficiently small, $\iota_{2n} \circ \varrho_r: \Gamma \to \PSL(2n,\RR)$ is Anosov with respect to the stabilizer $P_n$ of an $n$-plane in~$\RR^{2n}$.
\end{theorem}

\begin{proof}
By Theorem~\ref{thm:GLM-GT}, the Margulis invariant functional $\alpha_{(\rho,u)}: \mathcal C(S) \to \RR$ is well-defined and satisfies that $\alpha_{(\rho,u)}(\mu) \neq 0$ for any current $\mu \in \mathcal C(S)$. Since the space of currents is connected, $\alpha_{(\rho,u)}(\mu)$ has the same sign for all $\mu \in \mathcal C(S)$, and without loss in generality we assume $\alpha_{(\rho,u)}(\mu) > 0$ for all $\mu \in \mathcal C(S)$ (if not, simply conjugate by the orientation reversing isometry $-\mathrm{Id}_{2n-1}$, which does not affect proper discontinuity, but flips the sign of the Margulis invariants).
In particular, there exists $\eps > 0$ so that
\[ \alpha_{(\rho,u)}(\mathcal C_1(S)) > \eps > 0 \]
where $\mathcal C_1(S) \subset \mathcal C(S)$ denotes the currents with total mass one, a compact subset.
It then follows from Lemma~\ref{lem:key} that for all $r > 0$ sufficiently small, 
\begin{align}\label{eqn:growth}
\mathscr L_{\varrho_r}(\mathcal C_1(S)) > r \frac{\eps}{2}. 
\end{align}
Henceforth assume $r > 0$ is sufficiently small so that~\eqref{eqn:growth} holds.
Note that the stable length function $\gamma \mapsto |\gamma|_\infty$ is bi-Lipschitz to the length function $\gamma \mapsto \ell([\gamma])$ for the fixed hyperbolic metric on $S$;  denote the bi-Lipschitz constant by $M > 1$.  Equation~\ref{eqn:growth} and Proposition \ref{prop:Lperiods} then imply that for every $\gamma \in \Gamma \setminus \{1 \}$, the Lyapunov projection $\lambda'(\varrho_r(\gamma))$ satisfies that $\lambda'_n(\varrho_r(\gamma)) \geq Mr\frac{\eps}{2} |\gamma|_\infty$. In particular, $\lambda'_n(\varrho_r(\gamma)) > 0$.  
Letting $\gamma^+ = \lim_{m \to \infty} \gamma^m \in \partial \Gamma$, it then follows that $\xi^{(n)}_+(\gamma^+)$ is the attracting fixed point for the action of $\varrho(\gamma)$ on the full Grassmannian $\Gr_n(\RR^{2n})$ of $n$-planes in $\RR^{2n}$ (recall that $\xi^{(n)}_+(\gamma^+)$ is the positive isotropic $n$-plane in $\RR^{n,n}$ containing $\xi^{(n-1)}(\gamma^+)$). Hence, composing with the inclusion $\Gr_n^+(\RR^{n,n}) \hookrightarrow \Gr_n(\RR^{2n})$, the map $\xi^{(n)}_+$ determines a continuous embedding $\partial \Gamma \to \Gr_n(\RR^{2n})$ which is equivariant and dynamics preserving for the representation $\iota_{2n} \circ \varrho: \Gamma \to \PSL(2n, \RR)$.
Hence, the implication (3')~$\implies$~(1) in Theorem~\ref{thm: GGKW} shows that $\iota_{2n} \circ \varrho$ is $P_n$-Anosov.
\end{proof}

\begin{proof}[Proof of Theorem~\ref{thm:main-flat}]
Suppose $(\rho, u): \pi_1 S_g \to \mathrm{Isom^+}(\RR^{n,n-1}) = \SO(n,n-1) \ltimes \RR^{2n-1}$ is an action by isometries of $\RR^{n,n-1}$ with linear part $\rho$ a Hitchin representation in $\SO(n,n-1)$. Suppose for contradiction that the action is properly discontinuous.
Then, since $\rho$ is irreducible, Lemma~\ref{lem:smooth-point} gives the existence of a path $\varrho_r: \Gamma \to \PSO(n,n)$ such that $\varrho_0 = \iota_{n,n} \circ \rho$ and $\varrho_r^{c_r}$ converges to $(\rho,u)$.
Since $\rho$ is $P_{n-1}''$-Anosov, Theorem~\ref{thm:general} implies that  for $r > 0$ sufficiently small, $\iota_{2n} \circ \varrho_{r}: \Gamma \to \PSL(2n,\RR)$ is $P_n$-Anosov. However, $\varrho_r$ is a $\PSO(n,n)$ Hitchin representation for any $r$ and therefore cannot be $P_n$-Anosov by Theorem~\ref{thm:main2}.
\end{proof}

\begin{proof}[Proof of Theorem~\ref{thm:main-affine}]
Suppose for contradiction that $(\rho, u): \Gamma \to \Aff(\RR^d) = \GL(d,\RR) \ltimes \RR^d$ is a proper affine action with linear part~$\rho$ a lift of a representation $\sigma: \Gamma \to \PSL(d,\RR)$ in the $\PSL(d,\RR)$ Hitchin component. We show that, up to conjugation, $\rho(\Gamma) < \SO(n,n-1)$. 

First, let $\rho'$ be another lift of $\sigma$ which takes values in $\SL(d,\RR)$. Then $\rho$ and $\rho'$ differ by a scalar: $\rho(\gamma) = \lambda(\gamma)\rho'(\gamma)$, where $\lambda: \Gamma \to \RR^*$ is a homomorphism.
Guichard~\cite{Guichard2} has announced work showing that the Zariski closure $\overline{\rho'(\Gamma)}^Z < \SL(d,\RR)$ must contain the principle $\SL(2,\RR)$, i.e. the image of the irreducible representation $\tau_d: \SL(2,\RR)\to \SL(d,\RR)$. The following is a list of algebraic subgroups with that property:
\begin{enumerate}
\item\label{item:all} all of $\SL(d,\RR)$
\item\label{item:PSL2} the image of the irreducible representation $\tau_d: \SL(2,\RR)\to \SL(d,\RR)$.
\item\label{item:PSO} the orthogonal group $\SO(n,n-1)$ if $d = 2n-1$ is odd.
\item\label{item:PSp} the symplectic group $\Sp(2n,\RR)$ if $d = 2n$ is even.
\item\label{item:G2} the seven dimensional representation of $G_2$ if $d = 7$, which is contained in $\SO(4,3)$.
\end{enumerate}

Observe first that $\rho$ and $\rho'$ must agree on the commutator subgroup $[\Gamma, \Gamma]$. 
Hence
\begin{align*}
\overline{\rho(\Gamma)}^Z \supset \overline{\rho([\Gamma, \Gamma])}^Z &= \overline{\rho'([\Gamma, \Gamma])}^Z\\ &=  \overline{[\rho'(\Gamma), \rho'(\Gamma)])}^Z\\ &= [ \overline{\rho'(\Gamma)}^Z, \overline{\rho'(\Gamma)}^Z] =  \overline{\rho'(\Gamma)}^Z,
\end{align*}
where the last equality is easily checked for each of the groups listed above. 
Hence, in particular, $\overline{\rho(\Gamma)}^Z$ contains $\tau_d(\SL(2,\RR))$. 

On the other hand, it is a basic linear algebra fact that an affine transformation $g = (A_g, u_g)$ fixes a point unless $A_g$ has one as an eigenvalue, see e.g.~\cite{KS75}. Hence for all $\gamma \in \Gamma \setminus \{1\}$, $\rho(\gamma)$ has one as an eigenvalue and this property passes to the Zariski closure $\overline{\rho(\Gamma)}^Z$.  In the case that $d = 2n$ is even, $\tau_d(\SL(2,\RR))$ contains, for example 
\begin{align*}
\tau_{2n} \diag(e^t, e^{-t}) &= \diag(e^{(2n-1)t}, e^{(2n-3)t}, \ldots, e^t, e^{-t}, \ldots, e^{-(2n-1)t})
\end{align*}
which does not have one as an eigenvalue. 

Hence $d = 2n-1$ is odd.
By the above, $\overline{\rho(\Gamma)}^Z \supset \overline{\rho'(\Gamma)}^Z$. It is also true that $\overline{\rho(\Gamma)}^Z \subset \Pi^{-1}(\overline{\rho'(\Gamma)}^Z)$, where $\Pi: \GL(2n-1,\RR) \to \SL(2n-1,\RR)$ is the natural projection, since the algebraic equations in $\SL(2n-1,\RR)$ defining $\overline{\rho'(\Gamma)}^Z$ pull back to algebraic equations in $\GL(2n-1,\RR)$. Further, if $g_1, g_2 \in \overline{\rho(\Gamma)}^Z$ are such that $\Pi(g_1) = \Pi(g_2)$, then $g_1g_2^{-1} \in \overline{\rho(\Gamma)}^Z$ is a multiple of the identity, hence equal to the identity, by the eigenvalue one property. It follows that the projection $\Pi$ maps $\overline{\rho(\Gamma)}^Z$ to $\overline{\rho'(\Gamma)}^Z$ one to one, and hence that $\overline{\rho(\Gamma)}^Z = \overline{\rho'(\Gamma)}^Z$.
Therefore $\overline{\rho(\Gamma)}^Z$ is (conjugate to) one of the items on the above list, namely $\SO(n,n-1)$ (case~\ref{item:PSO}) or $G_2$ (case~\ref{item:G2}) or $\tau_d(\SL(2,\RR)$ (case~\ref{item:PSL2}). In all cases $\overline{\rho(\Gamma)}^Z < \SO(n,n-1)$. Hence $\rho(\Gamma) < \SO(n,n-1)$. This contradicts Theorem~\ref{thm:main-flat} and concludes the proof of Theorem~\ref{thm:main-affine}.
\end{proof}

\appendix
\section{Positivity in $\PSO(n,n)$}\label{app:matrices}
While an explicit description of positive triples of flags for the Lie group $G = \PSL(d,\RR)$ is given in almost every introduction to the study of positivity (see e.g.~\cite{GW16}), an explicit description of the positive triples of flags for $G' = \PSO(n,n)$ seems to be absent from the literature. 
The purpose of this appendix is to give one such description, by induction on $n$. Let $B'^\pm$ be the Borel subgroups of $G'$ described in Example~\ref{eg:SO(n,n)}. In Example~\ref{eg:SO(n,n)}, we also described the unipotent radicals $U'^\pm\subset B'^\pm$, the corresponding positive Weyl chamber $\mathfrak{a}'^+\subset\mathfrak{a}'$, and the simple roots $\Delta'=\{\alpha'_1,\dots,\alpha'_{n}\}$. Let $(B'^+,B'^-,\{x_{\alpha'_i}^+\}_{i=1}^n,\{x_{\alpha'_i}^-\}_{i=1}^n)$ be the pinning described in Example~\ref{eg:pinning}. We give an explicit inductive formula describing $U'^+_{>0}$, the set of positive elements in $U'^+$ corresponding to the chosen pinning.

We will now inductively define, for any positive integer $n$ and any $k=1,\dots,n-1$, a family of $(2n)\times (2n)$ matrices $M_{n,k}$ whose entries depend on variables $a_1,\dots,a_k,b_1,\dots,b_k$. When $n=2$ and $k=1$, define 
\[M_{2,1}=M_{2,1}(a_1,b_1):=\left(\begin{array}{cccc}
1&a_1&b_1&-a_1b_1\\
0&1&0&-b_1\\
0&0&1&a_1\\
0&0&0&1
\end{array}\right).\]
Now suppose that we have defined $M_{n,k}=M_{n,k}(a_1,\dots,a_k,b_1,\dots,b_k)$. Then define the $(2n+2)\times(2n+2)$ square matrices
\[M_{n+1,k}(a_1,\dots,a_k,b_1,\dots,b_k):=\left(\begin{array}{ccc}
1&0&0\\
0&M_{n,k}&0\\
0&0&1
\end{array}\right),\]
and 
\[M_{n+1,n}(a_1,\dots,a_n,b_1,\dots,b_n):=\left(\begin{array}{ccc}
1&v_1&v_3\\
0&M_{n,n-1}&v_2\\
0&0&1
\end{array}\right),\]
where $v_1$ is the $1\times (2n)$ matrix, $v_2$ is the $(2n)\times 1$ matrix, and $v_3$ is the $1\times 1$ matrix given by
\begin{eqnarray*}
v_1&:=&\big(a_n+b_n\,,\,a_n\cdot(M_{n,n-1})_{1,2}\,,\,a_n\cdot(M_{n,n-1})_{1,3}\,,\,\dots\,,\,a_n\cdot(M_{n,n-1})_{1,2n}\big),\\
v_2&:=&\left(
\begin{array}{c}
-b_n\cdot(M_{n,n-1})_{1,2n}\\
-b_n\cdot(M_{n,n-1})_{2,2n}\\
\vdots\\
-b_n\cdot(M_{n,n-1})_{2n-1,2n}\\
-(a_n+b_n)
\end{array}\right),\\
v_n&:=&-a_n\cdot b_n\cdot (M_{n,n-1})_{1,2n}=(-1)^{n-1}\prod_{i=1}^{n-1}a_i\cdot b_i.
\end{eqnarray*}
In the above formulas, $(M_{n,n-1})_{i,j}$ denotes the $(i,j)$-entry of $M_{n,n-1}$.

Let $s_i:=s_{\alpha'_i}$ be the generators of the Weyl group $W(\mathfrak{a}')$ of $G'$ described in Section \ref{sec:Positivity}, and let $x_i^+:=x_{\alpha'_i}^+$. Recall from Example \ref{eg:Mkl} that $\mu_1\cdot\mu_2\cdot\dots\cdot \mu_{n-1}$ is a reduced word expression for the longest word element in $W(\mathfrak{a}')$, where $\mu_1:=s_{n-1}\cdot s_n$, and $\mu_k:=s_{n-k}\cdot\mu_{k-1}\cdot s_{n-k}$ for all $k=2,\dots,n-1$. Using the description of $x^+_i$ in Example \ref{eg:pinning}, one can check via a straightforward induction argument that for all $k=1,\dots,n-1$, the matrix $M_{n,k}$ defined above satisfies
\[M_{n,k}=\left(\prod_{i=n-k}^{n-2}x_i^+(a_{n-i})\right)\cdot\big(x_{n-1}^+(a_1)\cdot x_n^+(b_1)\big)\cdot \left(\prod_{i=n-2}^{n-k}x_i^+(b_{n-i})\right).\]
It follows immediately that
 \[U^+_{>0}(G)=\left\{M_n(\{a_{k,l}\},\{b_{k,l}\}):\begin{array}{l}
k=1,\dots,n-1;\,\,l=1,\dots,k;\\ 
a_{k,l},b_{k,l}>0\text{ for all }k,l\end{array}\right\}
\]
where $M_n=M_n(\{a_{k,l}\},\{b_{k,l}\})$ is the $(2n)\times(2n)$ matrix given by
\[M_n:=\prod_{k=1}^{n-1}M_{n,k}(a_{k,1},\dots,a_{k,k},b_{k,1},\dots,b_{k,k}).\]
\section{The positive curve for $\PSO(n,n)$-Hitchin representations}\label{app:negative triples}

Here we prove the following proposition, which improves upon Theorem~\ref{thm: Fock-Goncharov} in the specific case of $G'$-Hitchin representations, where as usual $G'$ denotes $\PSO(n,n)$.
\begin{proposition}
Suppose $\varrho: \Gamma \to G'$ is a $G'$-Hitchin representation. Then the associated positive curve $\xi: \partial \Gamma \to G'/B'$ takes any triple of pairwise distinct points (independent of the cyclic ordering) to a positive triple of flags.
\end{proposition}

\begin{proof}
Consider $y,z,x\in\partial\Gamma$ pairwise distinct. We wish to show that $(\xi(y), \xi(z), \xi(x))$ is a positive triple. If $(y,z,x)$ is positive for the cyclic ordering $\partial\Gamma$, then this is given by Theorem~\ref{thm: Fock-Goncharov}. We assume $(y,z,x)$ is not a positive triple.

Recall that we chose an oriented hyperbolic structure on $S$ to identify $\partial\Gamma$ with $\partial\HH^2$, the visual boundary of the upper half plane. Let $j:\Gamma\to\PSL(2,\RR)$ be the Fuchsian representation corresponding to this choice, let $\varrho_0:=\tau_{G'}\circ j:\Gamma\to G'$, and let $\varrho_t$ be a continuous path so that $\varrho_1=\varrho$. For each $t$, let $\xi_t:\partial\Gamma\to\mathcal F_{B'}$ be the $\varrho_t$-equivariant positive boundary map. 
The positive triples of flags make up a union of connected components of the space of pairwise transverse triples of flags (\cite[Proposition 8.14]{Lusztig}). Thus if $(\xi_0(y),\xi_0(z),\xi_0(x))$ is a positive triple, then so is $(\xi_t(y),\xi_t(z),\xi_t(x))$ for all $t$, hence it is sufficient to prove that $(\xi_0(y),\xi_0(z),\xi_0(x))$ is a positive triple.

Observe that $\tau_{G'}$ extends to a homomorphism $\tau_{G'}: \PGL(2,\RR) \to \PO(n,n)$. Let $g \in \PGL(2,\RR) \setminus \PSL(2,\RR)$ be any orientation reversing element. One easily computes that 
\begin{itemize}
\item $\tau_{G'}(g) \in \PSO(n,n)$ if $n$ is even, or
\item $\tau_{G'}(g) \in \PO(n,n) \setminus \PSO(n,n)$ if $n$ is odd, hence $\tau_{G'}(g) = hm$ where $h \in \PSO(n,n)$ and $m \in \PO(n,n) \setminus \PSO(n,n)$ is the element which pointwise fixes the copy of $\RR^{n,n-1}$ invariant under $\tau_{G'}(\PSL(2,\RR)) \subset \SO(n,n-1)$ and flips the sign of the orthogonal $\RR^{0,1}$.	
\end{itemize}
The triple $(gy,gz,gx)$ has positive orientation, hence the triple of flags
$$(\xi_0(gy), \xi_0(gz), \xi_0(gx)) = \tau_{G'}(g)(\xi_0(y), \xi_0(z), \xi_0(x))$$
is positive by Theorem~\ref{thm: Fock-Goncharov}. If $n$ is odd, it follows that $(\xi_0(y), \xi_0(z), \xi_0(x))$ is also positive, since $\tau_{G'}(g) \in \PSO(n,n)$. 
Otherwise, if $n$ is even, we have that $\tau_{G'}(g) = hm$ as above. However, the $\tau_{G'}$-equivariant embedding $\RP^1 \to G'/B'$, which defines $\xi_0$, is fixed by $m$, hence again have that $(\xi_0(y), \xi_0(z), \xi_0(x))$ differs from $(\xi_0(gy), \xi_0(gz), \xi_0(gx))$ by an element of $\PSO(n,n)$ and hence is also positive.
\end{proof}

\bibliographystyle{amsalpha}
\bibliography{ref}
\end{document}